\documentclass[11pt]{article}
\usepackage{indentfirst}
\usepackage{amssymb,stmaryrd, latexsym,amsmath,amsbsy,amsthm,amsxtra,amsgen,graphicx,makeidx,dsfont,longtable,bbm,tikz,hyperref}
\usetikzlibrary {patterns,patterns.meta,decorations.pathmorphing, decorations.pathreplacing, decorations.shapes} 
\newtheorem{theorem}{Theorem}[section]
\newtheorem{remark}[theorem]{Remark}
\newtheorem{corollary}[theorem]{Corollary}
\newtheorem{lemma}[theorem]{Lemma}
\newtheorem{proposition}[theorem]{Proposition}

\def\eps{\varepsilon}
\def\Var{\textup{Var}}
\def\Cov{\textup{Cov}}
\def\Q{\mathbb{Q}}
\def\Z{\mathbb{Z}}
\def\N{\mathbb{N}}
\def\R{\mathbb{R}}

\def\E{\mathbb{E}}
\def\P{\mathbb{P}}
\def\V{\mathcal{V}}
\def\L{\mathcal{L}}
\def\H{\mathcal{H}}
\def\W{\mathcal{W}}
\DeclareMathOperator*{\argmax}{arg\,max}

\begin{document}
\title{Branch lengths for geodesics in the directed landscape and mutation patterns in growing spatially structured populations}
\author{Shirshendu Ganguly\footnote{Department of Statistics, University of California, Berkeley, sganguly@berkeley.edu, Partially supported by NSF career grant-1945172.
}, Jason Schweinsberg\footnote{Department of Mathematics, University of California San Diego, jschweinsberg@ucsd.edu}, Yubo Shuai\footnote{Department of Mathematics, University of California San Diego, yushuai@ucsd.edu}}

\maketitle
\vspace{-.3in}
\begin{abstract}
Consider a population that is expanding in two-dimensional space.  Suppose we collect data from a sample of individuals taken at random either from the entire population, or from near the outer boundary of the population.  A quantity of interest in population genetics is the site frequency spectrum, which is the number of mutations that appear on $k$ of the $n$ sampled individuals, for $k = 1, \dots, n-1$.  As long as the mutation rate is constant, this number will be roughly proportional to the total length of all branches in the genealogical tree that are on the ancestral line of $k$ sampled individuals.  
While the rigorous literature has primarily focused on models without any spatial structure, in many natural settings, such as tumors or bacteria colonies, growth is dictated by spatial constraints. A large number of such two dimensional growth models are expected to fall in the KPZ universality class exhibiting similar features as the Kardar-Parisi-Zhang equation.
Although we are not aware of previous rigorous work related to the site frequency spectrum for such models, nonrigorous predictions are available, for instance, in \cite{fgkah16}.  

In this article we adopt the perspective that for population models in the KPZ universality class, the genealogical tree can be approximated by the tree formed by the infinite upward geodesics in the directed landscape, a universal scaling limit constructed in \cite{dov22}, starting from $n$ randomly chosen points.  Relying on geodesic coalescence, we prove new asymptotic results for the lengths of the portions of these geodesics that are ancestral to $k$ of the $n$ sampled points and consequently obtain exponents driving the site frequency spectrum as predicted in \cite{fgkah16}.  An important ingredient in the proof is a new tight estimate of the probability that three infinite upward geodesics stay disjoint up to time $t$, i.e., a sharp quantitative version of the well studied N3G problem, which is of independent interest.
\end{abstract}

\setcounter{tocdepth}{2}
\tableofcontents
\section{Introduction}

One of the central goals in population genetics is to use information obtained from the DNA of individuals sampled from the present-day population to draw inferences about the evolutionary history of the population.  Since the seminal work of Kingman \cite{king82} in the early 1980s, coalescent theory has been an essential tool for this purpose.  Suppose one takes a sample of $n$ individuals from a population and follows their ancestral lines backwards in time.  Then the ancestral lineages will coalesce until eventually all of the lineages are traced back to one common ancestor, leading to a genealogical tree such as the one illustrated in Figure \ref{treefigure} below.

An important observable in this context is the site frequency spectrum (defined in the upcoming section).  In this article, we are interested in the site frequency spectrum in two-dimensional population models with spatial structure.  We leverage recent mathematical advances in the understanding of growth models in the Kardar-Parisi-Zhang universality class to build a  framework which allows us to verify rigorously some of the exponents that have been numerically observed in the literature, for example in \cite{fgkah16, epf22}, pertaining to the behavior of the site frequency spectrum in bacterial colonies and other two-dimensional population models.

\subsection{The site frequency spectrum}

From the DNA sequences obtained from the $n$ sampled individuals, one can identify mutations which are inherited by some but not all individuals in the sample.  Let $M_{k,n}$ denote the number of mutations inherited by $k$ of the $n$ sampled individuals.  Then the vector $(M_{1,n}, \dots, M_{n-1,n})$ is called the site frequency spectrum.  An example is given in Figure \ref{treefigure} below.

\begin{figure}[h]\label{treefigure}
\centering
\begin{tikzpicture}[scale=0.7]
\draw [very thick] (3,0)--(3.75,1);
\draw [very thick] (4.5,0)--(3.75,1);
\draw [very thick] (7.5,0)--(8.25,1.5);
\draw [very thick] (9,0)--(8.25,1.5);
\draw [very thick] (6,0)--(4.5,2);
\draw [very thick] (3.75,1)--(4.5,2);
\draw [very thick] (10.5,0)--(9,2.5);
\draw [very thick] (8.25,1.5)--(9,2.5);
\draw [very thick] (4.5,2)--(7,4.5);
\draw [very thick] (9,2.5)--(7,4.5);
\draw [fill=black](3.45,0.6) circle (2pt);
\draw [fill=black](4.35,0.2) circle (2pt);
\draw [fill=black](5,2.5) circle (2pt);
\draw [fill=black](7.25,4.25) circle (2pt);
\draw [fill=black](8.7,2.1) circle (2pt);
\draw [fill=black](10.2,0.5) circle (2pt);
\node at (3,-0.3){1};
\node at (4.5,-0.3){2};
\node at (6,-0.3){3};
\node at (7.5,-0.3){4};
\node at (9,-0.3){5};
\node at (10.5,-0.3){6};
\end{tikzpicture}
\caption{\small{A genealogical tree of a sample of size $n = 6$.  Dots indicate the times of mutations.  Three mutations near the bottom of the tree are inherited by only one individual (1, 2, or 6).  One mutation is inherited by the individuals 4 and 5.  The two mutations closest to the top are inherited by three individuals each.  Therefore, the site frequency spectrum is $M_{1,n} = 3$, $M_{2,n} = 1$, $M_{3,n} = 2$, and $M_{4,n} = M_{5,n} = 0$.}}
\end{figure}

We say that a branch of the coalescent tree supports $k$ leaves if the ancestral lines of $k$ of the $n$ sampled individuals include that branch.  A mutation that arises on a branch that supports $k$ leaves will be inherited by $k$ of the sampled individuals.  Let $L_{k,n}$ denote the total length of all branches in the coalescent tree that support $k$ leaves.  If we assume that mutations occur at the constant rate $\theta$ per unit time along each lineage, then the conditional distribution of $M_{k,n}$ given $L_{k,n}$ is Poisson with mean $\theta L_{k,n}$.  Therefore, to understand the distribution of the site frequency spectrum, it is enough to understand the distribution of the branch lengths $L_{k,n}$.

While the focus of this article is on models with spatial structure, let us first briefly review the literature on more tractable examples without an intrinsic geometry.   A classical population model for populations of fixed size is  the Moran model \cite{moran58}, in which there are $N$  individuals in the population at all times, each individual independently lives for an exponentially distributed time with rate $1$, and when an individual dies, an individual is chosen at random from the population to give birth to a new individual. In this model, the genealogy of a sample of $n$ individuals, after a rescaling of time, is given by Kingman's coalescent \cite{king82}, in which each pair of lineages merges at rate $1$.  For Kingman's coalescent, the distribution of the branch lengths $L_{k,n}$ is well understood. For instance, Fu \cite{f95} showed that for Kingman's coalescent,
\begin{equation}\label{kingexp}
\E[L_{k,n}] = \frac{2}{k} \qquad\mbox{for } 1 \leq k \leq n-1,
\end{equation}
and a Central Limit Theorem was obtained in \cite{dk15}.
The site frequency spectrum has also been studied, for example, in \cite{bbs07, dk19, gmss24, ksw21, sks16} for coalescent processes in which more than two lineages can merge at a time.

Also of interest are populations whose size is increasing rapidly over time, such as populations of cancer cells.  For exponentially growing populations, the coalescent tree has long external branches leading to the $n$ leaves of the tree, with most of the coalescence occurring close to the root.  Durrett \cite{d13} considered a supercritical branching process starting from one individual in which each individual gives birth at rate $\lambda$ and dies at rate $\mu$, where $\lambda > \mu$.  He showed that if we take a sample of size $n$ from the population at time $T$, then as $T \rightarrow \infty$, we have
\begin{equation}\label{durrexp}
\E[L_{k,n}] \rightarrow \frac{n}{\lambda - \mu} \cdot \frac{1}{k(k-1)} \qquad\mbox{for }2 \leq k \leq n-1.
\end{equation}
Durrett also showed that $\E[L_{1,n}] \sim n T$ as $T \rightarrow \infty$.  Further asymptotics were developed in \cite{jssc23, ss23}, and it was demonstrated in \cite{jssc23} that the formula \eqref{durrexp} provides a good fit to data from blood cancer.
See also \cite{glf21, glz23}, in which the authors investigate the site frequency spectrum for an exponentially growing population when we have genetic data from the entire population, rather than just a sample of size $n$.

\subsection{Spatial population models}

While models of fixed size populations and exponentially growing populations are applicable in a wide range of settings, they fail to capture instances when spatial constraints dictate population growth, preventing the population from growing exponentially fast.  Over the years, a wide variety of spatial population models have been proposed.  One example is the biased voter model on $\Z^d$, which was proposed as a model of tumor growth in 1972 by Williams and Bjerknes \cite{wb72} and was studied further in \cite{dfl16}.  Stepping stone models, which were introduced in \cite{k53} and were studied mathematically, for example, in \cite{cd02}, allow demes consisting of many individuals at each site.  Another well-known model is the Eden model on $\Z^d$, which was introduced in \cite{e61}.  In the Eden model, the population begins at time zero with a single individual occupying the origin.  At each time step, a new individual is born at a site that is randomly chosen from the sites that are vacant, but are adjacent to an occupied site.  The Eden model, and various extensions of it, has frequently been proposed as a model of tumor growth \cite{noble22, wbphvn15, chk19, wk20, li23, kpw25}.
One can alternatively model space as a continuum by using the spatial $\Lambda$-Fleming-Viot process, which was formulated in \cite{bev10} and was adapted to the setting of an expanding population in \cite{l23, lv24}, or the closely related model introduced by Deijfen in \cite{d03}.

We advance this discussion in the setting of first passage percolation (FPP), a close variant of the Eden model.  See \cite{adh17} for a survey of first passage percolation.
Individuals in the population are located at sites in $\mathbb{Z}^2$, and the process begins with one individual located at the origin.  For each edge $e$ of $\mathbb{Z}^2$, there is an independent random variable $\tau_e$, (say distributed as a standard exponential variable) which we call the length of the edge.  Then, for an edge $e$ that connects the vertices $u$ and $v$, once either the site $u$ or $v$ is occupied by an individual, a new individual is born at the other site adjacent to the edge after an additional time $\tau_e$, if that site is not already occupied.
For $x,y \in \Z^2$, let $d(x,y)$ denote the length of the shortest path between $x$ and $y$, which is called the geodesic path from $x$ to $y$.  Writing ${\bf 0}$ for the point $(0,0)$, the time when the vertex $x$ first becomes occupied by an individual is $d({\bf 0}, x)$.  The vertices along the geodesic path from ${\bf 0}$ to $x$ are the ancestors of the individual located at $x$.  Therefore, the genealogical tree of the individuals at $n$ sites $x_1, \dots, x_n$ in $\mathbb{Z}^2$ is the tree formed from the $n$ geodesic paths from ${\bf 0}$ to $x_i$, for $i = 1, \dots, n$.  See Figure \ref{simlineages} for a simulation of the geodesics in first passage percolation.

\begin{figure}
\centering
\includegraphics[scale=0.4, trim={4cm 6.5cm 3.2cm 6cm}, clip]{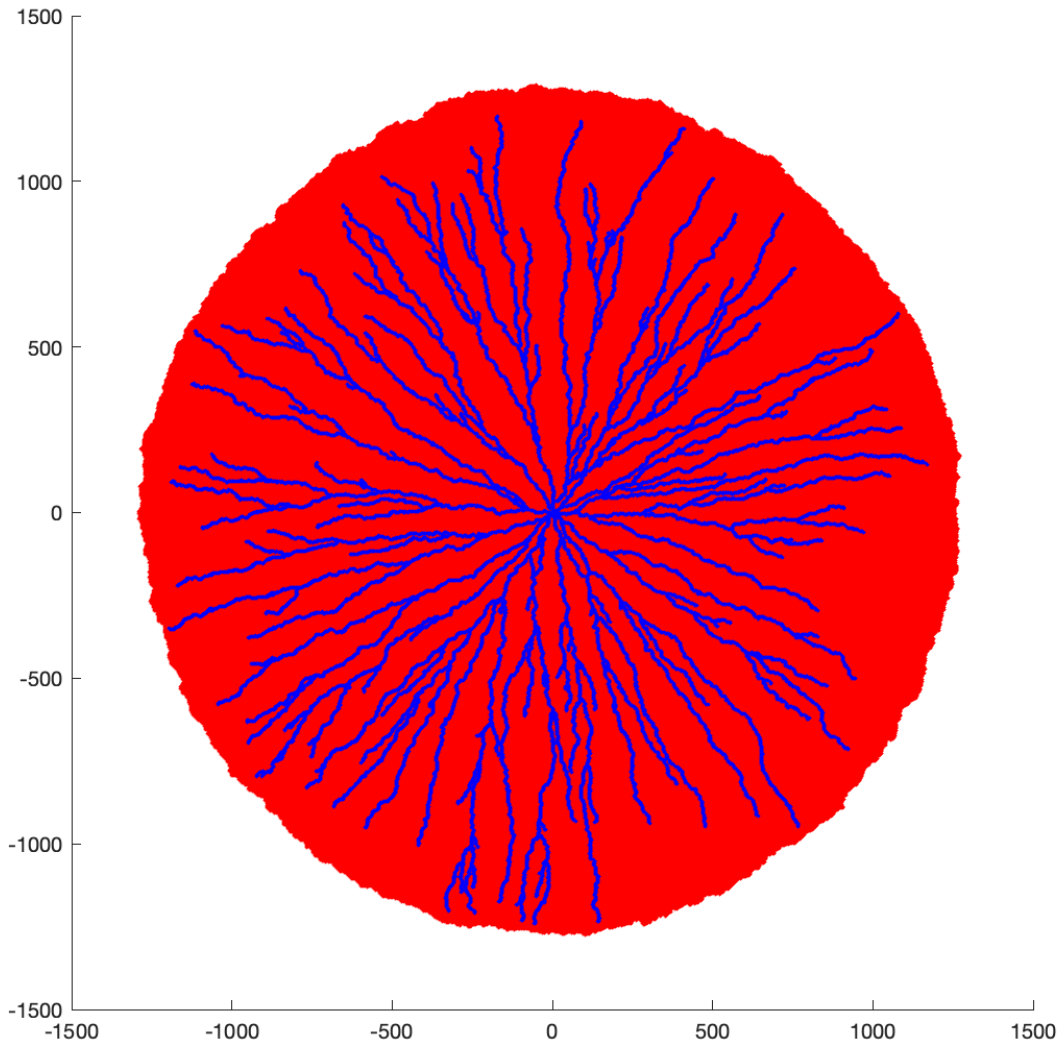} \hspace{.3in}
\includegraphics[scale=0.4, trim={4cm 6.5cm 3.2cm 6cm}, clip]{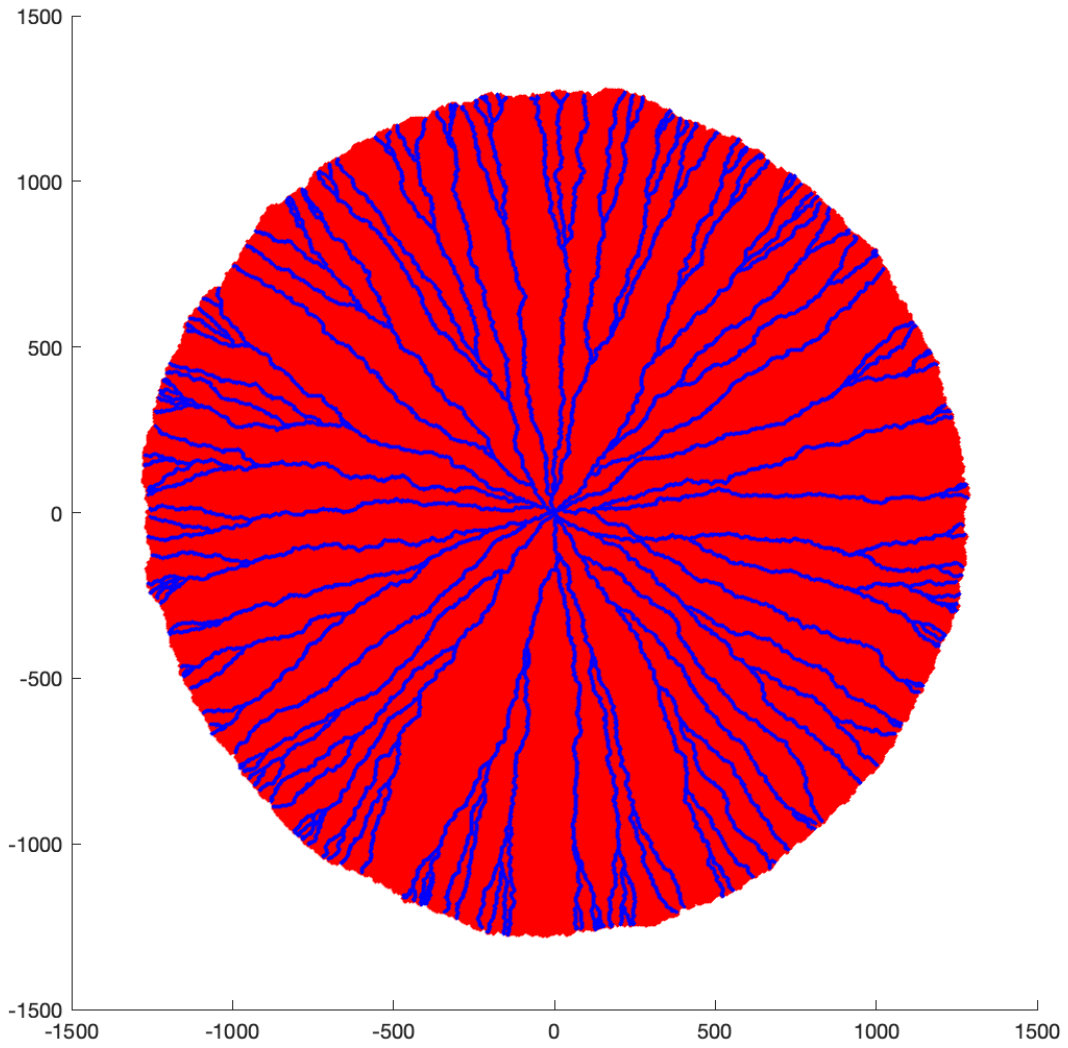}
\caption{Genealogical trees of samples of size $n = 200$ from first passage percolation.  The growth of the cluster was simulated until there were $N = 5,000,000$ occupied sites.  In the picture on the left, the blue lines show the geodesic paths from the origin to $n$ randomly chosen points from the cluster.  In the picture on the right, the $n$ points were sampled from the boundary of the cluster.}
\label{simlineages}
\end{figure}

While above we took the edge weights $\tau_e$ above to be standard exponential variables, it is believed that the model exhibits a universal fluctuation theory which does not depend on the distribution of the weights $\tau_e$ as long as it satisfies certain reasonable assumptions. Moreover, this universal behavior is expected to be exhibited by a large class of models forming the so-called KPZ universality class. The Eden model is also believed to belong to this class, and recent simulation results in \cite{il25} suggest that the spatial $\Lambda$-Fleming-Viot process belongs to this class as well.  All of these models are expected to exhibit critical fluctuation exponents in line with the behavior of a canonical non-linear stochastic partial differential equation known as the KPZ equation.  The KPZ equation was proposed by Kardar, Parisi, and Zhang \cite{kpz86} as a model of a random interface, one example of which is the outer boundary of the growth cluster in FPP.  The KPZ equation is given by
$$\frac{\partial}{\partial t} h(t,x) = \frac{1}{2} \frac{\partial^2}{\partial x^2} h(t,x) - \frac{1}{2} \bigg( \frac{\partial}{\partial x} h(t,x) \bigg)^2 + \xi(t,x),$$
where $h(t,x)$ represents the height of an interface at time $t$ and spatial location $x \in \mathbb{R}$, and $\xi$ is space-time white noise.  On account of the white-noise term, the solutions, if any, ought to be quite rough, rendering the gradient squared term ill-defined.  Nevertheless, at a heuristic level, from scaling arguments, one can deduce that when the growth of an interface is governed by the KPZ equation, $h(t,x)$ for large $t$ is typically at a height which is linear in $t$ with fluctuations on the scale of $t^{1/3}$, and exhibiting non-trivial correlation between $h(t,x)$ and $h(t,y)$  when $|x - y|$ is on the scale of $t^{2/3}$. This is often termed in the literature as the $(1:2:3)$ scaling property.

Although we are not aware of previous rigorous work related to the site frequency spectrum for populations that are growing according to a model in the KPZ universality class, some nonrigorous predictions are available.  Fusco, Gralka, Kayser, Anderson, and Hallatschek \cite{fgkah16} argued that the number of mutations that are inherited by a fraction at least $x$ of the population is approximately proportional to $x^{-2/5}$ when $x$ is small, and this expression was shown in \cite{gh19} to match some data from bacterial populations.  This leads to the prediction that in a random sample from the population, the number of mutations inherited by exactly $k$ of the sampled individuals should be proportional to $k^{-7/5}$ (essentially by taking the derivative in $x$ of the previous statement), in contrast to the $k^{-1}$ behavior for Kingman's coalescent in \eqref{kingexp} and the $k^{-2}$ behavior for exponentially growing populations in \eqref{durrexp}.  Eghdami, Paulose, and Fusco \cite{epf22} studied the case in which the sample comes entirely from the outer edge of the population.  They found in their equation (6) that the number of mutations inherited by $k$ sampled individuals should be proportional to $k^{-1/2}$, although some of their simulation results suggested a better fit to $k^{-2/3}$, a point to which we will return in section \ref{simsec} when we present simulations.  Our main results, which are Theorems \ref{Theorem: main2} and \ref{Theorem: main1} below, establish these exponents rigorously for the directed landscape.

On the experimental side, it was demonstrated in \cite{hhrn07, gh19} that the growth patterns of some bacterial populations resemble what would be predicted by the KPZ equation.  Allstadt et al. \cite{anwkc16} suggested that models in the KPZ universality class could describe the growth of invasive plant species, such as white clover.  Growth dynamics consistent with KPZ fluctuations were observed in \cite{hpgba12} for cervical cancer, and Lewinsohn et al. \cite{lbmf23} found that cancer cells near the boundary of a tumor divide faster than cells in the interior, consistent with the Eden model.

It is worth noting that the above discussion has been exclusively two dimensional, which means it is not applicable in many physical instances, such as the growth of solid tumors, which demands a three-dimensional theory. Our understanding of the latter remains significantly limited, notwithstanding some recent impressive developments \cite{csz24}. 
Further, the planar KPZ models do not account for cell death, nor do they account for the possibility that the tumor may consist of several types of cells which grow at different speeds.
Finally, some experimental results reported in \cite{bru98, bru03} indicate that the boundary of some tumors is more smooth than would be predicted by KPZ models, perhaps because of the movement of tumor cells near the boundary.  For a recent discussion of this issue, see \cite{mawk22, mawk23}.

\subsection{The directed landscape and its geodesics}

Instead of working with prelimiting models, we will base our framework on a process known as the directed landscape, which was recently constructed by Dauvergne, Ortmann and Vir\`ag \cite{dov22}.  The directed landscape is conjectured to be the universal scaling limit for various models in the KPZ universality class.  In particular, the genealogical structure in the limit for such models is believed to be described by the so-called geodesics in the directed landscape.

While the convergence of general pre-limiting models to the directed landscape remains one of the central open problems in field, the convergence has indeed been shown for a class of models possessing additional integrable structures. This includes Brownian last passage percolation \cite{dov22} and Geometric and Exponential last passage percolation \cite{dv21}. Convergence for related models such as the asymmetric simple exclusion process and the stochastic six-vertex model was shown recently in \cite{ach24} (see also \cite{dz24c}), and convergence of the KPZ equation and the log-gamma polymer to the directed landscape was shown in \cite{wu23} and \cite{z25} respectively. 

We now introduce the directed landscape, although we refrain from providing a fully rigorous definition and instead refer the reader to \cite{dov22} for the details.  Let $$\R^4_{\uparrow} = \{(x,s;y,t) \in \R^4: s < t\},$$ which can be viewed as the set of ordered pairs of spacetime points, where $x$ and $y$ are spatial coordinates and $s$ and $t$ are time coordinates.  The directed landscape is a random function $\L: \R^4_{\uparrow} \rightarrow \R$.  It is a directed metric (see \cite{dv21}, section 1.3), meaning that if $p$, $q$, and $r$ are points in $\R^2$ such that $(p;q)$ and $(q;r)$ are in $\R^4_{\uparrow}$, then $\L(p;p) = 0$ and $\L(p;r) \geq \L(p;q) + \L(q;r)$.  The directed landscape also satisfies the metric composition property (see \cite{dov22}, Definition 1.4), which means that if $x,y \in \R$ and $r < s < t$, then 
\begin{equation}\label{metric comp}
\L(x,r;y,t) = \max_{z \in \R} \big( \L(x,r;z,s) + \L(z,s;y,t) \big).
\end{equation}
The directed landscape is stationary in time and space.  It also satisfies a scaling property, which roughly states that the landscape looks the same after the horizontal axis is scaled by $\rho^2$ and the vertical axis is scaled by $\rho^3$.  These properties will be stated formally in Proposition \ref{Proposition: properties of the landscape} below.

For any $(x,s;y,t)\in\R^4_{\uparrow}$, a path from $(x,s)$ to $(y,t)$ is a continuous function $\gamma:[0,t-s]\rightarrow\R$ such that $\gamma(0)=x$ and $\gamma(t-s)=y$. We define its length to be
$$
||\gamma||_{\L}:=\inf_{k\in \N}\inf_{0=t_0<t_1<\dots<t_k=t-s}\sum_{i=1}^k\L(\gamma(t_{i-1}), s+t_{i-1}; \gamma(t_i), s+t_i).
$$
The metric composition property \eqref{metric comp} implies that $||\gamma||_{\L}\le \L(x,s; y,t)$ for any path $\gamma$, and we say $\gamma$ is a geodesic from $(x,s)$ to $(y,t)$ if equality holds (note that the path is directed).  Almost surely, a geodesic from $(x,s)$ to $(y,t)$ exists for all $ (x,s;y,t) \in \R^4_{\uparrow}$ (see \cite{dov22}, Lemma 13.2).  For each fixed $ (x,s;y,t) \in \R^4_{\uparrow}$, the geodesic from $(x,s)$ to $(y,t)$ is unique almost surely (see \cite{dov22}, Theorem 12.1), although the uniqueness will fail for a measure zero set of exceptional points (see \cite{bgh21,bgh22,gz22,d23, gm24} for a study of the associated fractal structures). Nonetheless, as noted in Lemma~13.2 of \cite{dov22}, almost surely for all $(x,s;y,t)\in\R^4_{\uparrow}$, there exist unique left-most and right-most geodesics from $(x,s)$ to $(y,t)$.

A continuous function $g: [0, \infty) \rightarrow \R$ is called an infinite geodesic starting from $(x,s)$ if $g(0) = x$ and the restriction of $g$ to $[0,u]$ is a geodesic from $(x,s)$ to $(g(u), s+u)$ for every $u > 0$.  We call $g$ an infinite upward geodesic if it is an infinite geodesic and if $\lim_{u \rightarrow \infty} u^{-1} g(u) = 0$.  As recorded in Corollary 3.20 of \cite{rv21}, almost surely there is an infinite upward geodesic starting from $(x,s)$ for all $(x,s) \in \R^2$.  For any fixed point $(x,s) \in \R^2$, this infinite upward geodesic is unique almost surely.  Again, while the uniqueness fails at an exceptional set of points, almost surely there is a unique left-most and a unique right-most infinite upward geodesic starting from all $(x,s) \in \R^2$.  We will denote by $g_{(x,s)}$ the left-most infinite upward geodesic starting from $(x,s)$.

We can interpret the infinite upward geodesics in the directed landscape as describing the genealogical structure in a population in which there is an individual at each point in $\mathbb{R}^2$, and the ancestors of the individual located at $(x, s)$ are the points on the infinite upward geodesic $g_{(x,s)}$.  The tree formed from the infinite upward geodesics starting from $n$ sampled points represents the genealogical tree of $n$ individuals sampled from a population.  We will consider two different methods of sampling.  To model choosing $n$ individuals uniformly at random from a growing two-dimensional population, we will sample $n$ points uniformly at random from a box in $\mathbb{R}^2$.  To model choosing $n$ individuals uniformly at random from the outer edge of a growing two-dimensional population, we will sample $n$ points uniformly at random from a horizontal line segment.  

This approach essentially captures the behavior for a pre-limiting model such as FPP on a patch of a characteristic scale.  After a long time, the particles in FPP will occupy a large convex region, as in Figure \ref{simlineages}.  Because the directed landscape is obtained as a scaling limit, it only captures the geometry at the scale of the correlation, evident from the fact that it is stationary in time and space.  Therefore, the tree formed by the infinite upward geodesics starting from $n$ points sampled from a box in $\mathbb{R}^2$ should resemble the genealogical tree of $n$ particles sampled from a small patch far from the origin in the left panel in Figure \ref{simlineages}.  Likewise, the tree formed by the infinite upward geodesics starting from $n$ points sampled from a line segment should resemble the genealogical tree of $n$ particles sampled from a small arc along the boundary of the occupied region in the right panel in Figure \ref{simlineages}.  Nevertheless, because the quantities that we will be studying depend only on the local features of the genealogical tree, we expect that our main results should be approximately valid in prelimiting models such as FPP, even when individuals are sampled from the entire population.

\subsection{Main results}

Before stating our main results, we introduce some notation.
Recall that for any $(x,s) \in \R^2$, space-time pairs along the infinite upward geodesic $g_{(x,s)}$ represent ancestors of $(x,s)$. Therefore, the descendants of $(y,t)$ are the points whose infinite upward geodesics pass through $(y,t)$.
For each $(y,t)\in\R^2$, we define the descendants of $(y,t)$ to be 
$$
D_{(y,t)} :=\{(x,s)\in\R^2: s\le t, \, g_{(x,s)}(t-s)=y\}.
$$
For any $A\subseteq\R^2$, we define the ancestors of $A$ as the union of the infinite upward geodesics emanating from points in $A$, so
$$
Anc(A):=\bigcup_{(x,s)\in A} \bigcup_{u\ge 0} \{(g_{(x,s)}(u), s+u)\} = \{(y,t):D_{(y,t)}\cap A\neq\emptyset\}.
$$
The individuals that are ancestral to exactly $k$ individuals in $A$ are denoted by
$$
Anc_k (A):=\{(y,t)\in Anc(A): \#( D_{(y,t)}\cap A)=k\},
$$
where $\#S$ denotes the cardinality of the set $S$. If $A$ is finite, we usually measure the set $Anc_k(A)$ in terms of its branch length. More generally, for any measurable $B\subseteq\R^2$, we write $B^t=\{y\in\R: (y,t)\in B\}$ for the time-$t$ slice of $B$, and define its branch length as
\begin{equation}\label{def length}
l(B) :=\int_{\R} \# B^t\ dt.
\end{equation}
Note that $l(B)$ is well-defined because it is not hard to see that $t \mapsto \# B^t$ is measurable.  Also, by a countability argument, one can show that $Anc(A)$ and $Anc_k(A)$ are measurable subsets of $\R^2$ when $A$ is a finite or countable random subset of $\R^2$ because infinite upward geodesics can be approximated by the infinite upward geodesics starting at rational points.

With the above preparation we are now in a position to state our main results. Theorem~ \ref{Theorem: main2} gives the asymptotic behavior of the total length of branches that are ancestral to $k$ of the $n$ points sampled uniformly from a horizontal line segment. Theorem \ref{Theorem: main1} gives the asymptotic behavior of the total length of branches that are ancestral to $k$ of the $n$ points sampled uniformly from a box. The final main result Theorem \ref{Theorem: 3-disjoint}, while being a crucial input in the proofs of the other results, is itself of independent interest. It pertains to the well known N3G problem which states that there are no three disjoint infinite geodesics in any direction (note that while our focus will only be on geodesics in the vertical direction, one may consider geodesics with any asymptotic direction lying in the upper hemisphere of the unit circle). 

Here and in the rest of the document, constants are strictly positive and finite unless otherwise stated. Numbered constants such as $C_1$, $C_2$, $C_1'$, and $C_2'$ will not change, but unnumbered constants could change from line to line.

\begin{theorem}\label{Theorem: main2}
Let $\widetilde{U}_n$ be $n$ points sampled uniformly at random from $[0,n]\times\{0\}$. Then there exists a constant $C_1\in(0,\infty)$ such that as $n\rightarrow\infty$, for all positive integers $k$,
\begin{equation}\label{Introduction: main2}
\frac{l(Anc_k(\widetilde{U}_n))}{n}\rightarrow C_1\cdot \frac{\Gamma(k + 1/2)}{k!}\qquad\text{in probability}.
\end{equation}
\end{theorem}

\begin{theorem}\label{Theorem: main1}
Let $U_n$ be $n$ points sampled uniformly at random from the box $A_n=[0,n^{2/5}]\times [0,n^{3/5}]$. Then there exists a constant $C_2\in(0,\infty)$ such that as $n\rightarrow\infty$, for all positive integers $k$,
\begin{equation}\label{Introduction: main1}
\frac{l(Anc_k(U_n))}{n}\rightarrow C_2\cdot\frac{\Gamma(k - 2/5)}{k!}\qquad\text{in probability}.
\end{equation} 
\end{theorem}

\begin{remark}
{\em We will see that the choice of the box $A_n=[0,n^{2/5}]\times [0,n^{3/5}]$ is consistent with the scaling property of the directed landscape $\L$ (see Proposition \ref{Proposition: properties of the landscape} below), and therefore one could also formulate this result in terms of sampling $n$ points from the unit box $[0,1] \times [0,1]$.}
\end{remark}
 
Note that the quantities $l(Anc_k(\widetilde{U}_n))$ and $l(Anc_k(U_n))$ play the same role as $L_{k,n}$ in the results for Kingman's coalescent and supercritical birth and death processes.  These quantities should therefore be proportional to the number of mutations that appear on $k$ of the $n$ sampled individuals as long as the mutation rate is constant.  Since we have $\Gamma(k + 1/2)/k! \sim k^{-1/2}$ and $\Gamma(k - 2/5)/k! \sim k^{-7/5}$, where $\sim$ means that the ratio of the two sides tends to one as $k \rightarrow \infty$,  these results are consistent with the predictions of \cite{epf22} and \cite{fgkah16}.

An important ingredient for the proofs of Theorems \ref{Theorem: main2} and Theorem \ref{Theorem: main1} is the following bound on the three-arm probability, i.e. the non-intersection probability of three infinite upward geodesics in the directed landscape.  
\begin{theorem}\label{Theorem: 3-disjoint}
Consider the infinite upward geodesics $g_{(-1,0)}$, $g_{(0,0)}$, and $g_{(1,0)}$. There exists a constant $C\in(0,\infty)$ such that for all $t \geq 2$,
$$
\P(g_{(-1,0)}, g_{(0,0)}, \text{ and } g_{(1,0)} \text{ are disjoint for all } s\in[0,t])\le Ct^{-5/3}\log^{52} t.
$$
\end{theorem}
The poly-logarithmic term in the bound is a byproduct of some of the estimates available and is expected to be superfluous. 

\newcommand{\ncg}{\mathsf{3DG}(1)}
In the literature, the event of the non-existence of three disjoint geodesics (in any possible direction) is known as $\mathsf{N3G}$. This was investigated recently in \cite{b24} which  ruled out the existence of three disjoint geodesics in any direction. An important ingredient in the proof was to consider the event in the above theorem which we will denote as $\ncg$ owing to the fact that the event involves the existence of three disjoint geodesics. The $1$ in $\ncg$ denotes that we only stipulate disjointness of the infinite upward geodesics up to height $1$. The main ingredient in \cite{b24} is an upper bound of $t^{-16/45}$ of the probability of $\ncg$. 
However, for our purposes we need a sharper estimate. The fact that $5/3$ is the right exponent can be reduced to events about the parabolic Airy line ensemble.  This connection was also observed in \cite{b24} and has been the basis of several recent developments in the study of fractal geometry in the directed landscape in particular with connections to non-coalescence of geodesics. {See e.g., \cite{bgh21, bgh22, gz22, gz22a, gm23} and also \cite{gm24} for a brief review of related developments. In light of this, we expect the above result to be of independent interest and find applications beyond those in this article.

To understand why Theorem \ref{Theorem: 3-disjoint} is important for our results, consider the setting of Theorem~\ref{Theorem: main2}, where we have $n$ infinite upward geodesics from points in $\widetilde{U}_n$, which are sampled at random from $[0,n] \times\{0\}$.  Denote the $x$-coordinates of these sampled points by $x_1 < \dots < x_n$.  For $2 \leq k \leq n-1$, consider the portion of the infinite upward geodesic starting from $(x_k, 0)$ that is ancestral to only one of the $n$ sampled points and therefore contributes to $Anc_1(\widetilde{U}_n)$.  The length of this portion of the geodesic is the time that it takes for this geodesic to coalesce with either the geodesic started from $(x_{k-1}, 0)$ or the geodesic starting from $(x_{k+1},0)$.  Theorem \ref{Theorem: 3-disjoint} can be used to bound the probability that this time is greater than $t$ and, in particular, establishes that this time has finite mean.  Thus for our purposes, we need a bound of the form  $t^{-a}$ for any $a > 1$ which the estimate from \cite{b24} falls short of. Finally, while we do not formally prove a matching lower bound, we expect that the arguments in the upper bound can be reversed without too much difficulty, for instance by invoking reasoning as in \cite{d23}. 

\subsection{Understanding the exponents}

\begin{figure}
\centering
\begin{tikzpicture}[scale=0.6]
\draw [thick] (0,4)--(8,4);
\draw [thick] [red] (2,2.5) ellipse (1cm and 1.5cm);
\draw [thick] [red] (4,2.3) ellipse (1cm and 1.7cm);
\draw [thick] [red] (6,2.7) ellipse (1cm and 1.3cm);
\draw [thick] (0,2.3)--(8,2.3);
\draw (9,4)--(9,2.3);
\draw (9,4)--(8.8,4);
\draw (9,2.3)--(8.8,2.3);
\node at (9.7,3.2){$r^{3/2}$};
\draw (5.2,0.5)--(6.8,0.5);
\draw (5.2,0.5)--(5.2,0.7);
\draw (6.8,0.5)--(6.8,0.7);
\node at (6,0){$r$};
\draw [fill=black](2,4) circle (4pt);
\draw [fill=black](4,4) circle (4pt);
\draw [fill=black](6,4) circle (4pt);
\draw [thick] [decorate, decoration={random steps,segment length=6pt,amplitude=1.5pt}] (1.1, 2.3)--(2,4);
\draw [thick] [decorate, decoration={random steps,segment length=6pt,amplitude=1.5pt}] (2.9, 2.3)--(2,4);
\draw [thick] [decorate, decoration={random steps,segment length=6pt,amplitude=1.5pt}] (3.1, 2.3)--(4,4);
\draw [thick] [decorate, decoration={random steps,segment length=6pt,amplitude=1.5pt}] (4.9, 2.3)--(4,4);
\draw [thick] [decorate, decoration={random steps,segment length=6pt,amplitude=1.5pt}] (5.2, 2.3)--(6,4);
\draw [thick] [decorate, decoration={random steps,segment length=6pt,amplitude=1.5pt}] (6.9, 2.3)--(6,4);
\end{tikzpicture}
\caption{The three dots represent individuals along the top line with descendants surviving at least $r^{3/2}$ time units into the future, and the red ovals represent the descendants of these three individuals.  The area of the ovals is $O(r^{5/2})$.}
\label{expfigure}
\end{figure}

The exponents $7/5$ and $1/2$ arise as consequences of geodesic coalescence and the aforementioned scaling properties of the directed landscape.  We give here a brief heuristic argument for these exponents.  Two infinite upward geodesics that begin one unit apart should coalesce after a time which is $O(1)$.  Therefore, by the scaling property of the directed landscape, if we scale the horizontal axis by $r$, then the vertical axis must be scaled by $r^{3/2}$ for the law of the geodesics to remain the same.  Therefore, two infinite upward geodesics that begin $r$ units apart should coalesce after a time which is $O(r^{3/2})$.  It follows that points along a line segment with descendants surviving $r^{3/2}$ time units into the future will be spaced $O(r)$ units apart, and the region representing the descendants of these individuals will have an area which is $O(r^{5/2})$, as shown in Figure \ref{expfigure}.  Letting $k = r^{5/2}$, we see that points along a horizontal line with $O(k)$ descendants should be spaced $O(k^{2/5})$ apart, which suggests, by taking a derivative, that points with exactly $k$ descendants will be spaced $O(k^{7/5})$ apart.  This leads to the $k^{-7/5}$ behavior in Theorem \ref{Theorem: main1}.

In the same way, to see how the exponent of $1/2$ arises, note that the points that are ancestors of $O(k)$ of the sampled points along the $t = 0$ horizontal line should be at a height of $O(k^{3/2})$.  Further, along any horizontal line at height $O(k^{3/2})$, these points should be spaced $O(k)$ units apart.  Therefore, the total length of the portions of the geodesics that are ancestral to $O(k)$ of the sampled points should be proportional to $k^{3/2} \cdot k^{-1} = k^{1/2}.$ Since the derivative of $k^{1/2}$ is $k^{-1/2}$, this suggests that the total length of the portions of the geodesics that are ancestral to exactly $k$ of the sampled points should be $O(k^{-1/2})$.

One can use similar scaling arguments to examine the potential applicability of Theorems~\ref{Theorem: main2} and \ref{Theorem: main1} to pre-limiting models such as FPP.  Suppose FPP is run until $N$ sites are occupied, at which time the occupied region has radius $O(N^{1/2})$.  Suppose we sample $n$ particles from the boundary of the occupied region.  If exactly $k$ of the sampled particles share a common ancestor, then these $k$ particles will lie along an arc of length $O(k N^{1/2}/n)$.  Letting $r = kN^{1/2}/n$, the most recent common ancestor of these $k$ particles will be much closer to the boundary than to the center of the circle as long as $r^{3/2} \ll N^{1/2}$, which is equivalent to the condition that
\begin{equation}\label{prelimitcond}
\frac{k}{n} \ll N^{-1/6}
\end{equation}
This condition, along with the condition that $n \ll N^{1/2}$ so that the sampled points are sufficiently far apart, is the condition under which we would expect Theorem \ref{Theorem: main2} to give accurate results in pre-limiting models.  If instead the $n$ points are sampled from the entire population, then $k$ particles sharing a common ancestor will occupy a region of area $O(k N/n)$.  This region will have a width of $O(r)$ and a height of $O(r^{3/2})$ for $r = (kN/n)^{2/5}$.  As before, the directed landscape approximation should be accurate as long as $r^{3/2} \ll N^{1/2}$, which is again equivalent to \eqref{prelimitcond}.  In this case, we also need $n \ll N$ to ensure the sampled points are far enough apart.  We therefore expect Theorem \ref{Theorem: main1} to give accurate results in pre-limiting models when \eqref{prelimitcond} holds and $n \ll N$.

When $k$ is large enough that the branches that are ancestral to $k$ sampled points are near the center of the circle, the directed landscape approximation is no longer valid, and nonrigorous predictions in \cite{fgkah16} indicate that the number of mutations inherited by $k$ sampled individuals begins to decay like $k^{-5}$.  Simulation results and site frequency spectrum data from bacterial populations in \cite{fgkah16} approximately show this transition between two different power laws.  Understanding the $k^{-5}$ behavior in a mathematically rigorous way remains an interesting open problem.

\subsection{Organization of the paper}

In Section~\ref{outlinesec}, we outline the strategy behind the proofs of our main results. 
In Section~\ref{Section: Preliminary}, we collect some properties of the directed landscape that we will use later in the paper. We prove Propositions~\ref{Proposition: expectation, fixed time, 1} -- \ref{Proposition: diff2} in Section~\ref{Section: 1-dimension} and Propositions~\ref{Proposition: expectation, 1} -- \ref{Proposition: diff1} in Section~\ref{Section: 2-dimension}. In Section~\ref{Section: 3-arm}, we prove Theorem~\ref{Theorem: 3-disjoint}.  Simulation results are presented in Section~\ref{simsec}.

\section{An outline of the proofs}\label{outlinesec}

The proofs of Theorems \ref{Theorem: main2} and \ref{Theorem: main1} have the same broad strategy with the latter being more involved. The structure of the proof in either case involves considering several closely related but more tractable quantities and arguing that one can control the error in the approximations. 

\subsection{Outline of the proof of Theorem \ref{Theorem: main2}}

The first step in the proof of Theorem \ref{Theorem: main2}} is a scaling argument.  For each $(y,t)\in\R^2$ with $t>0$, let
\begin{equation*}
\widetilde{D}_{(y,t)} :=\{x\in\R: (x,0)\in D_{(y,t)}\}
\end{equation*}
be the descendants of $(y,t)$ at time zero.  Let $\widetilde{m}$ be the Lebesgue measure on $\R$. For any $a\ge0$, let
\begin{equation}\label{Introduction: def T_a, fixed time}
\widetilde{\V}_{a} = \{(y,t)\in\R^2: t> 0,\widetilde{m}(\widetilde{D}_{(y,t)})> a\}
\end{equation}
be the collection of space-time pairs whose descendants at time zero have Lebesgue measure more than $a$. 
We consider the ``length" of the set of points whose descendants at time zero form an interval whose length is positive, but no larger than than $a$.  The following result, whose proof uses only the scaling property of the directed landscape, establishes that this length is proportional to $a^{1/2}$.

\begin{proposition}\label{Proposition: expectation, fixed time, 1}
There exists a constant $C_1' \in (0, \infty)$ such that for any $a > 0$ and any strip $A=\widetilde{A}\times \R_+$, where $\widetilde{A}\subseteq\R$ is measurable, we have
\begin{equation}\label{cdf2}
\E\left[l\big((\widetilde{\V}_{0}\setminus \widetilde{\V}_{a})\cap A\big)\right]=C_1'\widetilde{m}(\widetilde{A})a^{1/2},
\end{equation}
\end{proposition}

The second step in the proof of Theorem \ref{Theorem: main2} is a Poissonization argument.
We consider taking a sample from the outer boundary of the population using a family of Poisson point process $\widetilde{\Pi}=(\widetilde{\Pi}_{\lambda})_{\lambda>0}$ on $\R\times\{0\}$, where $\widetilde{\Pi}_{\lambda}$ has intensity $\lambda$.  The Poisson processes will be coupled so that $\widetilde{\Pi}_\lambda\subseteq \widetilde{\Pi}_{\lambda'}$ whenever $\lambda<\lambda'$. The following result then follows directly from the scaling in Proposition \ref{Proposition: expectation, fixed time, 1}.

\begin{proposition}\label{Proposition: expectation, fixed time, 2}
Let $C_1'$ be the constant from Proposition \ref{Proposition: expectation, fixed time, 1}.
\begin{equation}\label{sfs2}
\E\left[l(Anc_k(\widetilde{\Pi}_{\lambda})\cap A)\right]
= \frac{C_1'\lambda^{-1/2}}{2}\cdot\widetilde{m}(\widetilde{A}) \cdot \frac{\Gamma(k + 1/2)}{k!}.
\end{equation}
\end{proposition}

The third step in the proof of Theorem \ref{Theorem: main2} is to show that a strong law of large numbers also holds for $l(Anc_k(\widetilde{\Pi}_\lambda)\cap A_n)$ if we take $A_n = [0, n] \times \R^+$.  This result will be proved using Birkhoff's ergodic theorem.

\begin{proposition}\label{Proposition: lln, fixed time, points anywhere}
Let $A_n = [0,n]\times\R_+$, and let $C_1'$ be the constant Proposition \ref{Proposition: expectation, fixed time, 1}. Then, as $n\rightarrow\infty$,
$$
\frac{l(Anc_k(\widetilde{\Pi}_\lambda)\cap A_n)}{n}\rightarrow \frac{C_1'\lambda^{-1/2}}{2}\cdot \frac{\Gamma(k + 1/2)}{k!},\qquad \text{a.s. and in } L^1.
$$
\end{proposition}

In the proof of Theorem \ref{Theorem: main2}, we will take $C_1$ to be $C_1'/2$.  In this case, Theorem~\ref{Theorem: main2} and Proposition \ref{Proposition: lln, fixed time, points anywhere} take similar forms, except for two major differences. Firstly, the sampling mechanisms are different.  Theorem \ref{Theorem: main2} uses uniform sampling of exactly $n$ points while Proposition~\ref{Proposition: lln, fixed time, points anywhere} uses Poisson sampling.
Secondly, in Theorem \ref{Theorem: main2}, we consider the points $(y,t)\in \R^2$ that are ancestral to $k$ points sampled from the interval $[0, n]$, whereas in Proposition \ref{Proposition: lln, fixed time, points anywhere}, we consider the points $(y,t)\in [0,n] \times \R^+$ that are ancestral to $k$ points sampled from somewhere in $\R$.  

To bridge the gap between Proposition \ref{Proposition: lln, fixed time, points anywhere} and Theorem~\ref{Theorem: main2}, we let
\begin{equation}\label{lambda2}
\widetilde{\lambda}_n^*:=\inf\{\lambda>0: \#(\widetilde{\Pi}_{\lambda}\cap A_n)=n\} 
\end{equation}
be the (random) smallest $\lambda$ so that we sample $n$ points from $[0,n]\times\{0\}$. Note that $\widetilde{\Pi}_{\widetilde{\lambda}_n^*}\cap A_n$ has the same distribution as $\widetilde{U}_n$ in Theorem \ref{Theorem: main2}. The next proposition controls the difference between Poisson sampling and uniform sampling. Because the coupling of the processes for different values of $n$ is less natural when we use $\widetilde{\lambda}_n^*$, we prove convergence in probability rather than almost sure convergence, which is sufficient for the proof of Theorem \ref{Theorem: main2}.

\begin{proposition}\label{Proposition: lln, fixed time, points in [0,n]}
Let $A_n=[0,n]\times\R_+$. Let $\widetilde{\lambda}_n^*$ be defined as in \eqref{lambda2}, and let $C_1'$ be the constant Proposition \ref{Proposition: expectation, fixed time, 1}. Then, as $n\rightarrow\infty$,
$$
\frac{l(Anc_k(\widetilde{\Pi}_{\widetilde{\lambda}_n^*})\cap A_n)}{n}\rightarrow \frac{C_1'}{2}\cdot \frac{\Gamma(k + 1/2)}{k!},\qquad \text{in probability}.
$$
\end{proposition}

It remains to address the second issue mentioned above. We will show that with high probability, most points that are ancestral to $k$ points sampled from $[0,n] \times \{0\}$ are in the strip $A_n$, and most points in $A_n$ that are ancestral to $k$ points sampled from $\R \times \{0\}$ are ancestral to $k$ points sampled from $[0,n] \times \{0\}$.  
The exceptions to this are essentially edge effects, whose contributions can be bounded by using results on the transversal fluctuations of geodesics.
These observations are made precise in the following proposition.

\begin{proposition}\label{Proposition: diff2}
Let $A_n = [0,n]\times \R_+$. Let $\widetilde{\lambda}_n^*$ be defined as in \eqref{lambda2}. Then, as $n\rightarrow\infty$,
$$
\frac{l(Anc_k(\widetilde{\Pi}_{\widetilde{\lambda}_n^*})\cap A_n)-l(Anc_k(\widetilde{\Pi}_{\widetilde{\lambda}_n^*}\cap A_n))}{n}\rightarrow0,\qquad \text{in probability.} 
$$
\end{proposition}

Because $\widetilde{\Pi}_{\widetilde{\lambda}_n^*}\cap A_n$ has the same distribution as $\widetilde{U}_n$, Theorem \ref{Theorem: main2} follows immediately from Propositions \ref{Proposition: lln, fixed time, points in [0,n]} and \ref{Proposition: diff2}.

\subsection{Outline of the proof of Theorem \ref{Theorem: main1}}

While the execution is more complicated, the strategy of the proof of Theorem \ref{Theorem: main1} proceeds in similar steps as Theorem \ref{Theorem: main2}.  Let $m$ be the Lebesgue measure on $\R^2$. For any $a>0$, let
\begin{equation}\label{Introduction: def T_a}
\V_{a} = \{(y,t)\in\R^2: m(D_{(y,t)})> a\},
\end{equation}
be the collection of space-time pairs whose descendants have Lebesgue measure (or volume) more than $a$.
Recall that $\V_a^t$ is the time-$t$ slice of $\V_a$.  As in the case when points are sampled along a line, the translation invariance properties of the directed landscape allow us to use a scaling argument to prove the following proposition. 

\begin{proposition}\label{Proposition: expectation, 1}
There exists a constant $C_2'\in(0,\infty)$ such that for all measurable subsets $\widetilde{A}\subseteq\R$ and all $t\in\R$, we have
\begin{equation}\label{cdf1}
\E[\#(\V_{a}^t\cap \widetilde{A})]=C_2'\cdot\widetilde{m}(\widetilde{A}) a^{-2/5}.  
\end{equation}
\end{proposition}

We next use a Poissonization argument.  We take a sample from the population according to a family of Poisson point processes $\Pi=(\Pi_{\lambda})_{\lambda>0}$ on $\R^2$, where $\Pi_{\lambda}$ has intensity $\lambda$, coupled so that $\Pi_\lambda\subseteq \Pi_{\lambda'}$ whenever $\lambda<\lambda'$.  

\begin{proposition}\label{Proposition: expectation, 2}
Let $C_2'$ be the constant from Proposition \ref{Proposition: expectation, 1}.
Then for all measurable $A\subseteq\R^2$, we have
\begin{equation}\label{sfs1}
\E[l(Anc_k(\Pi_\lambda)\cap A)] =\frac{2C_2'\lambda^{2/5}}{5}\cdot m(A) \cdot \frac{\Gamma(k - 2/5)}{k!}.
\end{equation} 
\end{proposition}

The next step is to establish a strong law of large numbers.  The result relies on a counterpart of Birkhoff's ergodic theorem involving a $\Z^2$ action since it involves the bulk of the growth cluster (both one-dimensional space and one-dimensional time). In addition to these classical ergodic theorems, bounds on the transversal fluctuations of geodesics are necessary to show that the directed landscape has sufficient mixing properties for these results to hold. A key ingredient in the application of the two-dimensional ergodic theorem is a temporal mixing property of the directed landscape, established in \cite{dov22}.

\begin{proposition}\label{Proposition: lln, points anywhere}
Let $A_n = [0,n^{2/5}]\times[0,n^{3/5}]$, and let $C_2'$ be the constant in Proposition \ref{Proposition: expectation, 1}. Then, as $n\rightarrow\infty$,
$$
\frac{l(Anc_k(\Pi_\lambda)\cap A_n)}{n}\rightarrow \frac{2C_2'\lambda^{2/5}}{5}\cdot \frac{\Gamma(k - 2/5)}{k!},\qquad \text{a.s. and in } L^1. 
$$
\end{proposition}

To get Theorem \ref{Theorem: main1} from Proposition \ref{Proposition: lln, points anywhere}, we will take the constant $C_2$ to be $2C_2'/5$.  As in the case when points are sampled from a line, the sampling mechanisms are different between Theorem \ref{Theorem: main1} and Proposition \ref{Proposition: lln, points anywhere}.  To show that this difference does not affect the results, we let 
\begin{equation}\label{lambda1}
\lambda_n^*:=\inf\{\lambda>0: \#(\Pi_{\lambda}\cap A_n)\}=n.
\end{equation}
Note that $\Pi_{\lambda_n^*}\cap A_n$ has the same distribution as $U_n$ in Theorem \ref{Theorem: main1}.  We then have the following result.

\begin{proposition}\label{Proposition: lln, points in the box}
Let $A_n = [0,n^{2/5}]\times[0,n^{3/5}]$. Let $\lambda_n^*$ be defined as in \eqref{lambda1} and let $C_2'$ be the constant in Proposition \ref{Proposition: expectation, 1}. Then, as $n\rightarrow\infty$,
$$
\frac{l(Anc_k(\Pi_{\lambda_n^*})\cap A_n)}{n}\rightarrow \frac{2C_2'}{5}\cdot \frac{\Gamma(k - 2/5)}{k!},\qquad \text{in probability.} 
$$
\end{proposition}

The other difference between Theorem \ref{Theorem: main1} and Proposition \ref{Proposition: lln, points anywhere} is that in Theorem \ref{Theorem: main1}, we consider the points $(y,t)\in \R^2$ that are ancestral to $k$ points sampled from the box $A_n$, whereas in Proposition \ref{Proposition: lln, points anywhere}, we consider the points $(y,t)\in A_n$ that are ancestral to $k$ points sampled from somewhere in $\R^2$; see Figure \ref{Figure: diff1} for an illustration.  The following proposition shows that these differences do not affect the limit behavior.

\begin{figure}[h]
\centering
\begin{tikzpicture}[scale=0.25]
\draw[draw=black] (0,0) rectangle ++(20,20);

\draw plot [smooth, tension=1] coordinates {(-2,8) (2, 13) (3.5, 17) (5,21) (4,23) (4,24)};
\draw plot [smooth, tension=1] coordinates {(2,6) (1,8) (2, 13)};
\draw plot [smooth, tension=1] coordinates {(6,10) (4,13) (3.5,17) };
\draw plot [smooth, tension=1] coordinates {(5,15) (6.5,17) (5,21)};
\draw plot [smooth, tension=1] coordinates {(2,21) (3.5,22) (4,23)};
\begin{scope} 
\clip(0,17) rectangle (20,20);
\draw [ultra thick] plot [smooth, tension=1] coordinates {(-2,8) (2, 13) (3.5, 17) (5,21) (4,23) (4,24)};
\end{scope}
\node at (-2,7) {$a$};
\node at (2,5) {$b$};
\node at (6,9) {$c$};
\node at (5,14) {$d$};
\node at (1,22) {$T_1$};

\draw plot [smooth, tension=1]  coordinates {(9,6) (12,10)  (11, 14) (14, 18)};
\draw plot [smooth, tension=1] coordinates {(11,5) (12,8) (12,10)};
\draw plot [smooth, tension=1] coordinates {(13,4) (12,6) (12,8)};
\draw plot [smooth, tension=1] coordinates {(15,10) (14,12) (12, 14) (12, 16)};
\begin{scope} 
\clip(0,10) rectangle (20,16);
\draw [ultra thick] plot [smooth, tension=1]  coordinates {(9,6) (12,10)  (11, 14) (14, 18)};
\end{scope}
\node at (9,5) {$e$};
\node at (11,4) {$f$};
\node at (13,3) {$g$};
\node at (14,14) {$T_2$};

\draw plot [smooth, tension=1] coordinates {(2,6) (1,8) (2, 13)};
\draw plot [smooth, tension=1] coordinates {(6,10) (4,13) (3.5,17) };
\draw plot [smooth, tension=1] coordinates {(5,15) (6.5,17) (5,21)};
\draw plot [smooth, tension=1] coordinates {(2,21) (3.5,22) (4,23)};
\node at (28,7) {$a$};
\node at (32,5) {$b$};
\node at (36,9) {$c$};
\node at (35,14) {$d$};
\node at (31,22) {$T_1$};
\node at (10,-2) {$Anc_k(\Pi_{\lambda_n^*})\cap A_n$};

\draw[ draw=black] (30,0) rectangle ++(20,20);
\draw plot [smooth, tension=1] coordinates {(28,8) (32, 13) (33.5, 17) (35,21) (34,23) (34,24)};
\draw plot [smooth, tension=1] coordinates {(32,6) (31,8) (32, 13)};
\draw plot [smooth, tension=1] coordinates {(36,10) (34,13) (33.5,17) };
\draw plot [smooth, tension=1] coordinates {(35,15) (36.5,17) (35,21)};
\draw plot [smooth, tension=1] coordinates {(32,21) (33.5,22) (34,23)};
\begin{scope} 
\clip(30,21) rectangle (50,23);
\draw [ultra thick] plot [smooth, tension=1] coordinates {(28,8) (32, 13) (33.5, 17) (35,21) (34,23) (34,24)};
\end{scope}
\node at (39,5) {$e$};
\node at (41,4) {$f$};
\node at (43,3) {$g$};
\node at (44,14) {$T_2$};

\draw plot [smooth, tension=1]  coordinates {(39,6) (42,10)  (41, 14) (44, 18)};
\draw plot [smooth, tension=1] coordinates {(41,5) (42,8) (42,10)};
\draw plot [smooth, tension=1] coordinates {(43,4) (42,6) (42,8)};
\draw plot [smooth, tension=1] coordinates {(45,10) (44,12) (42, 14) (42, 16)};
\begin{scope} 
\clip(30,10) rectangle (50,16);
\draw [ultra thick] plot [smooth, tension=1]  coordinates {(39,6) (42,10)  (41, 14) (44, 18)};
\end{scope}
\node at (40,-2) {$Anc_k(\Pi_{\lambda_n^*}\cap A_n)$};
\end{tikzpicture}
\caption{A comparison between $Anc_k(\Pi_{\lambda_n^*})\cap A_n$ (left plot) and $Anc_k(\Pi_{\lambda_n^*}\cap A_n)$ (right plot) with $k=3$. The thick lines represent branches that account for the length. The tree $T_1$ illustrates the case where the branch is counted in one but not the other. In the left plot, the thick portion is in the box $A_n$ and is along the ancestral lineage of ${a,b,c}$ in $\Pi_{\lambda_n^*}$. In the right plot, the thick portion is along the ancestral lineage of ${b,c,d}$ in $\Pi_{\lambda_n^*}\cap A_n$. These differences typically occur near the boundary of the box.  The tree $T_2$ illustrates the typical case where the branch length counted in one is also counted in the other.}
\label{Figure: diff1}
\end{figure}

\begin{proposition}\label{Proposition: diff1}
Let $A_n = [0,n^{2/5}]\times[0,n^{3/5}]$. Let $\lambda_n^*$ be defined as in \eqref{lambda1}. Then, as $n\rightarrow\infty$,
$$
\frac{l(Anc_k(\Pi_{\lambda_n^*})\cap A_n)-l(Anc_k(\Pi_{\lambda_n^*}\cap A_n))}{n}\rightarrow0,\qquad \text{in probability.} 
$$
\end{proposition}

Because $\Pi_{\lambda_n^*}\cap A_n$ has the same distribution as $U_n$, Theorem \ref{Theorem: main1} follows immediately from Propositions \ref{Proposition: lln, points in the box} and \ref{Proposition: diff1}.\\

\subsection{Outline of the proof of Theorem \ref{Theorem: 3-disjoint}}

By rescaling it suffices to prove that for all $\varepsilon \le 1/2,$
\begin{equation}\label{scaled123}
\P(g_{(-\varepsilon,0)}, g_{(0,0)}, \text{ and } g_{(\varepsilon,0)} \text{ are disjoint for all } s\in[0,1])\le C\varepsilon^{5/2}\log^{52} (1/\eps).
\end{equation}

While a detailed proof idea is presented in Section \ref{Section: 3-arm} where all the relevant objects are introduced, for now let us simply mention that by surgery arguments involving geodesic watermelons (sets of disjoint paths cumulatively maximizing their energy), the above event via the parabolic Airy line ensemble (a central object in the KPZ universality class) can be approximated by essentially five independent events. Three of them are of the form that a Brownian excursion on a unit order interval comes within $O(\sqrt {\eps})$ of $0$ in the bulk of the interval which happens with probability $O(\sqrt {\eps}).$ These correspond to the initial part of the three geodesics (up to height $1$) being disjoint. The remaining events ensure that these finite geodesics are subsegments of semi-infinite geodesics. The latter translates to the maximum of three independent Brownian motions on $[0,1]$ falling within $O(\sqrt {\eps})$ of each other. If all the events were independent this would yield an $O(\eps^{5/2})$ upper bound. Thus, much of the proof, which turns out to be quite complicated and involves studying the random curves at progressively finer scales, is devoted towards proving this approximate independence. 

\section{Results for the directed landscape}\label{Section: Preliminary}

We collect here some facts about the directed landscape that we will repeatedly rely on.  We begin by recording some properties of the directed landscape which can be found in \cite{dov22}.

\begin{proposition}[\cite{dov22}, Lemma 10.2]\label{Proposition: properties of the landscape}
The directed landscape admits the following symmetries. Here $r,c\in\R$ and $\rho>0$, and $\stackrel{d}{=}$ refers to equality in distribution when the two sides are viewed as random functions in $C(\R_\uparrow^4, \R)$. 
\begin{enumerate}
\item (Time stationarity)
$$
\L(x,s;y,t)\stackrel{d}{=}\L(x,s+r;y,t+r).
$$   
\item (Spatial stationarity)
$$
\L(x,s;y,t)\stackrel{d}{=}\L(x+c,s;y+c,t).
$$
\item (Rescaling)
$$
\L(x,s;y,t)\stackrel{d}{=} \rho \L(\rho^{-2} x, \rho^{-3}s; \rho^{-2}y, \rho^{-3} t).
$$
\item (Skew stationarity/shear invariance)
$$
\L(x,s;y,t)\stackrel{d}{=}\L(x+cs,s;y+ct,t)+\frac{(x-y)^2 - (x-y-(t-s)c)^2}{t-s}.
$$
\end{enumerate}
\end{proposition}

\begin{proposition}[\cite{dov22}, Definition 10.1]\label{Proposition: independent increment}
For any disjoint time intervals $\{(s_i, t_i):i\in\{1,\dots, k\}\}$, the random functions 
$$
\L(\cdot,s_i,\cdot,t_i), \qquad i\in\{1,\dots,k\}
$$
are independent.
\end{proposition}

We next quote the following two estimates. The first one is a global estimate and the next one is a modulus of continuity estimate. While pre-limiting versions can be found in \cite{gh23}, we will quote results from \cite{dov22} which are directly for the landscape.
First, let $\mathcal{K}$ be the parabolically adjusted version of the directed landscape. That is 
$$\mathcal{K}(x,s;y,t)=\mathcal{L}(x,s;y,t)+ \frac{(x-y)^2}{t-s}.$$
In what follows, $\|\cdot\|$ denotes the Euclidean norm. 

\begin{theorem}[\cite{dov22}, Corollary 10.7]\label{hold432} There is a random number $R$ and constants $c$ and $d$ such that $\P(R>m)\le c \exp(-d m^{3/2})$, and the following holds. For any $u=(x,s;y,t+s) \in \R^4_{\uparrow}$, 
\begin{align} \label{hold12}
|\mathcal{K}(x,s;y,t+s)| \le R t^{1/3} \log^{4/3}\left(\frac{2(\|u\|+2)}{t}\right) \log^{2/3}(\|u\|+2).
\end{align}
\end{theorem}
For the next result, we continue to follow the notation from \cite{dov22}.  
Let $$K_b^{\eps}=[-b, b]^4\cap \{(x,s;y,t)\in \mathbb{R}^{4}_{\uparrow}: t-s \ge \varepsilon\}.$$
Given $u_1=(x_1,s_1;y_1,t_1)$ and $u_2=(x_2,s_2;y_2,t_2)$ in $\R^4_{\uparrow}$, let $\xi = \xi(u_1,u_2)=\|(x_1,y_1)-(x_2,y_2)\|$ and $\tau = \tau(u_1,u_2)=\|(s_1,t_1)-(s_2,t_2)\|$.  

\begin{theorem}[\cite{dov22}, Proposition 10.5]\label{modcont} Fix $b\ge 2$ and $\varepsilon \le 1.$ There exists a random number $R$ and constants $c$ and $d$ such that $\P(R\ge m) \le c b^{10} \varepsilon^{-6} e^{-d m^{3/2}},$ and the following holds.  For all $u_1, u_2 \in K_b^{\eps}$ for which $\tau \leq \eps^3/b^3$, we have
\begin{equation}\label{unifbnd345}
|\mathcal{K}(u_1)-\mathcal{K}(u_2)|\le R[\tau^{1/3} \log^{2/3} (\tau^{-1})+\xi^{1/2}\log^{1/2}(4b\xi^{-1})].
\end{equation}
\end{theorem}
The next result is a spatial de-correlation result. For any $A\subseteq\R^2$, by a slight abuse of notation, we write $\L|_A$ for the restriction of $\L$ to $\R_{\uparrow}^4\cap(A\times A)$, so that we consider the values of $\L(x,s;y,t)$ when $s<t$ and $(x,s),(y,t)\in A$.

\begin{proposition}[\cite{d24n}, Proposition 2.6]\label{Proposition: coupling}
Fix $n \in \N$. For every $i = 1,\dots,n$ and $\eps > 0$, let $K_{i,\eps} =[i-1/4,i+1/4]\times[0,\eps]$.
We can couple $n+1$ copies of the directed landscape $\L_0,\L_1,\dots,\L_n$ so that $\L_1,...,\L_n$ are independent and almost surely for all small enough $\eps >0$ and $i = 1,\dots,n$ we have $\L_0|_{K_{i,\eps}} =\L_i|_{K_{i,\eps}}$.  
\end{proposition}

We now collect some results concerning geodesics in the directed landscape.  Recall that for $(x,s) \in \R^2$, we denote by $g_{(x,s)}$ the left-most infinite upward geodesic starting from $(x,s)$, and for $u \geq 0$, we write $g_{(x,s)}(u)$ for the spatial position of this geodesic at time $s + u$.  Also, for $(x,s;y,t)\in\R^4_{\uparrow}$, we denote by $g_{(x,s;y,t)}$ the left-most finite geodesic connecting $(x,s)$ and $(y,t)$, and for $0 \le u \le t-s$, we write $g_{(x,s;y,t)}(u)$ for the spatial position of this geodesic at time $s+u$.

The next proposition, which pertains to the lengths of concatenated paths in the directed landscape, is useful for proving results about geodesics.

\begin{proposition}[\cite{d23}, Fact 3 on page 16]\label{Proposition: concatenation}
Let $\gamma_1$ be a path from $(x_1,s_1)$ to $(x_2,s_2)$, and let $\gamma_2$ be a path from $(x_2,s_2)$ to $(x_3,s_3)$. Define the concatenation of $\gamma_1$ and $\gamma_2$ as
$$
(\gamma_1\oplus\gamma_2)(t)=\left\{
\begin{split}
&\gamma_1(t),\qquad &t\in[0,s_2-s_1],\\
&\gamma_2(t-(s_2-s_1)),\qquad &t\in(s_2-s_1,s_3-s_1].
\end{split}
\right.
$$
Then $||\gamma_1\oplus\gamma_2||_\L=||\gamma_1||_\L+||\gamma_2||_\L$.
\end{proposition}

The next result establishes that geodesics do not cross.  We sometimes refer to this property as planarity, or the ordering of geodesics.

\begin{proposition}\label{Proposition: ordering of geodesics}  The following statements hold with probability one:
\begin{enumerate}
\item \textup{(\cite{bgh22}, Lemma 2.7)} For all $s<t$, $x_1\le x_2$, and $y_1\le y_2$, we have $g_{(x_1,s;y_1,t)}(u)\le g_{(x_2,s;y_2,t)}(u)$ for all $0\le u\le t-s$.

\item \textup{(\cite{rv21}, Theorem 3.19)} For all $s \in \R$ and $x < y$, we have $g_{(x,s)}(u) \leq g_{(y,s)}(u)$ for all $u \geq 0$.
\end{enumerate}
\end{proposition}

The following proposition states that all geodesics eventually meet, and that once they meet, they coalesce and never separate.

\begin{proposition}[\cite{rv21}, Corollary 3.20]\label{Proposition: geodesic coalesecent}
With probability one, for all points $(x,s) \in \R^2$ and $(y,t) \in \R^2$, there exists $u \geq \max\{s,t\}$ such that $g_{(x,s)}(u-s) = g_{(y,t)}(u-t)$.  Furthermore, we have $g_{(x,s)}(v-s) = g_{(y,t)}(v-t)$ for all $v \geq u$.
\end{proposition}

The next result gives a bound on the transversal fluctuations of geodesics.

\begin{theorem}\label{Theorem: transversal fluctuation}
There exist constants $C_3,C_4\in(0,\infty)$ such that for all $(x,s)\in\R^2$ and $T>0$, the infinite upward geodesic $g_{(x,s)}$ satisfies
\begin{equation*}
\P\left(\sup_{u\in[0,T]}|g_{(x,s)}(u)-g_{(x,s)}(0)|\ge a \right)\le C_3 e^{-C_4 a ^3/T^2}.   
\end{equation*}
The analogous statement holds for any finite geodesic as well. By the scaling property and the shear invariance of the directed landscape, we can without loss of generality assume that the finite geodesic is $g_{(0,0;0,1)}.$ In this case, $T$ must be taken to be at most $1.$ Note that using symmetry (switching $(0,0)$ and $(0,1)$), it suffices to take $T$ to be at most $1/2.$
\end{theorem}
\begin{proof} A version of this result for the pre-limiting model of Exponential LPP can be found in \cite[Proposition 2.1 (ii)]{bbb23}. While the same argument could be made to work for the directed landscape, a shorter argument relies on the convergence of geodesics in the pre-limit to their limiting counterparts (see \cite{dov22}). To translate the pre-limiting bounds to the limiting counterpart, we will use the fact that the geodesic in Exponential LPP from $(0,0)$ to $(n,n)$ rotated by $45^{\circ}$, scaled by $n$ vertically and $n^{2/3}$ horizontally, converges to the geodesic $g_{(0,0;0,1)}$ from which the statement about finite geodesics follows. Now to extend to the case of the infinite geodesic $g_{(0,0)}$, by scale invariance, it suffices to prove the statement for $T\to 0.$ Almost surely the finite geodesic $g_{(0,0;0,1)}$ and the semi-infinite geodesic $g_{(0,0)}$ agree on an initial interval $[0,\varepsilon]$ for some random strictly positive $\varepsilon>0.$ This for instance follows from the main result in \cite{bgh22} which proves that a fixed point is almost surely a $1$-star, i.e. any two geodesics emanating from the given point must have a non-trivial overlap. The result is now an immediate consequence of the statement for finite geodesics.
\end{proof}

The next result is a quantitative consequence of geodesic coalescence.

\begin{proposition}\label{Proposition: number of intersection points}
For each interval $I\subseteq\R$ and $t>0$, let 
$$
N_{I,t}:=\#\{y\in I: (y,t)=g_{(x,0)}(t)\text{ for some }x\in\R\}
$$
be the number of points in $I \times \{t\}$ that lie on an infinite upward geodesic starting from time zero. We write $N_t=N_{[-1,1],t}$. Then $\E\left[N_{t}\right]<\infty$ for each fixed $t>0$.
\end{proposition} 

\begin{proof} For simplicity let us assume $t=1.$ The argument for any other $t$ is exactly the same.  By Theorem \ref{Theorem: transversal fluctuation} and the ordering of geodesics, with probability at least $1-e^{-c k^3},$ no geodesic $g_{(x,0)}$ with $|x| \ge k$ intersects $[-1,1] \times \{1\}$, and every geodesic $g_{(x,1)}$ with $|x| \le 1$ intersects $[-k,k] \times \{2\}.$  On this event, $N_{t}$ is upper bounded by the number of points on the line $y = 1$ that intersect finite geodesics $g_{(x,0;y,2)}$, where $|x| \leq k$ and $|y|\le k.$  This observation along with \cite[Lemma 3.12]{gz22} implies that $\P(N_1 > k)\le e^{-ck^{\alpha}}$ for some $\alpha> 0$ which finishes the proof. 
\end{proof}
The next two results estimate the probability of non-coalescence. 

\begin{proposition}\label{Proposition: 2-disjoint}
There exists a constant $C\in(0,\infty)$ such that
\begin{equation}\label{coal1234}
\P(g_{(0,0)} \text{ and } g_{(1,0)} \text{ are disjoint for all } s\in[0,t])\le Ct^{-2/3}, 
\end{equation}
or equivalently,
$$
\P(g_{(0,0)} \text{ and } g_{(M,0)} \text{ are disjoint for all } s\in[0,t])\le CMt^{-2/3}.
$$
\end{proposition}

\begin{proof} By scale invariance, it suffices to compute the probability, which we call $p$, that $g_{(0,0)}$ and $g_{(t^{-2/3},0)}$ do not meet by time $1.$ Now consider the geodesics $g_{(it^{-2/3},0)}$ where $i \in [-t^{2/3},t^{2/3}] \cap \Z.$ 
Let $N$ be the number of points on the line $y = 1$ that intersect one or more of these geodesics.  By the same argument as in the proof of Proposition \ref{Proposition: number of intersection points}, it follows that $\E(N)< \infty.$  Further, by planarity, each $i$ such that $g_{(it^{-2/3},0)}$ and $g_{((i+1)t^{-2/3},0)}$ do not meet before height $1$, increases $N$ by $1$.   Finally, by translation invariance, the probability of the non-coalescence event is independent of $i.$  It follows that $E[N] \geq 2(t^{2/3} - 1)p$, and therefore $p \leq Ct^{-2/3}$ (a version of this argument for Exponential LPP first appeared in \cite{bhs22}).
\end{proof}

\begin{proposition}\label{Proposition: disjoint finite geodesics}
There exists a constant $C$ such that for all $M>1$ and $t\in\R_+$, we have
\begin{equation}\label{disfin1}
\P(g_{(-1,0;-M,t)} \text{ and } g_{(1,0;M,t)} \text{ are disjoint})\le CMt^{-1}\exp(\log^{5/6}(t\vee 1)).
\end{equation}
or equivalently, for $x_1<x_2$ and $y_1<y_2$ with $y_2-y_1>x_2-x_1$,
\begin{align}\label{disfin2}
&\P(g_{(x_1,0;y_1,t)} \text{ and } g_{(x_2,0;y_2,t)} \text{ are disjoint}) \nonumber \\
&\qquad \qquad \le C(x_2-x_1)^{1/2}(y_2-y_1)t^{-1}\exp\left(\log^{5/6}\left(\frac{t}{(x_2-x_1)^{3/2}}\vee 1\right)\right).
\end{align}
\end{proposition}

\begin{proof}
It suffices to prove \eqref{disfin1}, as \eqref{disfin2} follows from it by scale and shear invariance.
Further it also suffices to assume $t\ge 1.$ For brevity,  let us denote the event that 
$g_{(-1,0;x,t)}$ and $g_{(1,0;y,t)}$ are disjoint as $\mathsf{NC}(x,y)$ ($\mathsf{NC}$ for non-coalescence).
It follows by planarity that $\mathsf{NC}(-M,M)$ implies $\mathsf{NC}(i,i+1)$ for some integer $i$ such that $\lfloor -M \rfloor \leq i \leq \lfloor M\rfloor$.  Thus the following estimate holds:
$$
\P(\mathsf{NC}(-M,M))\le \sum_{i=\lfloor -M \rfloor}^{\lfloor M \rfloor}\P(\mathsf{NC}(i,i+1)).
$$
By the shear invariance of $\mathcal{L}$, the probability $\P(\mathsf{NC}(i,i+1))$ is the same for all $i.$ This probability is $O(t^{-1}\exp (\log^{5/6}(t \vee 1))$ (see \cite[Theorem 1.1]{h20}). 
\end{proof}

The next result says that the passage time from ``infinity" i.e., the Busemann function, is a Brownian motion. For any $x \in \R$ and $T \in \R$, define 
\begin{equation}\label{busemann1}
\mathcal{B}(x,T)= \lim_{n\to \infty} \big( \mathcal{L}(x,T;0,n)-\mathcal{L}(0,0;0,n) \big).
\end{equation}
That the limit exists is shown in various papers, for instance, \cite{gz22}.

\begin{theorem}\label{Theorem: Brownian motion}
Fix $T > 0$.  For each $x\in \R$, let $W(x) = \mathcal{B}(x, T) - \mathcal{B}(0,T)$.  Then $W$ is two-sided Brownian motion with variance parameter $2$.  Furthermore, $W$ is measurable with respect to $\sigma$-field generated by $\L|_{\R\times(T,\infty)}$, and for all $x\in\R$,
\begin{equation}\label{gargmax}
g_{(x,0)}(T)=\argmax_{y\in\R} \{\L(x,0;y,T)+W(y)-W(x)\},
\end{equation}
where the argmax is chosen to be the left-most one if it is not unique.
\end{theorem}

\begin{proof}
The fact that $W$ is two-sided Brownian motion is the $r = 0$ case of Lemma 4.4 in \cite{gz22}.  The result \eqref{gargmax} is the $r = 0$ case of Lemma 4.7 in \cite{gz22}.
\end{proof}

\section{Sampling from the edge of the population}\label{Section: 1-dimension}

This section is devoted to the proof of Theorem \ref{Theorem: main2}. As already mentioned, the proof proceeds through the approximating statements in Propositions \ref{Proposition: expectation, fixed time, 1}, \ref{Proposition: expectation, fixed time, 2}, \ref{Proposition: lln, fixed time, points anywhere}, \ref{Proposition: lln, fixed time, points in [0,n]}, and \ref{Proposition: diff2}.  In Section \ref{sub3.1}, we use scaling properties of the directed landscape to prove Propositions \ref{Proposition: expectation, fixed time, 1} and \ref{Proposition: expectation, fixed time, 2}.  In Section \ref{sub3.2}, we use Birkhoff's ergodic theorem to prove Proposition \ref{Proposition: lln, fixed time, points anywhere}.  In Section \ref{sub3.3}, we show that our results are not affected by small changes in the sampling rate, which will lead to a proof of Proposition \ref{Proposition: lln, fixed time, points in [0,n]}.  Finally, in Section \ref{sub3.4}, we prove Proposition \ref{Proposition: diff2}.

\subsection{Proof of Propositions \ref{Proposition: expectation, fixed time, 1} and \ref{Proposition: expectation, fixed time, 2}}\label{sub3.1}

In this section, we will use a scaling argument to prove Propositions \ref{Proposition: expectation, fixed time, 1} and and \ref{Proposition: expectation, fixed time, 2}.
We start by recording a semigroup property of the geodesics.

\begin{lemma}\label{Lemma: semigroup}
Let $g_{(x,s)}(\cdot)$ be the left-most infinite upward geodesic starting from $(x,s)$. Let $t>s$, and let $y=g_{(x,s)}(t-s)$.
Then $g_{(x,s)}(u+t-s)=g_{(y,t)}(u)$ for all $u\ge 0$. That is, restrictions of left-most infinite upward geodesics are still left-most infinite upward geodesics.
\end{lemma}
\begin{proof}
It is obvious that  $g_{(x,s)}(\cdot+t-s)$ is an infinite upward geodesic starting from $(y,t)$. It remains to show that it is also the left-most one. By Proposition \ref{Proposition: geodesic coalesecent}, the infinite upward geodesics $g_{(x,s)}(\cdot+t-s)$ and $g_{(y,t)}(\cdot)$ will coalesce at some $(y',t')$. Let $\gamma_1(u)=g_{(y,t)}(u)$ for $0\le u\le t-t'$ be the portion of $g_{(y,t)}(\cdot)$ that connects $(y,t)$ and $(y',t')$. Likewise, let $\gamma_2(u)=g_{(x,s)}(u+t-s)$ for $0\le u\le t-t'$ be the portion of $g_{(x,s)}(\cdot)$ that connects $(y,t)$ and $(y',t')$. See Figure \ref{Figure: semigroup} for an illustration. Suppose $\gamma_1(u)\neq \gamma_2(u)$ for some $u\in[0,t-t']$. We claim that we can construct a new infinite upward geodesic starting from $(x,s)$ by switching from $\gamma_2$ to $\gamma_1$, which contradicts the fact that $g_{(x,s)}(\cdot)$ is the left-most infinite upward geodesic.

\begin{figure}[h]
\centering
\begin{tikzpicture}[scale=0.25]
\draw plot [smooth, tension=1] coordinates {(15,0) (12, 2.5) (12, 5) };
\draw [dashed] plot [smooth, tension=1] coordinates {(12, 5) (13, 10) (10, 15) };
\draw plot [smooth, tension=1] coordinates {(12, 5) (16, 10) (10, 15)};
\draw  plot [smooth, tension=1] coordinates {(10, 15) (9, 17.5) (11,20)};

\fill (15,0) circle(8pt); 
\fill (12,5) circle(8pt); 
\fill (10,15) circle(8pt); 
\fill (11,20) circle(8pt); 

\node at (12,0) {$(x,s)$};
\node at (9,5) {$(y,t)$};
\node at (7,15) {$(y',t')$};
\node at (7, 20) {$(y'',t'')$};
\node at (10, 2.5) {$\gamma_0$};
\node at (11, 10) {$\gamma_1$};
\node at (18, 10) {$\gamma_2$};
\node at (10.5, 17.5) {$\gamma_3$};
\end{tikzpicture}
\caption{A geodesic $\gamma_1$ from $(y,t)$ to $(y',t')$ and a geodesic $\gamma_0\oplus\gamma_2\oplus\gamma_3$ from $(x,s)$ to $(y'',t'')$. One can switch from $\gamma_2$ to $\gamma_1$ so that $\gamma_0\oplus\gamma_1\oplus\gamma_3$ is a geodesic from $(x,s)$ to $(y'',t'')$.}
\label{Figure: semigroup}
\end{figure}

To validate the switching from $\gamma_2$ to $\gamma_1$, it suffices to consider the portion of the geodesic from time $s$ to time $t''$, where $t''>t'$, since portions of finite geodesics are also finite geodesics. Let $y''=g_{(x,s)}(t''-s)$. Also, let $\gamma_0$ be the portion of $g_{(x,s)}(\cdot)$ that connects $(x,s)$ and $(y,t)$ and let $\gamma_3$ be the portion of $g_{(x,s)}(\cdot)$ that connects $(y',t')$ and $(y'',t'')$. Because $\gamma_1$ and $\gamma_2$ are restrictions of infinite upward geodesics, we have $||\gamma_1||_\mathcal{L}=||\gamma_2||_\mathcal{L} = \L(y,t;y',t')$. Then, by Proposition \ref{Proposition: concatenation}, 
\begin{align*}
||\gamma_0\oplus \gamma_1\oplus\gamma_3||_\mathcal{L} = ||\gamma_0||_\mathcal{L}+||\gamma_1||_\mathcal{L}+||\gamma_3||_\mathcal{L} = ||\gamma_0||_\mathcal{L}+||\gamma_2||_\mathcal{L}+||\gamma_3||_\mathcal{L} = ||\gamma_0\oplus\gamma_2\oplus\gamma_3||_\mathcal{L}.
\end{align*}
Because $\gamma_0\oplus\gamma_2\oplus\gamma_3$ is the restriction of an infinite upward geodesic, the quantity above is equal to $\L(x,s;y'',t'')$, which shows that $\gamma_0\oplus \gamma_1\oplus\gamma_3$ is a finite geodesic.
\end{proof}

Recall that $\widetilde{\V}_a$ defined in \eqref{Introduction: def T_a, fixed time} consists of the individuals whose set of descendants at time zero has Lebesgue measure greater than $a$.  The following result is an immediate consequence of Lemma \ref{Lemma: semigroup}.

\begin{corollary}\label{Corollary: Vtildenew}
If $(y',t')$ is an ancestor of $(y,t)$, i.e. $g_{(y,t)}(t'-t)=y'$, then $\widetilde{D}_{(y',t')}\supseteq \widetilde{D}_{(y,t)}$. In particular, if $(y,t)\in \widetilde{\V}_{a}$, then $(y',t')\in \widetilde{\V}_{a}$.
\end{corollary}

We now establish a scaling property for $\widetilde{\V}_a$.  Define the scaling map $S_{\rho}: \R^2 \rightarrow \R^2$ by 
\begin{equation}\label{scaling map}
S_\rho(x,s)=(\rho^{-2}x,\rho^{-3}s)\qquad\text{ for all } (x,s)\in\R^2.
\end{equation}
For any $A\subseteq\R^2$, we write $S_{\rho}(A)$ for the image of $A$ under $S_{\rho}$.

\begin{proposition}\label{Proposition: properties of T_a, fixed time}
We have $S_\rho(\widetilde{\V}_{a})\stackrel{d}{=}\widetilde{\V}_{\rho^{-2}a}$.  Also, $\P(\widetilde{\V}_{a}=\emptyset)=0$ for all $a>0$.
\end{proposition}

\begin{proof}
Define $\L_{\rho}$ on $\R^4_{\uparrow}$ by
$$
\L_{\rho}(x,s;y,t)=\rho \L(\rho^{-2}x, \rho^{-3}s; \rho^{-2}y, \rho^{-3}t)=\rho\L(S_{\rho}(x,s);S_{\rho}(y,t)).
$$
Note that $\L$ and $\L_\rho$ are equal in distribution by the scaling property of the directed landscape (part 3 of Proposition \ref{Proposition: properties of the landscape}). Quantities defined with respect to $\L_{\rho}$ will be subscripted by $\rho$. For example, $g_{(x,s),\rho}$ denotes the geodesic starting from $(x,s)$ in $\L_{\rho}$, and $\widetilde{D}_{(y,t),\rho}$ denotes the descendants of $(y,t)$ in $\L_{\rho}$. Note that we have the following equality by definition:
$$
S_{\rho}(g_{(x,s)})=g_{S_{\rho}(x,s),\rho}.
$$
That is, if $g_{(x,s)}$ is an infinite upward geodesic in $\L$, then its image under $S_{\rho}$ is the infinite upward geodesic $g_{S_{\rho}(x,s),\rho}$ in $\L_{\rho}$. It follows that 
$$
S_{\rho}(\widetilde{D}_{(y,t)}) = \widetilde{D}_{S_{\rho}(y,t),\rho}.
$$
Note that $\widetilde{m}(\widetilde{D}_{(y,t)})\ge a$ if and only if $\widetilde{m}(\widetilde{D}_{S_{\rho}(y,t),\rho})=\widetilde{m}(S_{\rho}(\widetilde{D}_{(y,t)}))\ge \rho^{-2} a$. Therefore,
\begin{align*}
\widetilde{\V}_{a} = \{(y,t)\in\R^2: \widetilde{m}(\widetilde{D}_{(y,t)})\ge a\}=\{(y,t)\in\R^2 :\widetilde{m}(\widetilde{D}_{S_{\rho}(y,t),\rho})\ge \rho^{-2}a\}.
\end{align*}
Therefore, 
\begin{align*}
S_{\rho}(\widetilde{\V}_a)=\{S_{\rho}(y,t): \widetilde{m}(\widetilde{D}_{S_{\rho}(y,t),\rho})\ge \rho^{-2} a\}=\{(y,t): \widetilde{m}(\widetilde{D}_{(y,t),\rho})\ge \rho^{-2} a\}=\widetilde{\V}_{\rho^{-2}a,\rho}.
\end{align*}
Because $\L$ and $\L_\rho$ are equal in distribution, it follows that $S_{\rho}(\V_{a})\stackrel{d}{=}\V_{\rho^{-2}a}$.

To prove the second claim, consider the geodesics $g_{(0,0)}$ and $g_{(1,0)}$. By Proposition \ref{Proposition: geodesic coalesecent}, with probability 1, these geodesics coalesce at some $(y,t)\in\R^2$ and there exists a random $\eps>0$ such that $(y,t)\in\widetilde{\V}_{\eps}$. Therefore $\P(\cup_{a>0} \widetilde{\V}_{a} =\emptyset)=\lim_{a\rightarrow0}\P(\widetilde{\V}_{a}=\emptyset)=0$. By the first claim, the probability $\P(\widetilde{\V}_{a}=\emptyset)$ does not depend on $a$, and we have $\P(\widetilde{\V}_{a}=\emptyset)=0$ for all $a>0$. 
\end{proof}

The next lemma bounds the expected number of points in $(0,1]\times\{t\}$ whose set of descendants at time zero has Lebesgue measure greater than 2 but less than 4. Because the integral of this bound over $t$ converges, we will be able to use this bound to show that the constant $C_1'$ in Proposition \ref{Proposition: expectation, fixed time, 1} is finite. The primary tool that we need to prove this bound is the three-arm estimate in Theorem \ref{Theorem: 3-disjoint}.
\begin{lemma}\label{Lemma: C is finite, fixed time}
For all $p\in(1,\infty)$, there exists a constant $C_p$ such that 
$$
\E\big[\#\big((\widetilde{\V}_{2}^t\setminus \widetilde{\V}_{4}^t) \cap (0,1] \big) \big]\le C_pt^{-5/(3p)}\log^{52/p}(t\vee 2).
$$
\end{lemma}

\begin{proof}
Fix $t>0$. For each integer $i$, let $E_i$ be the event that the geodesics $g_{(i-4,0)}$, $g_{(i,0)}$, and $g_{(i+4,0)}$ are disjoint for all $s\in [0, t]$.
For each $y\in\R$, let $I_y$ be the collection of integers $i$ such that the event $E_i$ occurs and $g_{(i,0)}(t)=y$. We claim that $I_y\neq\emptyset$ if $y\in\widetilde{\V}_{2}^t\setminus \widetilde{\V}_{4}^t$. Indeed, since $\widetilde{m}(\widetilde{D}_{(y,t)})\ge 2$ and $\widetilde{D}_{(y,t)}$ is an interval, there exists an integer $i$ such that $g_{(i,0)}(t)=y$. Also, the geodesics $g_{(i-4,0)}$, $g_{(i,0)}$, and $g_{(i+4,0)}$ are disjoint for all $s\in [0, t]$, because otherwise we would have $y\in\widetilde{\V}_{4}^t$.
Let $K\ge 1$ be an integer to be determined later, and let $\widetilde{A}_K=(\widetilde{\V}_{2}^t\setminus \widetilde{\V}_{4}^t)\cap (-K,K]$. It then follows from the argument above that
\begin{align*}
\#\widetilde{A}_K&\le \sum_{y\in\widetilde{A}_K}\#I_y=\sum_{y\in\widetilde{A}_K} \sum_{i\in\Z}\mathbbm{1}_{E_i}\mathbbm{1}_{\{g_{(i,0)}(t)=y\}}\\
&=\sum_{i\in\Z}\sum_{y\in\widetilde{A}_K}\mathbbm{1}_{E_i}\mathbbm{1}_{\{g_{(i,0)}(t)=y\}}=\sum_{i\in\Z}\mathbbm{1}_{E_i}\mathbbm{1}_{\{g_{(i,0)}(t)\in \widetilde{A}_K\}}\le\sum_{i\in\Z}\mathbbm{1}_{E_i}\mathbbm{1}_{\{-K<g_{(i,0)}(t)\le K\}}.
\end{align*} 
Taking expectations, we have
$$
\E[\#\widetilde{A}_K]\le \sum_{i\in \Z} \E[\mathbbm{1}_{E_i} \mathbbm{1}_{\{-K< g_{(i,0)}(t)\le K\}}].
$$
By H\"older's inequality, for any $p,q\in(1,\infty)$ with $1/p+1/q=1$, we have
$$
\E[\#\widetilde{A}_K]\le \sum_{i\in \Z} \P(E_i)^{1/p}\P(-K< g_{(i,0)}(t)\le K)^{1/q}.
$$
By Theorem \ref{Theorem: 3-disjoint}, there exists a constant $C\in(0,\infty)$ such that 
$$
\P(E_i)\le Ct^{-5/3}\log ^{52}(t\vee 2).
$$
Therefore,
\begin{equation}\label{Sampling from a fixed time: C is finite}
\E[\#\widetilde{A}_K]\le C_pt^{-5/(3p)}\log^{52/p}(t\vee 2)\sum_{i\in \Z}\P(-K<g_{(i,0)}(t)\le K)^{1/q}.
\end{equation}  
We bound $\P(-K< g_{(i,0)}(t)\le K)$ by 1 if $|i|< 2K$, and use Theorem \ref{Theorem: transversal fluctuation} if $|i|\ge 2K$ to get
\begin{align*}
&\sum_{i\in\Z}\P(-K< g_{(i,0)}(t)\le K)^{1/q}\\
&\qquad=\sum_{|i|< 2K}\P(-K< g_{(i,0)}(t)\le K)^{1/q}+\sum_{|i|\ge 2K}\P(-K< g_{(i,0)}(t)\le K)^{1/q}\\
&\qquad\le (4K-1)+2\sum_{i\ge2K}\P(|g_{(i,0)}(0)-g_{(i,0)}(t)|\ge i-K)^{1/q}\\
&\qquad=(4K-1)+2\sum_{i\ge K}\P(|g_{(i+K,0)}(0)-g_{(i+K,0)}(t)|\ge i)^{1/q}\\
&\qquad\le (4K-1)+2C_3 \sum_{i\ge K} e^{-C_4 i^3 /(t^2 q)}.
\end{align*}
We now take $K=\lceil t^{2/3}\rceil $. For $i\ge K$, we have $e^{-C_4 i^3 /(t^2 q)}\le e^{-C_4 i /(t^{2/3}q)}$ and hence
\begin{align*}
\sum_{i\in\Z}\P(-K< g_{(i,0)}(t)\le K)^{1/q}\le (4K-1)+2C_3 \sum_{i\ge K} e^{-C_4 i /(t^{2/3}q)}\le(4K-1)+\frac{2C_3}{1-e^{-C_4 /(t^{2/3}q)}}.
\end{align*}
The last expression is bounded above by $C_p t^{2/3}$ for some constant $C_p$ as $t\rightarrow\infty$, and is bounded by a constant $C$ as $t\rightarrow0$. Therefore, there exists a constant $C_p$ such that
$$
\sum_{i\in\Z}\P(-K< g_{(i,0)}(t)\le K)^{1/q}\le C_p\max\{1,t^{2/3}\}.
$$
Plugging this in \eqref{Sampling from a fixed time: C is finite}, we get
$$
\E[\#\widetilde{A}_K]\le C_pt^{-5/(3p)}\log^{52/p}(t\vee 2)\max\{1,t^{2/3}\}.
$$
Recall that $\widetilde{A}_K=(\widetilde{\V}_{2}^t\setminus \widetilde{\V}_{4}^t)\cap (-K,K]$ and that $K=\lceil t^{2/3} \rceil\ge \max\{1,t^{2/3}\}$. It follows, using the translation invariance property of the directed landscape, that
$$
\E\big[\#\big((\widetilde{\V}_{2}^t\setminus \widetilde{\V}_{4}^t) \cap (0,1] \big) \big]=\frac{1}{2K}\E[\#\widetilde{A}_K]\le C_pt^{-5/(3p)}\log^{52/p}(t\vee 2),
$$
which completes the proof.
\end{proof}

We now complete the proof of Proposition \ref{Proposition: expectation, fixed time, 1}. 

\begin{proof}[Proof of Proposition \ref{Proposition: expectation, fixed time, 1}]
Let $A=\widetilde{A}\times\R_+$  be defined as in the statement of the proposition. For any $a>0$, we define a measure $\widetilde{\nu}_a$ on $\R$ by 
$$
\widetilde{\nu}_a(\widetilde{A})=\E\Big[l\big((\widetilde{\V}_0\setminus\widetilde{\V}_a)\cap A\big)\Big].
$$
The countable additivity follows from the linearity of expectation and the monotone convergence theorem. The measure $\widetilde{\nu}_a$ is translation invariant by part 2 of Proposition \ref{Proposition: properties of the landscape}. Therefore, there exists some constant $\eta(a)\in[0,\infty]$ such that $\widetilde{\nu}_a(A)=\eta(a)\widetilde{m}(\widetilde{A})$.

Recalling the definition of $\widetilde{S}_\rho$ from \eqref{scaling map, fixed time} and applying Proposition~\ref{Proposition: properties of T_a, fixed time} with $\rho = \sqrt{a}$, we have
\begin{align*}
\#(\widetilde{\V}^t_{a}\cap \widetilde{A})
=\#\Big(\widetilde{S}_{\sqrt{a}}(\widetilde{\V}_{a}^t\cap \widetilde{A})\Big) =\#\Big(\widetilde{S}_{\sqrt{a}}(\widetilde{\V}^t_{a})\cap \widetilde{S}_{\sqrt{a}}(\widetilde{A})\Big)
\stackrel{d}{=}\#\Big(\widetilde{\V}_{1}^{a^{-3/2}t}\cap \widetilde{S}_{\sqrt{a}}(\widetilde{A})\Big).
\end{align*}
Taking expectations, integrating over $t\ge 0$, and using the change of variable $s=a^{-3/2}t$, we have
\begin{align}\label{eta}
\E\left[l\big((\widetilde{\V}_{0} \setminus \widetilde{\V}_{a})\cap A\big)\right]&=
\int_0^\infty
\E\Big[\#\big((\widetilde{\V}_{0}^t\setminus \widetilde{\V}_{a}^t) \cap \widetilde{A}\big)\Big]\ dt\nonumber\\
&=\int_0^\infty \E\Big[\#\big((\widetilde{\V}_{0}^{a^{-3/2}t}\setminus \widetilde{\V}_{1}^{a^{-3/2}t}) \cap \widetilde{S}_{\sqrt{a}}(\widetilde{A})\big) \Big] \: dt\nonumber\\
&=\int_0^\infty \E\big[\#\Big((\widetilde{\V}_{0}^{s}\setminus \widetilde{\V}_{1}^{s})\cap \widetilde{S}_{\sqrt{a}}(\widetilde{A})\big)\Big]\ a^{3/2} \: ds\nonumber\\
&=\eta(1)\widetilde{m}(\widetilde{S}_{\sqrt{a}}(\widetilde{A}))a^{3/2}\nonumber\\
&=\eta(1) \widetilde{m}(\widetilde{A})a^{1/2}. 
\end{align}
Therefore, we have $\eta(a)=\eta(1)a^{1/2}$ and we can take $C_1'=\eta(1)$. To prove \eqref{cdf2}, it remains to show that $C_1'=\eta(1)\in(0,\infty)$.

We first rule out the possibility that $\eta(1)=0$. Proceeding by contradiction, suppose that $\eta(1)=0$. Then $\eta(a)=\eta(1)a^{1/2}=0$ for all $a>0$. Taking $A=\R^2$ and taking the limit as $a\rightarrow\infty$, we get $\E[l(\widetilde{\V}_0)]=0$. By Corollary \ref{Corollary: Vtildenew}, $\widetilde{\V}_a^s\neq\emptyset$ implies that $\widetilde{\V}_a^t\neq\emptyset$ for all $t\ge s$. Therefore, the fact that $\E[l(\widetilde{\V}_a)]=0$ implies that $\P(\widetilde{\V}_{a}=\emptyset)=1$, which contradicts the second claim in Proposition \ref{Proposition: properties of T_a, fixed time}.

We next rule out the possibility that $\eta(1)=\infty$. Let $A=(0,1]\times\R_+$. By Lemma \ref{Lemma: C is finite, fixed time}, we have
$$
\E\left[l\big((\widetilde{\V}_{2} \setminus \widetilde{\V}_{4})\cap A\big)\right]=
\int_0^\infty
\E\Big[\#\big((\widetilde{\V}_{2}^t\setminus \widetilde{\V}_{4}^t) \cap \widetilde{A}\big)\Big]\ dt=C_5<\infty.
$$
For any $a>0$, the same scaling argument given in \eqref{eta} gives 
$$
\E\left[l\big((\widetilde{\V}_{2a} \setminus \widetilde{\V}_{4a})\cap A\big)\right]=\E\left[l\big((\widetilde{\V}_{2} \setminus \widetilde{\V}_{4})\cap A\big)\right]a^{1/2}=C_5 a^{1/2}.
$$
Therefore,
$$
\eta(1)=\E\left[l\big((\widetilde{\V}_{0} \setminus \widetilde{\V}_{1})\cap A\big)\right]=\sum_{i=-\infty}^0\E\left[l\big((\widetilde{\V}_{2^{i-1}} \setminus \widetilde{\V}_{2^{i}})\cap A\big)\right]=\sum_{i=-\infty}^0 C_5 2^{(i-2)/2}<\infty,
$$
which completes the proof.
\end{proof}

\begin{proof}[Proof of Proposition \ref{Proposition: expectation, fixed time, 2}]
It suffices to consider the case when $\widetilde{m}(\widetilde{A})<\infty$. If $\widetilde{m}(\widetilde{D}_{(y,t)})=a$, then $(y,t)\in Anc_k(\widetilde{\Pi}_\lambda)$ with probability $e^{-\lambda a}(\lambda a)^k/k!$. Using \eqref{cdf2}, if we let 
\begin{equation}\label{Sampling from a fixed time, density}
\phi(a)=\frac{d}{da}\E\Big[l\big((\widetilde{\V}_0\setminus\widetilde{\V}_a)\cap ([0,1]\times\R)\big)\Big] = \frac{C_1' a^{-1/2}}{2} 
\end{equation}
be the ``density" for the length of the branches whose descendants at time zero have Lebesgue measure $a$, then one would expect to integrate over $a$ to get 
\begin{equation}\label{Sampling from a fixed time: integration over a}
\E\Big[l\big(Anc_k(\widetilde{\Pi}_\lambda)\cap A\big)\Big]=\int_0^\infty \frac{e^{-\lambda a}(\lambda a)^k}{k!}\cdot \frac{C_1'a^{-1/2}}{2}\cdot\widetilde{m}(\widetilde{A}) \ da.
\end{equation}
We establish \eqref{Sampling from a fixed time: integration over a} using a discrete approximation. We have
$$
\E\big[l\big(Anc_k(\widetilde{\Pi}_\lambda)\cap A\big)\Big] = \E\left[\int_0^\infty \#\big(Anc_k(\widetilde{\Pi}_\lambda)\cap A\big)^t\ dt\right]=\int_0^\infty \E\Big[\#\big(Anc_k(\widetilde{\Pi}_\lambda)\cap A\big)^t\Big]\ dt.
$$
For each $n\ge 1$, we have
\begin{align*}
\E\Big[\#\big(Anc_k(\widetilde{\Pi}_\lambda)\cap A\big)^t\Big] &=\sum_{i=-\infty}^\infty \E\big[\#\big(Anc_k(\widetilde{\Pi}_\lambda)\cap (\widetilde{\V}_{2^{i/n}}\setminus\widetilde{\V}_{2^{(i+1)/n}})\cap A\big)^t \Big]\\
&\le \sum_{i=-\infty}^\infty \left(\sup_{x\in (2^{i/n}, 2^{(i+1)/n}]} \frac{e^{-\lambda x}(\lambda x)^k}{k!}\cdot\E\Big[\#\big( (\widetilde{\V}_{2^{i/n}}\setminus\widetilde{\V}_{2^{(i+1)/n}})\cap A\big)^t \Big]\right).
\end{align*}
For $a\in(2^{i/n}, 2^{(i+1)/n}]$, we define
$$
f_{n,t}(a) = \sup_{x\in (2^{i/n}, 2^{(i+1)/n}]} \frac{e^{-\lambda x}(\lambda x)^k}{k!} \cdot\frac{\E[\#( (\widetilde{\V}_{2^{i/n}}\setminus\widetilde{\V}_{2^{(i+1)/n}})\cap A)^t ]}{2^{(i+1)/n}-2^{i/n}},
$$
and 
\begin{align*}
f_{n}(a) = \int_0^{\infty} f_{n,t}(a)\ dt= \sup_{x\in (2^{i/n}, 2^{(i+1)/n}]} \frac{e^{-\lambda x}(\lambda x)^k}{k!} \cdot\frac{\E[l( (\widetilde{\V}_{2^{i/n}}\setminus\widetilde{\V}_{2^{(i+1)/n}})\cap A) ]}{2^{(i+1)/n}-2^{i/n}}.
\end{align*}
Then
$$
\E\Big[\#\big(Anc_k(\widetilde{\Pi}_\lambda)\cap A\big)^t\Big] \le\int_0^\infty f_{n,t}(a)\ da,\qquad\E\Big[l\big(Anc_k(\widetilde{\Pi}_\lambda)\cap A\big)\big] \le\int_0^\infty f_{n}(a)\ da.
$$
We now claim that we can apply the dominated convergence theorem to conclude that
\begin{equation}\label{Sampling from a fixed time, discretization}
\E\Big[\#\big(Anc_k(\widetilde{\Pi}_\lambda)\cap A\big)\Big] \le\lim_{n\rightarrow\infty}\int_0^\infty f_n(a)\ da=\int_0^\infty \lim_{n\rightarrow\infty}f_n(a)\ da.
\end{equation}
For the pointwise limit, using \eqref{cdf2} again, we get
\begin{equation}\label{Sampling from a fixed time, pointwise limit}
\lim_{n\rightarrow\infty} f_n(a) = \frac{e^{-\lambda a}(\lambda a)^k}{k!}\widetilde{m}(\widetilde{A})\phi(a).
\end{equation}
To see that $\{f_n\}_{n\ge 2}$ is dominated, note that $\phi(\cdot)$ is decreasing. Suppose $a\in(2^{i/n}, 2^{(i+1)/n}]$. Then
\begin{align*}
f_n(a)&=\sup_{x\in (2^{i/n}, 2^{(i+1)/n}]} \frac{e^{-\lambda x}(\lambda x)^k}{k!} \cdot\frac{\E[l( (\widetilde{\V}_{2^{i/n}}\setminus\widetilde{\V}_{2^{(i+1)/n}})\cap A) ]}{2^{(i+1)/n}-2^{i/n}}\\
&\le\frac{e^{-\lambda 2^{i/n}} (\lambda 2^{(i+1)/n})^k}{k!} \cdot \widetilde{m}(\widetilde{{A}})\phi(2^{i/n})\\
&\le \frac{e^{-\lambda 2^{i/n}} (2^{(i+1)/n})^k \phi(2^{i/n})}{e^{-\lambda 2^{(i+1)/n}} (2^{i/n})^k \phi(2^{(i+1)/n})} \frac{e^{-\lambda a} (\lambda a)^k}{k!}\cdot \widetilde{m}(\widetilde{A}) \phi(a)\\
&\le e^{(\sqrt{2}-1)\lambda a} 2^{k+1/2} \frac{e^{-\lambda a} (\lambda a)^k}{k!}\cdot \frac{C_1'a^{-1/2}}{2}\cdot\widetilde{m}(\widetilde{A}),
\end{align*}
which is integrable. This concludes the proof of \eqref{Sampling from a fixed time, discretization}. Then by \eqref{Sampling from a fixed time, density} and \eqref{Sampling from a fixed time, pointwise limit}, we have
$$
\E\Big[l\big(Anc_k(\widetilde{\Pi}_\lambda)\cap A\big)\Big]\le \int_0^\infty \frac{e^{-\lambda a}(\lambda a)^k}{k!}\cdot \frac{C_1'a^{-1/2}}{2}\cdot\widetilde{m}(\widetilde{A})\ da.
$$
If we replace the supremum by the infimum in the definitions of  $f_{n,t}$ and $f_n$, then a similar argument yields
$$
\E\Big[l\big(Anc_k(\widetilde{\Pi}_\lambda)\cap A\big)\Big]\ge \int_0^\infty \frac{e^{-\lambda a}(\lambda a)^k}{k!}\cdot \frac{C_1'a^{-1/2}}{2}\cdot\widetilde{m}(\widetilde{A}) \ da.
$$
Therefore,
\begin{align*}
\E\Big[l\big(Anc_k(\widetilde{\Pi}_\lambda)\cap A\big)\Big]&=\int_0^\infty \frac{e^{-\lambda a}(\lambda a)^k}{k!}\cdot \frac{C_1'a^{-1/2}}{2}\cdot\widetilde{m}(\widetilde{A}) \ da\\
&= \frac{C_1'\lambda^k\widetilde{m}(\widetilde{A})}{2k!}\int_0^\infty e^{-\lambda a} a^{k-1/2}\ da\\
&=\frac{C_1' \lambda^{-1/2}}{2}\cdot\widetilde{m}(\widetilde{A}) \cdot\frac{\Gamma(k + 1/2)}{k!},
\end{align*}
which completes the proof.
\end{proof}

\subsection{Proof of Proposition \ref{Proposition: lln, fixed time, points anywhere}}\label{sub3.2}

We now establish a strong law of large numbers as stated in Proposition \ref{Proposition: lln, fixed time, points anywhere}. The spatial stationarity of $\L$ and $\widetilde{\Pi}_{\lambda}$ implies the convergence to a possibly random limit. We use a coupling argument based on Proposition \ref{Proposition: coupling} and a variance computation to deduce that the limit is deterministic.

We start with a general fact about identically distributed non-negative random variables.

\begin{lemma}\label{Lemma: probability fact}
Let $\{X_i\}_{i\in I}$ be a collection of identically distributed nonnegative random variables with finite mean. Let $\{I_n\}_{n=1}^\infty$ be a sequence of finite subsets of $I$ and $S_n=\sum_{i\in I_n} X_i$. Suppose that 
$S_n/\#I_n$ converges a.s. and in $L^1$ and that $\liminf_{n\rightarrow\infty} S_n/\#I_n\ge \E[X_1]$ a.s. Then, as $n\rightarrow\infty$, we have $S_n/\#I_n\rightarrow \E[X_1]$ a.s. and in $L^1$.
\end{lemma}

\begin{proof}
By Fatou's lemma, we have 
$$
\E\left[\liminf_{n\rightarrow\infty} S_n/\# I_n\right]\le \liminf_{n\rightarrow\infty}\E[S_n/\# I_n]=\E[X_1].
$$
Combining this with the assumption that $\liminf_{n\rightarrow\infty} S_n/\#I_n\ge \E[X_1]$ a.s., we have 
$$
\liminf_{n\rightarrow\infty} S_n/\#I_n=\E[X_1]\qquad\text{a.s.}.
$$ 
The claim then follows from the assumption that $S_n/\#I_n$ converges a.s. and in $L^1$.
\end{proof}

\begin{proof}[Proof of Proposition \ref{Proposition: lln, fixed time, points anywhere}]
Fix $T>0$. By part 2 of Proposition \ref{Proposition: properties of the landscape}, the random variables 
$$
\Big(l\big(Anc_{k}(\widetilde{\Pi}_{\lambda})\cap([i-1, i)\times[0,T])\big)\Big)_{i\in \Z}
$$ 
are stationary in $i$. Also, equation \eqref{sfs2} implies that
$$
\E\Big[l\big(Anc_{k}(\widetilde{\Pi}_{\lambda})\cap([i-1, i)\times[0,T])\big)\Big]<\infty.
$$ 
It then follows from Birkhoff's ergodic theorem that 
\begin{equation}\label{Sampling from a fixed time, law of large numbers, ergodic}
\lim_{N\rightarrow\infty}\frac{1}{N}\sum_{i=1}^N l\left(Anc_{k}(\widetilde{\Pi}_{\lambda})\cap([i-1, i)\times[0,T])\right)\text{ exists a.s. and in } L^1. 
\end{equation}
We claim that there is there exists a subsequence $(N_m)_{m\ge 1}$ and some constant $c$ such that
\begin{equation}\label{Sampling from a fixed time, law of large numbers, subsequence converges to constant}
\frac{1}{N_m}\sum_{i=1}^{N_m} l\big(Anc_{k}(\widetilde{\Pi}_{\lambda})\cap([i-1, i)\times[0,T])\big)\rightarrow c,\qquad\text{in probability}.
\end{equation}
Suppose for now that \eqref{Sampling from a fixed time, law of large numbers, subsequence converges to constant} holds. Then the limit in \eqref{Sampling from a fixed time, law of large numbers, ergodic} is deterministic, and using the $L^1$ convergence of \eqref{Sampling from a fixed time, law of large numbers, ergodic}, we can conclude that
\begin{equation}\label{Sampling from a fixed time, slln}
\lim_{N\rightarrow\infty}\frac{1}{N}\sum_{i=1}^N l\big(Anc_{k}(\widetilde{\Pi}_{\lambda})\cap([i-1, i)\times[0,T])\big)= \E\big[l\big(Anc_{k}(\widetilde{\Pi}_{\lambda})\cap([0, 1)\times[0,T])\big)\big],\ \text{ a.s. and in } L^1.  
\end{equation}
We then apply Lemma \ref{Lemma: probability fact} with $X_i=l(Anc_{k}(\widetilde{\Pi}_{\lambda})\cap([i-1, i)\times\R_+))$ and $I_n=\{1,2,\dots, n\}$. Note that $\E[X_1]<\infty$ by \eqref{sfs2}. For each $T>0$, let $X_{i,T}=l(Anc_{k}(\widetilde{\Pi}_{\lambda})\cap([i-1, i)\times[0,T]))$ and $S_{n,T}=\sum_{i=1}^n X_{i,T}$. Using \eqref{Sampling from a fixed time, slln}, we have
$$
\liminf_{n\rightarrow\infty}\frac{S_n}{n}\ge \sup_{T>0}\left(\lim_{n\rightarrow\infty} \frac{S_{n,T}}{n}\right)=\sup_{T>0}\E[X_{1,T}]=\E[X_1].
$$
Also, by Birkhoff's ergodic theorem, $S_{n}/n$ converges a.s. and in $L^1$. It then follows from Lemma ~\ref{Lemma: probability fact} that $S_{n}/n$ converges to $\E[X_1]$ a.s and in $L^1$. Proposition \ref{Proposition: lln, fixed time, points anywhere} then follows from \eqref{sfs2}.

It remains to show \eqref{Sampling from a fixed time, law of large numbers, subsequence converges to constant}. To prove \eqref{Sampling from a fixed time, law of large numbers, subsequence converges to constant}, we will make use of the approximate independence of the directed landscape restricted to disjoint boxes. For any $\eps>0$, let $K_{i,\eps} =[i-1/4,i+1/4]\times[0,\eps]$ be defined as in Proposition \ref{Proposition: coupling}. Also, for any $n\in \N$, let $\L_0,\L_1,\dots,\L_n$ be coupled as described in Proposition \ref{Proposition: coupling}. We can choose $\eps_n>0$ which goes to 0 as $n\rightarrow\infty$ such that
$$
\lim_{n\rightarrow\infty}\P(\L_0|_{K_{i,\eps_n}} =\L_i|_{K_{i,\eps_n}}\text{ for all }i=1,\dots, n)=1.
$$
Recall the scaling map $S_{\rho}$ from \eqref{scaling map}. Let $\rho_n=(\eps_n/T)^{1/3}$, $a_{n,i}=\rho_n^{-2}(i-1/4)$ and $b_{n,i}=\rho_n^{-2}(i+1/4)$ so that the box $B_{n,i}:=S_{\rho_n}(K_{i,\eps_n})=[a_{n,i}, b_{n,i}]\times [0,T]$ has height $T$. By choosing a smaller $\eps_n$ if necessary, we may assume that $\rho_n^{-1}$ is an integer. By the KPZ scaling property in Proposition \ref{Proposition: properties of the landscape}, we may choose $\L_{n,1},\dots, \L_{n,n}$ to be independent, and coupled with $\L$ such that
$$
\lim_{n\rightarrow\infty}\P(\L|_{B_{n,i}} =\L_{n,i}|_{B_{n,i}}\text{ for all }i=1,\dots, n)=1.
$$
Let $W$ be the two sided Brownian motion given by Theorem \ref{Theorem: Brownian motion}. For each box $B_{n,i}$ and each $x\in[a_{n,i},b_{n,i}]$, we define
$$
g^{B_{n,i}}_{(x,0)}(T)=\argmax_{y\in[a_{n,i},b_{n,i}]} \{\L_{n,i}(x,0; y,T)+W(y)-W(x)\},
$$
and for $0< t< T$, we define
$$
g^{B_{n,i}}_{(x,0)}(t)=\argmax_{y\in[a_{n,i},b_{n,i}]} \{\L_{n,i}(x,0; y,t)+\L_{n,i}(y,t; g^{B_{n,i}}_{(x,0)}(T),T)\},
$$
where the argmax is chosen to be the left-most one if it is not unique. We think of $g_{(x,0)}^{B_{n,i}}(\cdot)$ as geodesic starting from $(x,0)$ associated with the box $B_{n,i}$.

Since $W$ is independent of $\{\L_{n,i}\}_{i=1}^{n}$ and has independent increments, the geodesics associated with $B_{n,i}$ and $B_{n,j}$ are independent if $i$ and $j$ are distinct. For each $(y,t)\in B_{n,i}$, each integer $k\ge 1$, and each $a>0$, we also define the analogs of $\widetilde{D}_{(y,t)}$, $Anc_{k}(\widetilde{\Pi}_{\lambda})$, and $\widetilde{\V}_{a}$ as
\begin{align*}
\widetilde{D}_{(y,t)}^{B_{n,i}} &:= \{x\in[a_{n,i},b_{n,i}]: g_{(x,0)}^{B_{n,i}}(t)=y\},\\
Anc_{k}(\widetilde{\Pi}_{\lambda})^{B_{n,i}}&:= \{(y,t)\in {B_{n,i}}: \#(\widetilde{D}_{(y,t)}^{B_{n,i}}\cap \widetilde{\Pi}_\lambda)= k\},\\
\widetilde{\V}_{a}^{B_{n,i}} &:= \{(y,t)\in B_{n,i}: \widetilde{m}(\widetilde{D}_{(y,t)}^{B_{n,i}})> a\}.
\end{align*}

Let $B'_{n,i}:=[(2a_{n,i}+b_{n,i})/3,(a_{n,i}+2b_{n,i})/3)\times[0,T]$ be a box contained in $B_{n,i}$. Let $E_{n,i}$ be the event that the following holds; see Figure \ref{Figure} for an illustration. 
\begin{enumerate}
\item $\L|_{B_{n,i}} =\L_{n,i}|_{B_{n,i}},$
\item $\sup_{t\in[0,T]} |g_{((5a_{n,i}+b_{n,i})/6,0)}(t)-(5a_{n,i}+b_{n,i})/6|< (b_{n,i}-a_{n,i})/6,$
\item $\sup_{t\in[0,T]}|g_{((a_{n,i}+5b_{n,i})/6,0)}(t)-(a_{n,i}+5b_{n,i})/6|< (b_{n,i}-a_{n,i})/6.$
\end{enumerate}
Finally, we write $E_n:=\cap_{i=1}^n E_{n,i}$. Recall that $\P(\L|_{B_{n,i}} =\L_{n,i}|_{B_{n,i}}\text{ for all }i=1,\dots, n)\rightarrow1$ as $n\rightarrow\infty$. Also, by Theorem \ref{Theorem: transversal fluctuation}, we have
\begin{align*}
&\P\left(\sup_{t\in[0,T]}\left|g_{((a_{n,i}+5b_{n,i})/6,0)}(t)-\frac{a_{n,i}+5b_{n,i}}{6}\right|< \frac{b_{n,i}-a_{n,i}}{6}\right)\ge 1-C_3e^{-C_4/(12^3\eps_n^2)},
\end{align*}
and the same bound holds for $g_{((5a_{n,i}+b_{n,i})/6,0)}$. Therefore, $\P(E_n)\rightarrow1$ because $\eps_n\rightarrow0$.

\newcommand{\Emmett}[5]{
\draw (#1,#2)
\foreach \x in {1,...,#3}
{   -- ++(rand*#4,#5)
};
}
\pgfmathsetseed{1336}

\begin{figure}[h]
\centering
\begin{tikzpicture}[scale=0.3]
\draw[draw=white, fill=gray!30] (16,0) rectangle ++(16,15);
\draw[black] (16,0) -- (16,15);
\draw[black, dashed] (32,0) -- (32,15);
\draw[draw=black] (0,0) rectangle ++(48,15);
\Emmett{8}{0}{1500}{0.2}{0.01}
\Emmett{40}{0}{1500}{0.2}{0.01}
\node at (0,-1) {$a_{n,i}$};
\node at (16,-1) {$(2a_{n,i}+b_{n,i})/3$};
\node at (32,-1) {$(a_{n,i}+2b_{n,i})/3$};
\node at (48,-1) {$b_{n,i}$};
\node at (24, 15/2) {$B'_{n,i}$};
\node at (50, 15/2) {$B_{n,i}$};
\node at (5,10) {$g_{((5a_{n,i}+b_{n,i})/6,0)}$};
\node at (37,5) {$g_{((a_{n,i}+5b_{n,i})/6,0)}$};
\end{tikzpicture}
\caption{An illustration of the event $E_{n,i}$. Here $B_{n,i}$ is the whole box and $B'_{n,i}$ is the shaded box in the middle. On the event $E_{n,i}$, the geodesics $g_{((a_{n,i}+5b_{n,i})/6,0)}$ and $g_{((5a_{n,i}+b_{n,i})/6,0)}$ stay in the white boxes.}
\label{Figure}
\end{figure}

The ordering of the geodesics in Proposition \ref{Proposition: ordering of geodesics} implies that, on the  event $E_{n,i}$, for any $(y,t)\in B'_{n,i}$, we have $\widetilde{D}_{(y,t)}\subseteq[(5a_{n,i}+b_{n,i})/6, (a_{n,i}+5b_{n,i})/6]$. Also, for any $x\in[(5a_{n,i}+b_{n,i})/6,(a_{n,i}+5b_{n,i})/6]$, the geodesic $g_{(x,0)}$ lies in the box $B_{n,i}$ for $t\in[0,T]$. Therefore, using Theorem \ref{Theorem: Brownian motion}, on $E_{n,i}$, for all $x\in[(5a_{n,i}+b_{n,i})/6,(a_{n,i}+5b_{n,i})/6]$, we have
\begin{align*}
g_{(x,0)}(T)&=\argmax_{y\in\R} \{\L(x,0; y,T)+W(y)-W(x)\}\\
 &=\argmax_{y\in[a_{n,i}, b_{n,i}]} \{\L(x,0; y,T)+W(y)-W(x)\}\\
 &=\argmax_{y\in[a_{n,i}, b_{n,i}]} \{\L_{n,i}(x,0; y,T)+W(y)-W(x)\}=g_{(x,0)}^{B_{n,i}}(T),
\end{align*}
and $g_{(x,0)}(t)=g_{(x,0)}^{B_{n,i}}(t)$ for all $0\le t\le T$. It follows that
\begin{equation}\label{Sampling from a fixed time, law of large numbers, coupling difference}
\lim_{n\rightarrow\infty}\P\left(\sum_{i=1}^n l\big(Anc_{k}(\widetilde{\Pi}_{\lambda})\cap B'_{n,i}\big)\neq\sum_{i=1}^n l\big(Anc_{k}(\widetilde{\Pi}_{\lambda})^{B_{n,i}}\cap B'_{n,i}\big)\right)=0.  
\end{equation}

For $n\ge 1$ and $1\le i\le n$, let $X_{n,i}=l(Anc_{k}(\widetilde{\Pi}_{\lambda})^{B_{n,i}}\cap{B'_{n,i}})/(b_{n,i}-a_{n,i})$. Then for each fixed $n$, the random variables $\{X_{n,i}\}_{i=1}^n$ are independent. We now show that
\begin{equation}\label{Sampling from a fixed time, law of large numbers, SLLN for triangular array}
\lim_{n\rightarrow\infty}\frac{1}{n}\sum_{i=1}^n (X_{n,i}-\E[X_{n,i}])= 0\qquad\text{ in } L^2.
\end{equation}
Note that one can upper bound $l(Anc_{k}(\widetilde{\Pi}_{\lambda})^{B_{n,i}}\cap{B'_{n,i}})$ by considering the worst case where the geodesics starting from points in $\widetilde{\Pi}_{\lambda}\cap[a_{n,i}, b_{n,i}]$ immediately coalesce into groups of size $k$. We get
$$
\#\{x\in[a_{n,i}, b_{n,i}]: (x,t)\in Anc_{k}(\widetilde{\Pi}_{\lambda})^{B_{n,i}}\}\le \frac{\#(\widetilde{\Pi}_{\lambda}\cap[a_{n,i}, b_{n,i}])}{k},
$$ 
and hence
$$
X_{n,i}=\frac{l(Anc_k(\widetilde{\Pi}_{\lambda})^{B_{n,i}}\cap{B'_{n,i}})}{b_{n,i}-a_{n,i}}\le\frac{T\#(\widetilde{\Pi}_{\lambda}\cap[a_{n,i}, b_{n,i}])}{k(b_{n,i}-a_{n,i})}.
$$
Note also that $b_{n,i}-a_{n,i}=1/(2\rho_n^2)\rightarrow\infty$ as $n\rightarrow\infty$. It follows that 
$$
\E[X_{n,i}^2]\le \E\left[\left(\frac{T\#(\widetilde{\Pi}_{\lambda}\cap[a_{n,i}, b_{n,i}])}{k(b_{n,i}-a_{n,i})}\right)^2\right]=\frac{T^2}{k^2}\cdot\frac{\lambda^2(b_{n,i}-a_{n,i})^2+\lambda(b_{n,i}-a_{n,i})}{(b_{n,i}-a_{n,i})^2},
$$
which is uniformly bounded in $n$. Therefore, as $n\rightarrow\infty$
$$
\E\left[\left(\frac{1}{n}\sum_{i=1}^n (X_{n,i}-\E[X_{n,i}])\right)^2\right]\le \frac{\E[X_{n,1}^2]}{n} \rightarrow0,
$$
which completes the proof of \eqref{Sampling from a fixed time, law of large numbers, SLLN for triangular array}.
Also $\E[X_{n,i}]\le \lambda T/k$ is uniformly bounded in $n$. Therefore we may choose a subsequence $(n_m)_{m\ge 0}$ such that 
$$
\frac{1}{n_m}\sum_{i=1}^{n_m} X_{n_m,i}=\frac{1}{n_m(b_{n_m,i}-a_{n_m,i})}\sum_{i=1}^{n_m} l(Anc_k(\widetilde{\Pi}_{\lambda})^{B_{n_m,i}}\cap{B'_{n_m,i}})
$$
converges to a constant in probability. Combining this with \eqref{Sampling from a fixed time, law of large numbers, coupling difference}, we conclude that there is a subsequence $(n_m)_{m\ge 1}$ such that
$$
\frac{1}{n_m(b_{n_m,i}-a_{n_m,i})}\sum_{i=1}^{n_m} l(Anc_{k}(\widetilde{\Pi}_\lambda)\cap{B'_{n_m,i}})
$$
converges to some constant in probability. Recall that 
$$B'_{n,i}=[(2a_{n,i}+b_{n,i})/3,(a_{n,i}+2b_{n,i})/3)\times[0,T]=[\rho_n^{-2}(i-1/12),\rho_n^{-2}(i+1/12))\times[0,T].$$ 
Applying the same argument where we replace $[i-1/12, i+1/12)$ by $[i-1,i-5/6),\dots,[i-1/6,i)$ and choosing another subsequence which we still write as $(n_m)_{m\ge 1}$, we conclude that 
$$
\frac{1}{n_m\rho_{n_m}^{-2}}l(Anc_{k}(\widetilde{\Pi}_\lambda)\cap[0,n_m\rho_{n_m}^{-2}])
$$
converges to a constant in probability. Since $\rho_n^{-1}$ was chosen to be an integer, this concludes the proof of \eqref{Sampling from a fixed time, law of large numbers, subsequence converges to constant}.
\end{proof}

\subsection{Proof of Proposition \ref{Proposition: lln, fixed time, points in [0,n]}}\label{sub3.3}

In this section, we prove Proposition \ref{Proposition: lln, fixed time, points in [0,n]}, which establishes that the convergence proved in Proposition \ref{Proposition: lln, fixed time, points anywhere} still holds in probability when the sampling parameter is chosen so that exactly $n$ points are sampled from $[0,n]$.  As noted in the introduction, convergence in probability is more natural in this setting and is sufficient for the proof of Theorem \ref{Theorem: main2}.  We first prove a lemma which states that $Anc_k(\widetilde{\Pi}_{\lambda})$ does not change much if we change the sampling intensity $\lambda$ by a small amount. 

\begin{lemma}\label{Lemma: perturb the sampling intensity2}
Fix $\lambda>\eps>0$. Let $A=\widetilde{A}\times \R_+$ where $\widetilde{A}$ is a measurable subset of $\R$. Let
$$
B=\bigcup_{1\le i\le k\le j, i\neq j} \left(Anc_{i}(\widetilde{\Pi}_{\lambda-\eps})\cap Anc_{j}(\widetilde{\Pi}_{\lambda+\eps})\right)
$$
be the pairs $(y,t)\in\R^2$ that are ancestral to at most $k$ points in $\widetilde{\Pi}_{\lambda-\eps}$, at least $k$ points in $\widetilde{\Pi}_{\lambda+\eps}$, and at least one point in $\widetilde{\Pi}_{\lambda+\eps}\setminus\widetilde{\Pi}_{\lambda-\eps}$. Then there exists a constant $C_{k,\lambda}$ depending on $k$ and $\lambda$ such that if $\eps<\lambda/2$, then
$$
\E[l(A\cap B)]\le C_{k,\lambda}\widetilde{m}(\widetilde{A})\eps.
$$  
\end{lemma}

\begin{proof}
For each $i\ge 1$, let
$$
B_i = Anc_{i}(\widetilde{\Pi}_{\lambda-\eps}) \cap\{(y,t)\in\R^2: \widetilde{D}_{(y,t)}\cap (\widetilde{\Pi}_{\lambda+\eps}\setminus\widetilde{\Pi}_{\lambda-\eps})\neq\emptyset\}
$$
be the pairs $(y,t)\in\R^2$ that are ancestral to $i$ points in $\widetilde{\Pi}_{\lambda-\eps}$ and at least one point in $\widetilde{\Pi}_{\lambda+\eps}\setminus\widetilde{\Pi}_{\lambda-\eps}$.
Then, $B\subseteq \cup_{1\le i\le k} B_i$ by definition. Therefore
$$
\E[l(A\cap B)]\le \sum_{1\le i\le k}\E[l(A\cap B_i)].
$$
Let $\phi(\cdot)$ be defined as in \eqref{Sampling from a fixed time, density}. If $\widetilde{m}(\widetilde{D}_{(y,t)})=a$, then the probability that $(y,t)$ is ancestral to $i$ points in $\widetilde{\Pi}_{\lambda-\eps}$ and at least one point in $\widetilde{\Pi}_{\lambda+\eps}\setminus\widetilde{\Pi}_{\lambda-\eps}$ is
$$
\frac{e^{-(\lambda-\eps)a}(\lambda-\eps)^i a^i }{i!}(1-e^{-2\eps a}).
$$
Therefore, using that $1 - e^{-2 \eps a} \leq 2 \eps a$, 
\begin{align}\label{ElAB}
\E[l(A\cap B)]&\le\sum_{1\le i\le k}
\int_0^\infty \widetilde{m}(\widetilde{A})\phi(a) \cdot \frac{e^{-(\lambda-\eps)a}(\lambda-\eps)^i a^i }{i!}(1-e^{-2\eps a})\ da \nonumber \\
&\leq  \sum_{1 \leq i \leq k} \int_0^\infty \widetilde{m}(\widetilde{A}) \cdot \frac{C_1' a^{-1/2}}{2} \cdot \frac{e^{-(\lambda-\eps)a}(\lambda-\eps)^i a^i }{i!} \cdot 2 \eps a \ da \nonumber \\
&\le C_{k,\lambda}\widetilde{m}(\widetilde{A})\eps
\end{align}
because the integral is finite, which completes the proof.
\end{proof}

\begin{proof}[Proof of Proposition \ref{Proposition: lln, fixed time, points in [0,n]}]

For $n$ sufficiently large, we apply the result of Lemma \ref{Lemma: perturb the sampling intensity2} with $\lambda=1$, $\eps_n = (\log n)/\sqrt{n}$ and $A_n=[0,n]\times\R_+$ to get
$$
\E[l(A_n\cap B)]\le C_{k,\lambda}\sqrt{n}\log n.
$$
In particular, as $n\rightarrow\infty$, we have
$$
\frac{l(A_n\cap B)}{n}\rightarrow 0,\qquad\text{in probability}.
$$
Recall the definition of $\widetilde{\lambda}_n^*$ from \eqref{lambda2}. Then, as $n\rightarrow\infty$,
\begin{equation}\label{lambda star}
\P(\widetilde{\lambda}_n^*\in [1-\eps_n,1+\eps_n])\ge \P(\#(\widetilde{\Pi}_{1-\eps_n}\cap A_n)< n<\#(\widetilde{\Pi}_{1+\eps_n}\cap A_n))\rightarrow1.    
\end{equation}
On the event $\widetilde{\lambda}_n^*\in [1-\eps_n,1+\eps_n]$, we have
$$
Anc_k(\widetilde{\Pi}_{1})\ \triangle\  Anc_k(\widetilde{\Pi}_{\widetilde{\lambda}_n^*})\subseteq B.
$$
It follows that as $n\rightarrow\infty$, 
\begin{equation}\label{Sampling from a fixed time, comparison}
\frac{l\Big(\big(Anc_k(\widetilde{\Pi}_{1})\ \triangle\  Anc_k(\widetilde{\Pi}_{\widetilde{\lambda}_n^*})\big)\cap A_n\Big)}{n}\rightarrow 0,\qquad\text{in probability}.   
\end{equation}
The claim of the proposition then follows from \eqref{Sampling from a fixed time, comparison} and Proposition \ref{Proposition: lln, fixed time, points anywhere}.
\end{proof}

\begin{remark}\label{Remark}
The same reasoning applies if we take $A_n=[0,n]\times[0,T]$ for a fixed $T>0$. Using \eqref{Sampling from a fixed time, slln}, we conclude that as $n\rightarrow\infty$,
$$
\frac{l\big(Anc_k(\widetilde{\Pi}_{\widetilde{\lambda}_n^*})\cap([0,n]\times[0,T])\big)}{n}\rightarrow\E\Big[l\big(Anc_{k}(\widetilde{\Pi}_1)\cap([0, 1)\times[0,T])\big)\Big]\qquad\text{in probability}.
$$
\end{remark}

\subsection{Proof of Proposition \ref{Proposition: diff2}}\label{sub3.4}

To obtain Proposition \ref{Proposition: diff2} from Proposition \ref{Proposition: lln, fixed time, points in [0,n]}, we will derive lower and upper bounds for $l(Anc_k(\widetilde{\Pi}_{\widetilde{\lambda}_n^*}\cap A_n))$ in terms of $l(Anc_k(\widetilde{\Pi}_{\widetilde{\lambda}_n^*})\cap A_n)$. 
The following lemma bounds the length of the set of points that are ancestors of $k$ points in $[0,n]$ and additional sampled points outside of the interval.

\begin{lemma}\label{Lemma: Height of the Anc_k1}
Let $\eps_n=(\log n)/\sqrt{n}$. Let $A_n$ and $\widetilde{\lambda}_n^*$ be defined as in Proposition \ref{Proposition: lln, fixed time, points in [0,n]}. Then, as $n\rightarrow\infty$, we have
$$
\frac{l\big(Anc_k(\widetilde{\Pi}_{\widetilde{\lambda}_n^*}\cap A_n)\setminus Anc_k(\widetilde{\Pi}_{\widetilde{\lambda}_n^*})\big)}{n}\rightarrow 0\qquad \text{ in probability}.
$$
\end{lemma}

\begin{proof}
For any $\lambda>0$, let $\{(x_{i,\lambda},0)\}_{i\in\Z}$ be the sampled points in $\widetilde{\Pi}_{\lambda}$, ordered so that $$
\dots<x_{0,\lambda}<0\le x_{1,\lambda}<x_{2,\lambda}<\dots<x_{N_{\lambda}, \lambda}\le n<x_{N_{\lambda}+1,\lambda}<\dots.
$$ 
Then there are $N_{\lambda}$ points sampled from $[0,n]$. Note that the choice of $\widetilde{\lambda}_n^*$ guarantees that $N_{\widetilde{\lambda}_n^*}=n$. For ease of notation, we write $x_i=x_{i,\widetilde{\lambda}_n^*}$. 

Because $\P(\widetilde{\lambda}_n^*\in[1-\eps_n,1+\eps_n])\rightarrow1$ as $n\rightarrow\infty$ by \eqref{lambda star}, we may assume that $\widetilde{\lambda}_n^*\in[1-\eps_n,1+\eps_n]$ holds in the rest of the argument. If $(y,t)\in Anc_k(\widetilde{\Pi}_{\widetilde{\lambda}_n^*} \cap A_n)\setminus Anc_k(\widetilde{\Pi}_{\widetilde{\lambda}_n^*})$, then there exist $i\in\{1,\dots, n\}$ and $j\notin\{1,\dots, n\}$ such that both $(x_i,0)$ and $(x_j,0)$ are descendants of $(y,t)$. By the ordering of geodesics in Proposition \ref{Proposition: ordering of geodesics}, we deduce that $(x_0,0)\in \widetilde{D}_{(y,t)}$ or $(x_{n+1},0)\in\widetilde{D}_{(y,t)}$. Therefore, for any $t>0$, we have
$$
\#\Big\{y\in\R: (y,t)\in Anc_k(\widetilde{\Pi}_{\widetilde{\lambda}_n^*} \cap A_n)\setminus Anc_k(\widetilde{\Pi}_{\widetilde{\lambda}_n^*})\Big\}\le 2.
$$
Moreover, if $(x_0,0)\in \widetilde{D}_{(y,t)}$ with $(y,t)\in Anc_k(\widetilde{\Pi}_{\widetilde{\lambda}_n^*}\cap A_n)$, then using the ordering of geodesics again, we have $\widetilde{D}_{(y,t)}\cap\widetilde{\Pi}_{\widetilde{\lambda}_n^*} \cap A_n=\{(x_i,0)\}_{i=1}^k$. In this case, the time $t$ is bounded above by the coalescence time of the geodesics $g_{(x_1,0)}$ and $g_{(x_{k+1},0)}$ because otherwise we would have $\{(x_i,0)\}_{i=1}^{k+1}\subseteq \widetilde{D}_{(y,t)}$, contradicting that $(y,t)\in Anc_k(\widetilde{\Pi}_{\widetilde{\lambda}_n^*}\cap A_n)$. Also, if $n$ is sufficiently large and $\widetilde{\lambda}_n^*\in[1-\eps_n,1+\eps_n]$, then
$$
x_{0,1/2}\le x_{0,1-\eps_n}< x_1<x_{k+1}\le x_{k+1,1-\eps_n}\le x_{k+1,1/2}.
$$
Therefore, by the ordering of geodesics, the time $t$ is bounded from above by the coalescence time of $g_{(x_{0,1/2},0)}$ and $g_{(x_{k+1,1/2}, 0)}$ which we call $\xi_1$.
Similarly, if $(x_n,0)\in\widetilde{D}_{(y,t)}$, $n$ is sufficiently large, and $\widetilde{\lambda}_n^*\in[1-\eps_n,1+\eps_n]$, then the time $t$ is bounded from above by the coalescence time of $g_{(x_{N_{1/2}-k,1/2},0)}$ and $g_{(x_{N_{1/2}+1,1/2},0)}$ which we call $\xi_2$. Therefore, 
$$
l\left(Anc_k(\widetilde{\Pi}_{\widetilde{\lambda}_n^*}\cap A_n)\setminus Anc_k(\widetilde{\Pi}_{\widetilde{\lambda}_n^*})\right)
$$
is bounded above by $\xi_1+\xi_2$. Because this is a tight random variable, the lemma follows.
\end{proof}

The next lemma bounds the probability that there are points at height greater than $t$ that are ancestral to exactly $k$ sampled points, all of which are in $[0,n]$. In the proof of Lemma \ref{Lemma: Height of the Anc_k2}, we will again use the three-arm estimate given in Theorem \ref{Theorem: 3-disjoint}. Note that by combining Theorem \ref{Theorem: 3-disjoint} with the scaling property in part 3 of Proposition \ref{Proposition: properties of the landscape}, we can see that for $M\ge 1$, the probability that the geodesics $g_{(-M,0)}$, $g_{(0,0)}$ and $g_{(-M,0)}$ remain disjoint until time $t$ is bounded above by $CM^{5/2}t^{-5/2}\log^{52}(t\vee 2)$.

\begin{lemma}\label{Lemma: Height of the Anc_k2}
Let $\eps_n=(\log n)/\sqrt{n}$. Let $A_n$ and $\widetilde{\lambda}_n^*$ be defined as in Proposition \ref{Proposition: lln, fixed time, points in [0,n]}. Then there exists a constant $C_{k}$ depending only on $k$ such that for all $t\ge 2$, we have
\begin{align*}
&\P\Big(Anc_k(\widetilde{\Pi}_{\widetilde{\lambda}_n^*}\cap A_n)\cap Anc_k(\widetilde{\Pi}_{\widetilde{\lambda}_n^*})\cap (\R\times[t,\infty))\neq \emptyset\Big)\\
&\qquad\qquad\le C_{k}nt^{-5/3}\log^{52}t + \P\Big(\widetilde{\lambda}_n^{*}\notin[1-\eps_n,1+\eps_n]\Big).    
\end{align*}
\end{lemma}

\begin{proof}
If $(y,t)\in Anc_k(\widetilde{\Pi}_{\widetilde{\lambda}_n^*}\cap A_n)\cap Anc_k(\widetilde{\Pi}_{\widetilde{\lambda}_n^*})$, then there exists $i\in\{1,\dots, n-k+1\}$ such that
\begin{equation*}
D_{(y,t)}\cap\widetilde{\Pi}_{\widetilde{\lambda}_n^*}=D_{(y,t)}\cap\widetilde{\Pi}_{\widetilde{\lambda}_n^*}\cap A_n=\{(x_i,0),\dots, (x_{i+k-1},0)\}. 
\end{equation*}
Then, before time $t$, the geodesics  $g_{(x_{i},0)}$ and $g_{(x_{i+k-1},0)}$ coalesce, and neither $g_{(x_{i-1},0)}$ and $g_{(x_{i},0)}$ nor $g_{(x_{i+k-1},0)}$ and $g_{(x_{i+k},0)}$ coalesce. In particular, the geodesics $g_{(x_{i-1},0)}$, $g_{(x_{i},0)}$, and $g_{(x_{i+k},0)}$ are disjoint until time $t$. Note also that the probability that these geodesics are disjoint is monotone in the gaps $x_{i}-x_{i-1}$ and $x_{i+k}-x_{i}$. Conditioning on $(\widetilde{\Pi}_{\lambda})_{\lambda>0}$ and applying Theorem \ref{Theorem: 3-disjoint} with $M=\max\{x_i-x_{i-1}, x_{i+k}-x_i,1\}$ for $i=1,2,\dots, n-k+1$, we have
\begin{align*}
&\P\Big(Anc_k(\widetilde{\Pi}_{\widetilde{\lambda}_n^*}\cap A_n)\cap Anc_k(\widetilde{\Pi}_{\widetilde{\lambda}_n^*})\cap (\R\times[t,\infty))\neq \emptyset\Big|\widetilde{\Pi}\Big)\\
&\qquad\le Ct^{-5/3}\log^{52} (t\vee 2) \sum_{i=1}^{n-k+1}\max\{x_i-x_{i-1}, x_{i+k}-x_i,1\}^{5/2}.
\end{align*}
Taking expectations, we have
\begin{align}\label{height Anc_k}
&\P\Big(Anc_k(\widetilde{\Pi}_{\widetilde{\lambda}_n^*}\cap A_n)\cap Anc_k(\widetilde{\Pi}_{\widetilde{\lambda}_n^*})\cap (\R\times[t,\infty))\neq \emptyset\Big)\nonumber\\
&\qquad\le Ct^{-5/3}\log^{52} (t\vee 2) \sum_{i=1}^{n-k+1}\E\Big[\max\{x_i-x_{i-1}, x_{i+k}-x_i,1\}^{5/2}\mathbbm{1}_{\{\widetilde{\lambda}_n^*\in[1-\eps_n,1+\eps_n]\}}\Big]\nonumber\\
&\qquad\qquad+\P\Big(\widetilde{\lambda}_n^*\in[1-\eps_n,1+\eps_n]\Big).
\end{align}
For $1\le i<i'\le n$, the random variable $(x_{i'}-x_{i})/n$ has a $Beta(i'-i, n+1-(i'-i))$ distribution (see Section 4.1 of \cite{p65}). Therefore, for $2\le i\le n-k$, we have
$$
\E\Big[\max\{x_i-x_{i-1}, x_{i+k}-x_i,1\}^{5/2}\mathbbm{1}_{\{\widetilde{\lambda}_n^*\in[1-\eps_n,1+\eps_n]\}}\Big]\le \E\big[(x_{i+k}-x_{i-1}+1)^{5/2}\big]\le C_k.
$$
Also, if $n$ is sufficiently large and $\widetilde{\lambda}_n^*\in[1-\eps_n,1+\eps_n]$, then $x_{0,1/2}\le x_0< 0$ and $n<x_{n+1}\le x_{N_{1/2}+1,1/2}$. Then for $i=1$, we have
\begin{align*}
\E\Big[\max\{x_1-x_{0}, x_{k+1}-x_1,1\}^{5/2}\mathbbm{1}_{\{\widetilde{\lambda}_n^*\in[1-\eps_n,1+\eps_n]\}}\Big]
&\le \E\Big[(x_{k+1}-x_0+1)^{5/2}\mathbbm{1}_{\{\widetilde{\lambda}_n^*\in[1-\eps_n,1+\eps_n]\}}\Big]\\
&\le \E[(x_{k+1}-x_{0,1/2}+1)^{5/2}]\\
&\le 3^{5/2}\big(\E[x_{k+1}^{5/2}]+\E[|x_{0,1/2}|^{5/2}]+1\big)\\
&\le C_{k},
\end{align*}
and a similar estimate holds for $i=n-k+1$. Combining this bound with the estimate for $2\le i\le n-k$, we have
\begin{equation}\label{beta expectation}
\sum_{i=1}^{n-k+1}\E\Big[\max\{x_i-x_{i-1}, x_{i+k}-x_i\}^{5/2}\mathbbm{1}_{\{\widetilde{\lambda}_n^*\in[1-\eps_n,1+\eps_n]\}}\Big]\le C_{k}n.
\end{equation}
The lemma now follows from \eqref{height Anc_k} and \eqref{beta expectation}.
\end{proof}

We now prove Proposition \ref{Proposition: diff2}. We will show that with high probability, one can lower and upper bound $l(Anc_k(\widetilde{\Pi}_{\widetilde{\lambda}_n^*}\cap A_n))$ using $l(Anc_k(\widetilde{\Pi}_{\widetilde{\lambda}_n^*})\cap A_n)$ and some error terms.
\begin{proof}[Proof Proposition \ref{Proposition: diff2}]
We will prove that for any subsequence $(n_m)_{m\ge 1}$, there is a further subsequence $(n_{m_l})_{l\ge 1}$ such that
\begin{equation}\label{Sampling from a fixed time, lower bound}
\liminf_{l\rightarrow\infty}\frac{l\Big(Anc_k\Big(\widetilde{\Pi}_{\widetilde{\lambda}_{n_{m_l}}^*}\cap A_{n_{m_l}}\Big)\Big)}{n}\ge  \frac{C_1'}{2}\cdot\frac{\Gamma(k + 1/2)}{k!}\qquad\text{a.s.},
\end{equation}
and another further subsequence $(n_{m_{l'}})_{l'\ge 1}$ such that
\begin{equation}\label{Sampling from a fixed time, upper bound}
\limsup_{l\rightarrow\infty}\frac{l\Big(Anc_k\Big(\widetilde{\Pi}_{\widetilde{\lambda}_{n_{m_{l'}}}^*}\cap A_{n_{m_{l'}}}\Big)\Big)}{n_{m_{l'}}}\le  \frac{C_1'}{2}\cdot \frac{\Gamma(k+1/2)}{k!}\qquad\text{a.s.}.
\end{equation}
The statement of the proposition then follows from \eqref{Sampling from a fixed time, lower bound} and \eqref{Sampling from a fixed time, upper bound}.

We start with the proof of \eqref{Sampling from a fixed time, lower bound}, which does not require Lemmas \ref{Lemma: Height of the Anc_k1} and \ref{Lemma: Height of the Anc_k2}. Fix $\eps>0$ and $T>0$. Let $E_n$ be the event that the following hold:
\begin{enumerate}
\item $\sup_{t\in[0,T]} |g_{(\eps n,0)}(t)-\eps n|\le \eps n,$
\item $\sup_{t\in[0,T]}|g_{((1-\eps) n, 0)}(t)-(1-\eps)n|\le \eps n.$
\end{enumerate}
By Theorem \ref{Theorem: transversal fluctuation}, we have $\P(E_n^c)\le 2C_3 e^{-C_4 \eps^3 n^3/T^2}$, which is summable in $n$. Also, on the event $E_n$, we have
$$
Anc_k(\widetilde{\Pi}_{\widetilde{\lambda}_n^*}\cap A_n)\supseteq   Anc_k(\widetilde{\Pi}_{\widetilde{\lambda}_n^*})\cap ([2\eps n,(1-2\eps)n]\times[0,T]),
$$
and hence
$$
\liminf_{n\rightarrow\infty}\frac{l\big(Anc_k(\widetilde{\Pi}_{\widetilde{\lambda}_n^*}\cap A_n)\big)}{n}\ge  \liminf_{n\rightarrow\infty} \frac{l\big(Anc_k(\widetilde{\Pi}_{\widetilde{\lambda}_n^*})\cap ([2\eps n,(1-2\eps)n]\times[0,T])\big)}{n}.
$$
Therefore, by Remark \ref{Remark}, we can choose a subsequence $(n_{m_l})$ such that
$$
\liminf_{l\rightarrow\infty}\frac{l\Big(Anc_k\Big(\widetilde{\Pi}_{\widetilde{\lambda}_{n_{m_l}}^*}\Big)\cap A_{n_{m_l}}\Big)}{n}\ge (1-4\eps)\E[l\big(Anc_{k}(\widetilde{\Pi}_1)\cap([0, 1)\times[0,T])\big)].
$$
Recall from \eqref{sfs2} that 
$$
\E[l\big(Anc_{k}(\widetilde{\Pi}_1)\cap([0, 1)\times\R_+)\big)]=\frac{C_1'}{2}\cdot\frac{\Gamma(k+1/2)}{k!}.
$$
Equation \eqref{Sampling from a fixed time, lower bound} then follows from a diagonalization argument as we let $\eps\rightarrow0$ and $T\rightarrow\infty$.

We now prove \eqref{Sampling from a fixed time, upper bound}.
Fix $\eps>0$. Let $F_n$ be the event that
\begin{enumerate}
\item $Anc_k(\widetilde{\Pi}_{\widetilde{\lambda}_n^*}\cap A_n)\cap Anc_k(\widetilde{\Pi}_{\widetilde{\lambda}_n^*})\cap (\R\times[n,\infty))=\emptyset$,
\item $\sup_{t\in[0,n]} |g_{(-\eps n,0)}(t)+\eps n|\le \eps n,$
\item $\sup_{t\in[0,n]}|g_{((1+\eps) n, 0)}(t)-(1+\eps)n|\le \eps n.$
\end{enumerate}
By Theorem \ref{Theorem: transversal fluctuation} and Lemma \ref{Lemma: Height of the Anc_k2}, we have
$$
\P(F_n^c)\le Cn^{-2/3}\log^{52} n+\P\big(\widetilde{\lambda}_n^*\notin \big[1-(\log n)/\sqrt{n},1+(\log n)/\sqrt{n}\big]\Big)+2C_3 e^{-C_4 \eps^3 n}.
$$
One may choose $(n_{m_l})_{l\ge 1}$ so that $\sum_l \P(F_{n_{m_l}}^c)<\infty$. If the event $F_n$ occurs, then for each $(y,t)\in Anc_k(\widetilde{\Pi}_{\widetilde{\lambda}_n^*}\cap A_n)$, either $(y,t)\notin Anc_k(\widetilde{\Pi}_{\widetilde{\lambda}_n^*})$ or $(y,t)\in Anc_k(\widetilde{\Pi}_{\widetilde{\lambda}_n^*})$. For the latter case, we have $t\le n$ by the first condition of $F_n$ and then $-2\eps n\le y\le (1+2\eps )n$ by the second and third conditions of $F_n$. Then 
\begin{align*}
&l\big(Anc_k(\widetilde{\Pi}_{\widetilde{\lambda}_n^*}\cap A_n)\big)\\
&\le l\big(Anc_k(\widetilde{\Pi}_{\widetilde{\lambda}_n^*}\cap A_n)\setminus Anc_k(\widetilde{\Pi}_{\widetilde{\lambda}_n^*})\big)+ l\big(Anc_k(\widetilde{\Pi}_{\widetilde{\lambda}_n^*})\cap ([-2\eps n,(1+2\eps)n]\times[0,n])\big)\\
&\le l\big(Anc_k(\widetilde{\Pi}_{\widetilde{\lambda}_n^*}\cap A_n)\setminus Anc_k(\widetilde{\Pi}_{\widetilde{\lambda}_n^*})\big)+ l\big(Anc_k(\widetilde{\Pi}_{\widetilde{\lambda}_n^*})\cap ([-2\eps n,(1+2\eps)n]\times\R_+)\big).
\end{align*}
Then, by Remark \ref{Remark} and Lemma \ref{Lemma: Height of the Anc_k1}, choosing a further subsequence if necessary, we have
\begin{align*}
\limsup_{l\rightarrow\infty}\frac{l\big(Anc_k(\widetilde{\Pi}_{\lambda_{n_{m_{l'}}}^*}\cap A_{n_{m_{l'}}})\big)}{n_{m_{l'}}}\le (1+4\eps)\frac{C_1'}{2}\cdot\frac{\Gamma(k+1/2)}{k!}.
\end{align*}
Equation \eqref{Sampling from a fixed time, upper bound} then follows from a diagonalization argument as we let $\eps\rightarrow0$.
\end{proof}

\section{Sampling from the entire population}\label{Section: 2-dimension}

We will proceed by proving Propositions~\ref{Proposition: expectation, 1} -- \ref{Proposition: diff1}. In Section~\ref{Subsection: expectation1}, we will prove Propositions \ref{Proposition: expectation, 1} and \ref{Proposition: expectation, 2} using the scaling invariance property of the directed landscape. In Section~\ref{Subsection: lln, points anywhere}, we prove a strong law of large numbers as stated in Proposition~\ref{Proposition: lln, points anywhere}. It is worth noting that, while the directed landscape is independent over disjoint time intervals, the quantity $l(Anc_k(\Pi_\lambda)\cap A_n)$ does not enjoy this property because the geodesic $g_{(x,s)}$ depends on $\L|_{\R\times(s,\infty)}$. A coupling argument is carried out in Lemma~\ref{Lemma: localization for sampling from the population} so that this dependence is localized in the temporal direction.  In Section~\ref{Subsection: robust1}, we establish a stability result in terms of the sampling rate, allowing us to change from a deterministic sampling rate in Proposition~\ref{Proposition: lln, points anywhere} to a random sampling rate in Proposition~\ref{Proposition: lln, points in the box}. Section~\ref{Subsection: diff1} concludes the proof of Theorem~\ref{Theorem: main1} by proving Proposition~\ref{Proposition: diff1}.

\subsection{Proof of Propositions \ref{Proposition: expectation, 1} and \ref{Proposition: expectation, 2}}\label{Subsection: expectation1}

We first need to establish a scaling result that is similar to Proposition~\ref{Proposition: properties of T_a}.
For any $A\subseteq \R^2$, let 
\begin{equation*}
H(A)= \sup\{|t'-t|:(y,t), (y',t')\in A\} 
\end{equation*}
be the height of $A$ and 
\begin{equation*}
W(A)= \sup\{|x'-x|: (x,s),(x',s')\in A\}
\end{equation*}
be the width of $A$. We take $H(\emptyset)=W(\emptyset)=0$ by convention. For any $a\ge 0$, let
\begin{equation}\label{Definition: height and width of descendants}
\H_{ a}:=\{(y,t)\in\R^2:H(D_{(y,t)})> a\},\qquad\W_{a}:=\{(y,t)\in\R^2:W(D_{(y,t)})> a\}   
\end{equation}
be the collection of space-time pairs whose descendants have height (resp. width) more than $a$. Also, recall $\V_a$ from \eqref{Introduction: def T_a}. The following result, which is similar to Corollary \ref{Corollary: Vtildenew} is an immediate consequence of Lemma \ref{Lemma: semigroup}.

\begin{corollary}\label{Corollary: semigroup}
If $(y',t')$ is an ancestor of $(y,t)$, i.e., $g_{(y,t)}(t'-t)=y'$, then $D_{(y',t')}\supseteq D_{(y,t)}$. In particular, if $(y,t)\in\V_{a}$ (resp. $\H_a$ or $\W_a$), then $(y',t')\in\V_{a}$ (resp. $\H_a$ or $\W_a$).
\end{corollary}

We now collect some results about $\H_a,\W_a,$ and $\V_a$ which are consequences of the scaling property of $\L$.  This result is proved in exactly the same way as Proposition \ref{Proposition: properties of T_a, fixed time}, so we do not repeat the argument.

\begin{proposition}\label{Proposition: properties of T_a}
Recall the scaling map $S_{\rho}: \R^2 \rightarrow \R^2$ defined by 
$S_\rho(x,s)=(\rho^{-2}x,\rho^{-3}s)$.
The following hold for $\V_{a}$, $\H_{a}$, and $\W_{a}$:
\begin{enumerate}
\item We have $S_\rho(\V_{a})\stackrel{d}{=}\V_{\rho^{-5}a}$, $S_\rho(\H_{a})\stackrel{d}{=}\H_{\rho^{-3}a}$, and $S_\rho(\W_{a})\stackrel{d}{=}\W_{\rho^{-2}a}$.

\item We have $\P(\V_{a}=\emptyset)=\P(\H_{a}=\emptyset)=\P(\W_{a}=\emptyset)=0$ for all $a>0$.
\end{enumerate}
\end{proposition}

\begin{proof}[Proof of Proposition \ref{Proposition: expectation, 1}]
Firstly, we show that \eqref{cdf1} holds when $a=1$, with $C_2'\in[0,\infty]$. Fix $t\in\R$. We define a measure $\widetilde{\mu}$ on $\R$ by
\begin{equation}\label{mu1}
\widetilde{\mu}(\widetilde{A}) :=\E[\#(\V_{1}^t\cap \widetilde{A})]\qquad\text{ for all measurable } \widetilde{A}\subseteq\R.
\end{equation}
The countable additivity of $\widetilde{\mu}$ follows from the linearity of expectation and the monotone convergence theorem. The measure $\widetilde{\mu}$ is translation invariant by part 2 of Proposition \ref{Proposition: properties of the landscape}. It follows that there exists a constant $C_2'\in[0,\infty]$ such that $\widetilde{\mu}(\widetilde{A})=C_2'\widetilde{m}(\widetilde{A})$. Part 1 of Proposition \ref{Proposition: properties of the landscape} implies that the constant $C_2'$ does not depend on $t$.

We now show that \eqref{cdf1} holds for all $a>0$, with $C_2'\in[0,\infty]$.  Recall the scaling map $S_\rho$ from \eqref{scaling map}. We define its action on the spatial axis by 
\begin{equation}\label{scaling map, fixed time}
\widetilde{S}_{\rho}(x)=\rho^{-2} x,\qquad\text{for all }x\in\R. 
\end{equation}
For any $\widetilde{A}\subseteq\R$, we write $\widetilde{S}_{\rho}(\widetilde{A})$ for the image of $\widetilde{A}$ under $\widetilde{S}_{\rho}$.
Applying part 1 of Proposition \ref{Proposition: properties of T_a} with $\rho = a^{1/5}$, we have $S_{a^{1/5}}(\V_a)\stackrel{d}{=}\V_1$. Taking the time-$(a^{-3/5}t)$ slice, we get
$$
(\widetilde{S}_{a^{1/5}}(\V_a^t))=(S_{a^{1/5}}(\V_a))^{a^{-3/5}t}\stackrel{d}{=}\V_1^{a^{-3/5}t}.
$$
Therefore,
\begin{align*}
\#(\V^t_{a}\cap \widetilde{A})=\#\left(\widetilde{S}_{a^{1/5}}(\V_{a}^t\cap \widetilde{A})\right)=\#\left(\widetilde{S}_{a^{1/5}}(\V^t_{a})\cap \widetilde{S}_{a^{1/5}}(\widetilde{A})\right)\stackrel{d}{=}\#\left(\V_1^{a^{-3/5}t}\cap \widetilde{S}_{a^{1/5}}(\widetilde{A})\right).
\end{align*}
Taking expectations, and then using \eqref{mu1}, we have
$$
\E[\#(\V^t_{a}\cap \widetilde{A})] = \E\Big[\#\big(\V_{1}^{a^{-3/5}t}\cap \widetilde{S}_{a^{1/5}}(\widetilde{A})\big)\Big]=C_2'\widetilde{m}\big(\widetilde{S}_{a^{1/5}}(\widetilde{A})\big)=C_2'\widetilde{m}(\widetilde{A})a^{-2/5}.
$$
To prove \eqref{cdf1}, it remains to show that $C_2'\in(0,\infty)$.

We first rule out the possibility that $C_2'=0$. Fix $a>0$. Seeking a contradiction, we suppose that $C_2'=0$. Then $\E[\#\V^t_{a}]=0$ for all $t\in\R$. This implies that $\P(\V^t_{a}=\emptyset \text{ for all }t\in\Q)=1$. It then follows from Corollary \ref{Corollary: semigroup} that $\P(\V_{a}=\emptyset)=1$, which contradicts part 2 of Proposition \ref{Proposition: properties of T_a}.

\begin{figure}[h]
\centering
\begin{tikzpicture}[scale=0.25]
\draw (-20,0)--(20,0);
\draw (-20,10)--(20,10);
\draw [dashed] plot [smooth cycle,tension=1] coordinates {(-15,10) (-16, 4) (-14, -2) (-11, -6) (-10, 2) (-14, 6)};
\draw [dashed] plot [smooth cycle,tension=1] coordinates {(10,10) (9,9) (0, 8) (3, 4) (8, 2) (12, 4) (17, 6) (12,8)};
\draw plot [smooth ,tension=1] coordinates {(-13,0) (-12, 3) (-15, 6) (-15,8) (-15,10) (-14, 12)};
\draw plot [smooth ,tension=1] coordinates {(8,0) (3, 1) (-2,4) (-1,6) (-1, 8) (1,10) (-2, 12)};
\fill (-15,10) circle(8pt);
\fill (10,10) circle(8pt); 
\node at (-21,0) {$0$};
\node at (-21,10) {$t$};
\node at (-12,11) {$(y_1,t)$};
\node at (-7,-4) {$D_{(y_1,t)}$};
\node at (11,11) {$(y_2,t)$};
\node at (19,7) {$D_{(y_2,t)}$};
\end{tikzpicture}
\caption{If $H(D_{(y,t)})>t$, then there is a geodesic starting from time zero passing through $(y,t)$. Otherwise, the geodesic has to be ``wide" to have a large volume, in which case we can find a geodesic with a large transversal fluctuation.}\label{Figure: tall wide}
\end{figure}

We next rule out the possibility that  $C_2'=\infty$. Note that if $m(D_{(y,t)})$ is large, then $D_{(y,t)}$ has to be either ``tall" or ``wide" (see Figure \ref{Figure: tall wide}). We take $t=1$, $a=1$, and $\widetilde{A}=[-1,1]$. Recall the definition of $\H_a$ from \eqref{Definition: height and width of descendants}. We decompose $\V^1_{1}\cap \widetilde{A}$ according to whether the height is greater than $1$ or not:
$$
\V^1_{1}\cap \widetilde{A}= (\V^1_{1}\cap \widetilde{A}\cap \H_{1}^1) \cup (\V^1_{1}\cap \widetilde{A}\cap(\H_{1}^c)^1).
$$
If $(y,1)\in\V^1_{1}\cap \widetilde{A}\cap \H_{1}^1$, then by the definition of $\H_{1}$, we have $H(D_{(y,1)})>1$ and there exists a geodesic $g_{(x,0)}$ that passes through $(y,1)$. Recall the definition of $N_1$ from Proposition \ref{Proposition: number of intersection points}. Then
$$
\#(\V^1_{1}\cap \widetilde{A}\cap \H_{1}^1)\le N_{1}.
$$
In particular, 
$$
\E\big[\#(\V^1_{1}\cap \widetilde{A}\cap \H_{1}^1)\big]\le \E[N_{1}]<\infty.
$$
Let $\widetilde{B}=\V^1_{1}\cap \widetilde{A}\cap (\H_{1}^1)^c$. It remains to show that $\E[\# \widetilde{B}]<\infty$. 
Note that the sets $D_{(y,1)}$ are disjoint for distinct $y\in \widetilde{A}$, and that $m(D_{(y,1)})> 1$ and $H(D_{(y,1)})\le 1$ for all $y\in\widetilde{B}$. Therefore, if we let
\begin{equation*}
D_A:=\bigcup_{(x,s)\in A}D_{(x,s)}
\end{equation*}
be the descendants of $A\subseteq\R^2$, then 
$$
\#\widetilde{B}\le \sum_{y\in \widetilde{B}}m(D_{(y,1)}) =m (D_{\widetilde{B}\times \{1\}})\le m \left(D_{\widetilde{A}\times \{1\}}\cap(\R\times[0,1])\right).
$$
Taking expected values, we have
$$
\E[\#\widetilde{B}]\le \E\left[m \left(D_{\widetilde{A}\times \{1\}}\cap(\R\times[0,1])\right)\right].
$$
To bound the right-hand side, for any $K>0$, let $E_K$ be the event 
that 
$$
m\left(D_{\widetilde{A}\times \{1\}}\cap(\R\times[0,1])\right)\ge 2(K+1).
$$
We claim that the event $E_K$ occurs only if one of the following events occurs:
\begin{enumerate}
\item $\sup_{u\in[0,1]}|g_{(K/2+1,0)}(u)-(K/2+1)|\ge K/2$,
\item $\sup_{u\in[0,1]}|g_{(-K/2-1,0)}(u)+(K/2+1)|\ge K/2$.
\end{enumerate}
Indeed, unless one of these two events occurs, if $(x,s)\in D_{(y,1)}$, where $-1\le y\le 1$,  $|x|>K+1$ and $0\le s\le 1$, then the geodesic $g_{(x,s)}$ must cross either the geodesic $g_{(K/2+1,0)}$ or $g_{(-K/2-1,0)}$, which contradicts the ordering of geodesics in Proposition \ref{Proposition: ordering of geodesics}. Therefore, by Theorem \ref{Theorem: transversal fluctuation}, we have
$$
\P(E_K)\le 2\P\left(\sup_{u\in[0,1]} |g_{(K/2+1,0)}(u)-(K/2+1)|\ge K/2\right)\le 2C_3 e^{-C_4 K^3/8}.
$$
It follows that $\E[\#\widetilde{B}]<\infty$, which concludes the proof.
\end{proof}

\begin{proof}[Proof of Proposition \ref{Proposition: expectation, 2}]
It suffices to consider the case where $\widetilde{m}(A^t)<\infty$ for all $t\in \R$, as one can replace $A$ by $A\cap([-n,n]\times\R)$ and take the limit as $n\rightarrow\infty$. Recalling the definition of the branch length from \eqref{def length}, we have
\begin{equation*}
\E\Big[l\big(Anc_k(\Pi_\lambda)\cap A\big)\Big]
= \E\left[\int_{-\infty}^\infty \#(Anc_k(\Pi_\lambda)\cap A)^t\ dt\right]=\int_{-\infty}^\infty \E\Big[\#\big(Anc_k(\Pi_\lambda)\cap A\big)^t\Big]\ dt.
\end{equation*}
To evaluate the expectation, note that if $m(D_{(y,t)})=a$, then $(y,t)\in Anc_k(\Pi_\lambda)$ with probability $e^{-\lambda a}(\lambda a)^k/k!$. 
Using \eqref{cdf1}, let 
\begin{equation}\label{Sampling from the population, density}
\phi(a)=-\frac{d}{da}\E[\#([0,1]\cap\V_a^t)] = \frac{2C_2' a^{-7/5}}{5}  
\end{equation}
be the ``density" for the number of points whose descendants have Lebesgue measure ``exactly" $a$.  Just as in the proof of \eqref{Sampling from a fixed time: integration over a}, we can use a discrete approximation to justify integrating over $a$ to obtain
\begin{align}\label{Sampling from the entire population, expectation}
\E\Big[\#\big(Anc_k(\Pi_\lambda)\cap A\big)^t\Big]&=\int_0^\infty \frac{e^{-\lambda a}(\lambda a)^k}{k!}\cdot \frac{2C_2' a^{-7/5}}{5}\cdot\widetilde{m}(A^t)\ da\nonumber\\
&= \frac{2C_2'\lambda^k}{5k!}\cdot\widetilde{m}(A^t)\int_0^\infty e^{-\lambda a} a^{k-7/5}\ da\nonumber\\
&=\frac{2C_2' \lambda^{2/5}}{5}\cdot\widetilde{m}(A^t) \cdot\frac{\Gamma(k-2/5)}{k!}.
\end{align}
Equation \eqref{sfs1} then follows from integrating over $t$.
\end{proof}

\begin{remark}\label{Remark: scaling of H_a}
Similar to the proof of \eqref{cdf1}, one can show that there exists a constant $C_2''\in(0,\infty)$ such that 
$$
\E\Big[\#\big(\mathcal{H}_a^t \cap \widetilde{A}\big)\Big] = C_2''\widetilde{m}(\widetilde{A})a^{-2/3},
$$
for all $t\in\R$ and all measurable subsets $\widetilde{A}\subseteq\R$. We will use this remark in Lemma~\ref{Lemma: tall bubbles} later.
\end{remark}

\subsection{Proof of Proposition \ref{Proposition: lln, points anywhere}}\label{Subsection: lln, points anywhere}

We now turn to the law of large numbers result.  As mentioned at the beginning of Section \ref{Section: 2-dimension}, we now introduce a coupling so that the dependence of $l(Anc_k(\Pi_\lambda)\cap A_n)$ on $\L$ is localized. In the next lemma, we consider the set $R_M$ of points in a unit box $B$ that are ancestral to $k$ sampled points and whose set of descendants has height at most $M$. We show that with high probability, $l(R_M)$ agrees with a random variable $X$ which is obtained by replacing infinite upward geodesics by finite geodesics in the definition of $Anc_{k}(\Pi_\lambda)$.  The random variable $X$ therefore depends on $\L$ only within a horizontal strip of height $T$. The dominant term in the bound on $\P(X\neq l(R_M))$ is an estimate for the probability that two finite geodesics do not intersect and comes from Proposition \ref{Proposition: disjoint finite geodesics} with $x_2-x_1=O(\log T)$ and $y_2-y_1=O(T^{2/3}\log T)$.

\begin{lemma}\label{Lemma: localization for sampling from the population}
Recall the definition of $\H_a$ from \eqref{Definition: height and width of descendants}. Let $B=[-1/2, 1/2]\times[0,1]$. For any $M>1$ and any $T>2M+1$, let
$$
R_M=(Anc_{k}(\Pi_\lambda)\setminus\H_M)\cap B,
$$
as shown in Figure \ref{Figure: R_M} below. Then there exists a random variable $X$ measurable with respect to the $\sigma$-field generated by $\Pi_{\lambda}\cap (\R\times [-2M,1])$ and $\L|_{\R\times[-2M,T-2M]}$ and a constant $C_M$ depending on $M$ such that
$$
\P(X\neq l(R_M))\le C_M  T^{-1/3}(\log ^{3/2}T)\exp(\log ^{5/6}T).
$$
\end{lemma}

\begin{figure}[h]
\centering
\begin{tikzpicture}[scale=0.3]
\draw[draw=black] (-3,0) rectangle ++(6,6);
\draw [dashed] plot [smooth cycle,tension=1] coordinates {(0,4) (-5,0) (-4, -2) (-1, -3) (1, -1) (1, 2)};
\draw[dotted] (0,4)--(5,4);
\draw[dotted] (-1,-3)--(5,-3);
\draw[<->] (5,-3) -- (5,4);

\draw (0,1) [fill=black]circle (2pt);
\draw (-1,-1) [fill=black]circle (2pt);
\draw (-3,-2) [fill=black]circle (2pt);
\draw (0.5,4) [fill=black]circle (6.67pt);

\node at (-4,6){$B$};
\node at (0,5){$(y,t)$};
\node at (9.5,0){$H(D_{(y,t)})\le M$};
\node at (-7.5,-1){$D_{(y,t)}$};
\end{tikzpicture}
\caption{A point $(y,t)\in R_M=(Anc_{k}(\Pi_\lambda)\setminus\H_M)\cap B$ with $k=3$. The small dots in the figure represent sampled points in $\Pi_\lambda$.}\label{Figure: R_M}
\end{figure}

\begin{proof}
For any $(x,s;y,t)\in \R^4_{\uparrow}$ and $0\le u\le t-s$, recall that $g_{(x,s;y,t)}$ is the left-most geodesic connecting $(x,s)$ and $(y,t)$ and $g_{(x,s;y,t)}(u)$ is the spatial position of the geodesic $g_{(x,s;y,t)}$ at time $s+u$. The random variable $X$ is obtained by considering the finite geodesics from $(x,s)\in\R\times[-2M,T-2M)$ to the point $(0,T-2M)$. For $y\in \R$ and $-2M\le t\le T-2M$, we define the analogs of $D_{(y,t)}$, $\H_a$, and $Anc_k(\Pi_{\lambda})$ for these fininte geodesics as 
\begin{align*}
D'_{(y,t)} &:= \{(x,s): -2M\le s\le t,  g_{(x,s;0,T-2M)}(t-s)=y\},\\
\H'_a&:=\{(y,t): -2M\le t\le T-2M, H(D'_{(y,t)})>a\},\\
Anc'_k(\Pi_{\lambda}) &:= \{(y,t):-2M\le t\le T-2M, \#(D_{(y,t)}'\cap\Pi_{\lambda})=k\}.
\end{align*}
We define the random variable $X$ by
$$
R_M':=(Anc'_{k}(\Pi_\lambda)\setminus\H'_M)\cap B,\qquad
X:=l(R_M').
$$
Note that $X$ is defined in terms of the finite geodesics $\{g_{(x,s;0,T-2M)}\}_{x\in\R,-2M\le s\le T-2M}$, and we are considering $(y,t)\in B$ with $H'(D_{(y,t)})\le M$. It follows that $X$ is measurable with respect to $\Pi_{\lambda}\cap (\R\times [-2M,1])$ and $\L|_{\R\times[-2M,T-2M]}$. 

It remains to show that $X=l(R_M)$ with high probability. Let $K>0$ be a constant to be determined. Let $E_K$ be the intersection of the following events. The events $E_{K,1}$ to $E_{K,6}$ impose conditions on transversal fluctuations of four infinite upward geodesics, and the event $E_{K,7}$ is a coalescence event between two finite geodesics. See Figure \ref{Figure: coupling} for an illustration. We will show that $\P(E_k^c)$ is small and that $X=l(R_M)$ on $E_k$.
\begin{figure}[h]
\centering
\begin{tikzpicture}[scale=0.25]
\draw[draw=black, fill=gray!30] (8,0) rectangle ++(24,12);
\draw[draw=black, fill=gray!30] (18,8) rectangle ++(4,4);
\draw[black, dashed] (0,0) -- (40,0);
\draw[black, dashed] (0,8) -- (40,8);
\draw[black, dashed] (0,12) -- (40,12);
\draw[black, dashed] (-10,32) -- (50,32);

\draw plot [smooth,  tension=1] coordinates {(36,0) (38, 10) (25,16) (20,18)(20,20) };
\draw plot [smooth, tension=1] coordinates {(4,0) (6, 10) (12,15.5)(20,20) };
\draw plot [smooth, tension=1] coordinates {(20,20) (21,22) (20, 24) };
\draw plot [smooth, tension=1] coordinates {(28,0) (29, 5) (25,10) (27, 13) (25, 16) };
\draw plot [smooth, tension=1] coordinates {(12,0) (10, 4) (9,9) (12,15.5) };
\draw plot [smooth, tension=1] coordinates {(20,24) (12,26) (6, 28) (0,30) (-10, 32)};
\draw plot [smooth, tension=1] coordinates {(20,24) (20,25) (24, 26) (32, 28) (40, 30) (50,32)};
\draw plot [smooth, tension=1] coordinates {(20,24)(18, 27) (20, 30) (15, 33) (12,36) };

\begin{scope}
\clip(0,0) rectangle (100,20);
\draw [pattern={ Dots[angle=45,distance={10pt}]}]
(36,0) --  plot [smooth,  tension=1] coordinates {(36,0) (38, 10) (25,16) (20,18)(20,20) } --
plot [smooth, tension=1] coordinates {(20,20) (12, 15.5) (6,10)(4,0)}  -- (4,0);
\end{scope}

\fill (-10,32) circle (8pt);
\fill (50,32) circle (8pt);
\fill (20,20) circle (8pt);
\fill (4,0) circle (8pt);
\fill (12,0) circle (8pt);
\fill (28,0) circle (8pt);
\fill (36,0) circle (8pt);

\node at (23,20) {$(y^*, t^*)$};
\node at (12, 25) {$g_1$};
\node at (3, 5) {$g_1$};
\node at (32, 27) {$g_2$};
\node at (40, 5) {$g_2$};
\node at (23, 35) {$\{g_{(x,-2M)}\}_{-K-1/2\le x\le K+1/2}$};
\node at (45,0) {$t=-2M$};
\node at (48,30) {$t=T-2M$};
\node at (45,8) {$t=0$};
\node at (45,12) {$t=1$};
\node at (36, 6) {$B_1$};
\node at (20,10) {$B$};
\node at (20,5) {$B_2$};
\node at (4,-1) {$-K-1/2$};
\node at (13,-1) {$-K/2-1/2$};
\node at (27,-1) {$K/2+1/2$};
\node at (36,-1) {$K+1/2$};
\node at (-5, 33.3) {$-K-1/2-T^{2/3}\log T$};
\node at (45,33.3) {$K+1/2+T^{2/3}\log T$};
\end{tikzpicture}
\caption{A realization of the event $E_{K}$. The geodesics $g_1$ and $g_2$ coalesce at $(y^*,t^*)$. The region $B_1$ is the dotted region enclosed by $g_1$, $g_2$, and $t=-2M$. For $(x,s)\in B_1$, the infinite upward geodesic $g_{(x,s)}$ and the finite geodesic $g_{(x,s; 0,T-2M)}$ pass through $(y^*,t^*)$ and agree up to time $t^*$.  If the event $E_K$ occurs, then the box $B_2 :=[-3K/4-1/2,3K/4+1/2]\times[-2M,1]$ is a subset of $B_1$.}
\label{Figure: coupling}
\end{figure}
\begin{enumerate}
\item $E_{K,1}: \sup_{u\in[0,2M+1]}|g_{(K/2+1/2,-2M)}(u)-(K/2+1/2)|< K/4$,
\item $E_{K,2}: \sup_{u\in[0,2M+1]}|g_{(-K/2-1/2,-2M)}(u)+(K/2+1/2)|< K/4$,
\item $E_{K,3}: \sup_{u\in[0,2M+1]}|g_{(K+1/2,-2M)}(u)-(K+1/2)|< K/4$,
\item $E_{K,4}: \sup_{u\in[0,2M+1]}|g_{(-K-1/2,-2M)}(u)+(K+1/2)|< K/4$,
\item $E_{K,5}: \sup_{u\in[0,T]}|g_{(K+1/2,-2M)}(u)-(K+1/2)|< T^{2/3}\log T$,
\item $E_{K,6}: \sup_{u\in[0,T]}|g_{(-K-1/2,-2M)}(u)+(K+1/2)|< T^{2/3}\log T$.
\item $E_{K,7}$ is the event that the finite geodesics $g_1=g_{(-K-1/2,-2M;-K-1/2-T^{2/3}\log T,T-2M)}$ and \newline $g_2=g_{(K+1/2,-2M;K+1/2+T^{2/3}\log T,T-2M)}
$ are nondisjoint.
\end{enumerate}
By Theorem \ref{Theorem: transversal fluctuation}, we have
$$
\P(E_{K,1}^c)=\P(E_{K,2}^c)=\P(E_{K,3}^c)=\P(E_{K,4}^c)\le C_3 e^{-C_4 K^3/(64(2M+1)^2)},
$$
and
$$
\P(E_{K,5}^c)=\P(E_{K,6}^c)\le C_3 e^{-C_4 \log ^3 T}.
$$
By Proposition \ref{Proposition: disjoint finite geodesics}, we have
$$
\P(E_{K,7}^c)\le C(2K+1)^{1/2}(2K+1+2T^{2/3}\log T)\cdot\frac{1}{T}\cdot\exp\left(\log^{5/6}\left(\frac{T}{(2K+1)^{3/2}}\vee 1\right)\right).
$$
Take $K=\log T$, and we have $\P(E_K)\ge 1-C_M T^{-1/3}(\log ^{3/2}T)\exp(\log ^{5/6}T)$.

It remains to show that $X=l(R_M)$ on $E_K$, which will imply the statement of the lemma. We make the following observations:
\begin{enumerate}
\item If the event $E_{K,7}$ occurs, then the geodesics $g_1$ and $g_2$ coalesce at some point $(y^*,t^*)$ where $t^*\le T-2M$. If the events $E_{K,5}$ and $E_{K,6}$ also occur, then applying Proposition \ref{Proposition: ordering of geodesics} to the finite geodesics $g_1$ and $g_2$ and the restrictions of $g_{(-K-1/2,-2M)}$ and $g_{(-K-1/2,-2M)}$ up to time $T-2M$, we conclude that the infinite upward geodesics $g_{(-K-1/2,-2M)}$ and $g_{(K+1/2,-2M)}$ must pass through $(y^*,t^*)$. Also, let
$$
B_1=\{(y,t)\in \R^2: -2M\le t\le t^*, g_1(t+2M)\le y\le g_2(t+2M)\}
$$
be the region between the geodesics $g_1$ and $g_2$ up to time $t^*$ (see Figure \ref{Figure: coupling}). Using Proposition~\ref{Proposition: ordering of geodesics} again, we see that for any $(x,s)\in B_1$, the geodesics $g_{(x,s)}$ and $g_{(x,s;0,T-2M)}$ agree up to time $t^*$, i.e., 
\begin{equation}\label{Sampling from the entire population: local1}
g_{(x,s)}(u)=g_{(x,s; 0,T-2M)}(u)\qquad\text{for all } (x,s)\in B_1 \text{ and }0\le u\le t^*-s.
\end{equation}
\item If the events $E_{K,3}$ and $E_{K,4}$ occur and we define $B_2 :=[-3K/4-1/2,3K/4+1/2]\times[-2M,1]$, then 
\begin{equation}\label{B2 B1}
B_2\subseteq B_1 .  
\end{equation}
\item If the events $E_{K,1}$ and $E_{K,2}$ occur, then descendants of individuals in $B$ stay within the box $B_2$. That is,
\begin{equation}\label{Sampling from the entire population: local2}
D_{(y,t)}\cap(\R\times[-2M,1])\subseteq B_2,\qquad\text{ for all }(y,t)\in B.
\end{equation}
\item If the event $E_K$ occurs, then we claim that
\begin{equation}\label{coupling height}
H(D_{(y,t)})\le M \qquad\text{ for all }(y,t)\in R_M'.
\end{equation}
Indeed, if $H(D_{(y,t)})>M$, then there exists $x\in\R$ and $s\in[-2M,t-M)$ such that the infinite upward geodesic $g_{(x,s)}$ passes through $(y,t)$. By \eqref{B2 B1} and \eqref{Sampling from the entire population: local2}, we have $(x,s)\in B_2\subseteq B_1$. Then \eqref{Sampling from the entire population: local1} implies that the finite geodesic $g_{(x,s;0,T-2M)}$ also passes through $(y,t)$. Therefore, we have $H(D'_{(y,t)})>M$, contradicting that $(y,t)\in R_M'\subseteq(\H'_M)^c$.
\end{enumerate}

We now prove that on the event $E_K$, the equality $R_M=R_M'$ holds, which implies that $X=l(R_M')=l(R_M)$. 

If $(y,t)\in R_M$, then we have $H(D_{(y,t)})\le M$ by definition and hence $D_{(y,t)}\subseteq \R\times[-M,1]$. Combining this with \eqref{B2 B1} and \eqref{Sampling from the entire population: local2} gives $D_{(y,t)}\subseteq B_2\subseteq B_1$. It then follows from \eqref{Sampling from the entire population: local1} that $D_{(y,t)}=D'_{(y,t)}$ and hence $(y,t)\in R_M'$. 

If $(y,t)\in R_M'$, then we have $H(D_{(y,t)})\le M$ from \eqref{coupling height}. Repeating the argument as in the case where $y\in R_M$ gives  us $D_{(y,t)}=D'_{(y,t)}$, which implies that $(y,t)\in R_M$. 
\end{proof}

We are now ready to prove Proposition \ref{Proposition: lln, points anywhere}. We use an ergodic theorem and the space-time stationarity of $\L$ to deduce the a.s. and $L^1$ convergence to a possibly random limit, and then use a variance computation to conclude that a strong law of large numbers holds, as stated in Proposition \ref{Proposition: lln, points anywhere}.
\begin{proof}[Proof of Proposition \ref{Proposition: lln, points anywhere}]
Recall the definition of $A_n$ from the statement of Proposition \ref{Proposition: lln, points anywhere}. For $m\in[n^5,(n+1)^5)$, where $n$ is a positive integer, we have
\begin{align*}
\frac{l(Anc_k(\Pi_\lambda)\cap A_{n^5})}{(n+1)^5}\le \frac{l(Anc_k(\Pi_\lambda)\cap A_m)}{m}\le\frac{l(Anc_k(\Pi_\lambda)\cap A_{(n+1)^5})}{n^5}.
\end{align*}
Therefore, it suffices to show that
\begin{equation}\label{Sampling from the entire population, integer side length}
\lim_{n\rightarrow\infty}\frac{l(Anc_k(\Pi_\lambda)\cap A_{n^5})}{n^5}= \frac{2C_2'\lambda^{2/5}}{5}\cdot \frac{\Gamma(k - 2/5)}{k!},\qquad \text{a.s. and in } L^1,  
\end{equation}
holds when we only consider positive integers $n$.

We will divide $A_{n^5}$ into $n^5$ unit boxes and consider the restrictions of $Anc_k(\Pi_\lambda)$ to those boxes. We will also apply a truncation so that the variance can be bounded. Let $M,N>1$ be fixed integers. For $i,j\in\Z$, let $B_{i,j}=[i-1,i]\times[j-1,j]$, and let 
$$
X_{i,j}=l\Big((Anc(\Pi_{\lambda})\setminus\H_M)\cap B_{i,j}\Big).
$$ 
By parts 1 and 2 of Proposition \ref{Proposition: properties of the landscape}, $\{X_{i,j}\}_{i,j\in\Z}$ is stationary under shifts in the time or space direction. Therefore, we can apply an ergodic theorem which holds for random variables indexed by $\Z^2$, as stated in Theorems~1 and 3 of \cite{b71}. Using the notation of \cite{b71}, let $G=\Z^2$, let the set $A_n=\{1,\dots ,n^2\}\times\{1,\dots, n^3\}$, and let $\gamma$ be the counting measure. We can then conclude from \cite{b71} that as $n\rightarrow\infty$,
\begin{equation}\label{Sampling from the entire population, ergodic}
\frac{1}{n^5}\sum_{i=1}^{n^2}\sum_{j=1}^{n^3} (X_{i,j}\wedge N)\text{ converges a.s. and in } L^1.
\end{equation}

We next show that the limit in \eqref{Sampling from the entire population, ergodic} is $\E[X_{1,1}\wedge N]$. Consider two boxes $B_{i,j}$ and $B_{i',j'}$ where $j'>j+2M+1$. Applying Lemma \ref{Lemma: localization for sampling from the population} to the box $B_{i,j}$ with $T=j'-j$, there exists a constant $C_M$ which could vary from line to line, and a random variable $X$, measurable with respect to the $\sigma$-field generated by $\Pi_{\lambda}\cap (\R\times[j-2M-1,j])$ and $\L|_{\R\times[j-2M-1,j'-2M-1]}$, such that
$$
\P(X\neq X_{i,j})\le C_M  (j'-j)^{-1/3}\log^{3/2} (j'-j)\exp(\log^{5/6}(j'-j)).
$$
Then for any $N>0$, we have
$$
|\E[X\wedge N]-\E[X_{i,j}\wedge N]|\le C_M N (j'-j)^{-1/3}\log^{3/2} (j'-j)\exp(\log^{5/6}(j'-j)),
$$
and 
\begin{align*}
&|\E[(X\wedge N)(X_{i',j'}\wedge N)]-\E[(X_{i,j}\wedge N)(X_{i',j'}\wedge N)]|\\
&\qquad\le C_M N^2 (j'-j)^{-1/3}\log^{3/2}(j'-j)\exp(\log^{5/6}(j'-j)).
\end{align*}
Since $X_{i',j'}$ is measurable with respect to the $\sigma$-field generated by $\Pi_\lambda\cap (\R\times(j'-2M-1,j'])$ and $\L|_{\R\times(j'-2M-1,\infty)}$, by the independent increments property in Proposition \ref{Proposition: independent increment}, the random variables $X$ and $X_{i',j'}$ are independent. Therefore, $\Cov(X\wedge N,X_{i',j'}\wedge N)=0$, and it follows that
\begin{align*}
&|\Cov(X_{i,j}\wedge N,X_{i',j'}\wedge N)|\\
&\qquad=|\Cov((X_{i,j}\wedge N)-(X\wedge N),X_{i',j'}\wedge N)+\Cov(X\wedge N,X_{i',j'}\wedge N)|\\
&\qquad=|\Cov((X_{i,j}\wedge N)-(X\wedge N),X_{i',j'}\wedge N)|\\
&\qquad=\big|\E\big[((X_{i,j}\wedge N)-(X\wedge N))(X_{i',j'}\wedge N)\big]-\E\big[(X_{i,j}\wedge N)-(X\wedge N)\big]\E\big[X_{i',j'}\wedge N\big]\big|\\
&\qquad\le C_M N^2 (j'-j)^{-1/3}\log^{3/2} (j'-j)\exp(\log^{5/6}(j'-j)).
\end{align*}
For two boxes $B_{i,j}$ and $B_{i',j'}$ where $|j-j'|\le 2M+1$, we use the trivial bound
$$
|\Cov(X_{i,j}\wedge N,X_{i',j'}\wedge N)|\le N^2.
$$
Therefore,
\begin{align*}
\Var\left(\sum_{i=1}^{n^2}\sum_{j=1}^{n^3} \left(X_{i,j}\wedge N\right)\right)&= \sum_{1\le i,i'\le n^2}\sum_{1\le j,j'\le n^3} \Cov(X_{i,j}\wedge N,X_{i',j'}\wedge N)\\
&=\sum_{1\le i,i'\le n^2}\Bigg(\sum_{|j-j'|\le 2M+1} \Cov(X_{i,j}\wedge N,X_{i',j'}\wedge N)\\
&\qquad\qquad\qquad+\sum_{|j-j'|> 2M+1} \Cov(X_{i,j}\wedge N,X_{i',j'}\wedge N)\Bigg)\\
&\le Cn^4\left(Mn^3\cdot N^2+n^3\sum_{T=2M+2}^{n^3}C_M N^2 T^{-1/3}(\log^2 T) \exp(\log^{5/6}T)\right)\\
&\le C_M N^2 n^9 (\log^2 n)\exp(\log^{5/6}n^3).
\end{align*}
It follows that 
\begin{equation}\label{Sampling from the entire population, wlln}
\frac{1}{n^5}\sum_{i=1}^{n^2}\sum_{j=1}^{n^3} (X_{i,j}\wedge N)\rightarrow \E[X_{1,1}\wedge N]\qquad\text{in probability}.   
\end{equation}
Note that the definition of $X_{i,j}$ depends on $M$, so the right hand side depends on both $M$ and $N$. Combining \eqref{Sampling from the entire population, ergodic} with \eqref{Sampling from the entire population, wlln}, we get
$$
\frac{1}{n^5}\sum_{i=1}^{n^2}\sum_{j=1}^{n^3} (X_{i,j}\wedge N)\rightarrow \E[X_{1,1}\wedge N]\qquad\text{a.s. and in } L^1.
$$
Taking the supremum over $M$ and $N$, we have
\begin{align*}
\liminf_{n\rightarrow\infty}\frac{1}{n^5} l\big(Anc_k(\Pi_\lambda)\cap A_{n^5}\big)&\ge \sup_{M,N}\liminf_{n\rightarrow\infty}\frac{1}{n^5}\sum_{i=1}^{n^2}\sum_{j=1}^{n^3} (X_{i,j}\wedge N)\\
&=\sup_{M,N} \E[X_{1,1}\wedge N]\\
&=\E\big[l(Anc_k(\Pi_\lambda)\cap A_1)\big].
\end{align*}
Recall from Proposition \ref{Proposition: expectation, 2} that $\E[l(Anc_k(\Pi_\lambda)\cap A_{n^5})]<\infty$. We apply Theorems 1 and 3 of \cite{b71} again to conclude that $l(Anc_k(\Pi_\lambda)\cap A_{n^5})/n^5$ converges a.s. and in $L^1$. It then follows from Lemma \ref{Lemma: probability fact} that
$$
\frac{l(Anc_k(\Pi_\lambda)\cap A_{n^5})}{n^5}\rightarrow \E[l(Anc_k(\Pi_{\lambda})\cap A_1)]\qquad\text{a.s. and in } L^1,
$$
which, together with equation \eqref{sfs1}, concludes the proof of \eqref{Sampling from the entire population, integer side length}.
\end{proof}

\subsection{Proof of Proposition \ref{Proposition: lln, points in the box}}\label{Subsection: robust1}

Now that we have proved Proposition \ref{Proposition: lln, points anywhere}, which gives a strong law of large numbers for $l(Anc_k(\Pi_{\lambda})\cap A_n)$, the next step is to show that we get a similar result, but with convergence in probability instead of almost sure convergence, when we sample in such a way that there are exactly $n$ points sampled from $A_n$.  
To prove this convergence, we first establish a result similar to Lemma \ref{Lemma: perturb the sampling intensity2}
which shows that $Anc_k(\Pi_\lambda)$ does not change much if we change the sampling intensity $\lambda$ by a small amount.

\begin{lemma}\label{Lemma: perturb the sampling intensity1}

Fix $\lambda> \eps>0$, and fix a positive integer $k$. Let $A$ be a measurable subset of $\R^2$. Let
$$
B=\bigcup_{1\le i\le k\le j, i\neq j} \left(Anc_{i}(\Pi_{\lambda-\eps})\cap Anc_{j}(\Pi_{\lambda+\eps})\right)
$$
be the pairs $(y,t)\in\R^2$ that are ancestral to at most $k$ points in $\Pi_{\lambda-\eps}$, at least $k$ points in $\Pi_{\lambda+\eps}$, and at least one point in $\Pi_{\lambda+\eps}\setminus\Pi_{\lambda-\eps}$. Then there exists a constant $C_{k,\lambda}$ depending on $k$ and $\lambda$ such that for all $\eps<\lambda/2$ and all measurable $A\subseteq\R^2$,
$$
\E[l(A\cap B)]\le C_{k,\lambda}m(A)\eps.
$$  
\end{lemma}
\begin{proof}
For each $i>0$, let
$$
B_i = Anc_{i}(\Pi_{\lambda-\eps}) \cap\{(y,t)\in\R^2: D_{(y,t)}\cap (\Pi_{\lambda+\eps}\setminus\Pi_{\lambda-\eps})\neq\emptyset\}
$$
be the pairs $(y,t)\in\R^2$ that are ancestral to $i$ points in $\Pi_{\lambda-\eps}$ and at least one point in $\Pi_{\lambda+\eps}\setminus\Pi_{\lambda-\eps}$.
Then, $B\subseteq \cup_{1\le i\le k} B_i$ by definition. Therefore,
\begin{align*}
\E[l(A\cap B)]\le \sum_{1\le i\le k} \E[l(A\cap B_i)]=\sum_{1\le i\le k} \int_{-\infty}^\infty \E\left[\#(A\cap B_i)^t\right]\ dt.
\end{align*}
Similar to the proof of Proposition \ref{Proposition: expectation, 2}, we let $\phi(\cdot)$ be defined as in \eqref{Sampling from the population, density}. If $m(D_{(y,t)})=a$, then the probability that $(y,t)$ is ancestral to $i$ points in $\Pi_{\lambda-\eps}$ and at least one point in $\Pi_{\lambda+\eps}\setminus\Pi_{\lambda-\eps}$ is
$$
\frac{e^{-(\lambda-\eps)a}(\lambda-\eps)^i a^i }{i!}(1-e^{-2\eps a}).
$$
Therefore, using an argument similar to the ones used to prove \eqref{Sampling from the entire population, expectation} and \eqref{ElAB},
\begin{align}\label{perturb bound}
\E[l(A\cap B)]&\le\sum_{1\le i\le k} \int_{-\infty}^\infty 
\int_0^\infty \widetilde{m}(A^t)\phi(a) \cdot \frac{e^{-(\lambda-\eps)a}(\lambda-\eps)^i a^i }{i!}(1-e^{-2\eps a})\ da\ dt\nonumber\\
&\leq \sum_{1\le i\le k} \int_0^\infty 
\left(\int_{-\infty}^\infty   \widetilde{m}(A^t)\ dt\right) \cdot \frac{2C_2' a^{-7/5}}{5}\cdot \frac{e^{-(\lambda-\eps)a}(\lambda-\eps)^i a^i }{i!} \cdot 2 \eps a \ da\nonumber\\
&\le C_{k,\lambda}m(A)\eps
\end{align}
because the integral is finite for all $i$, which completes the proof.
\end{proof}

\begin{proof}[Proof of Proposition \ref{Proposition: lln, points in the box}]
Define $B$ as in the statement of Lemma \ref{Lemma: perturb the sampling intensity1}. For $n$ sufficiently large, we apply Lemma \ref{Lemma: perturb the sampling intensity1} with $\lambda=1$, $\eps_n=(\log n)/\sqrt{n}$ and $A_n=[0,n^{2/5}]\times[0,n^{3/5}]$ to get
$$
\E[l(A_n\cap B)]\le C_{k,\lambda}\sqrt{n}\log n.
$$
In particular, as $n\rightarrow\infty$, by Markov's inequality we have
$$
\frac{l(A_n\cap B)}{n}\rightarrow 0,\qquad\text{in probability}.
$$
Recall the definition of $\lambda_n^*$ from \eqref{lambda1}. Then, as $n\rightarrow\infty$,
\begin{equation}\label{Sampling from the population: Poisson deviation}
\P(\lambda_n^*\in [1-\eps_n,1+\eps_n])\ge \P(\#(\Pi_{1-\eps_n}\cap A_n)< n<\#(\Pi_{1+\eps_n}\cap A_n))\rightarrow1.    
\end{equation}
On the event $\lambda_n^*\in [1-\eps_n,1+\eps_n]$, we have
$$
Anc_k(\Pi_{1})\ \triangle\  Anc_k(\Pi_{\lambda_n^*})\subseteq B.
$$
It follows that as $n\rightarrow\infty$, 
\begin{equation}\label{pertubation1}
\frac{l\Big(\big(Anc_k(\Pi_{1})\ \triangle\ Anc_k(\Pi_{\lambda_n^*})\big)\cap A_n\Big)}{n}\rightarrow 0,\qquad\text{in probability}.
\end{equation}
The claim of the proposition then follows from \eqref{pertubation1} and Proposition \ref{Proposition: lln, points anywhere}.
\end{proof}

\subsection{Proof of Proposition \ref{Proposition: diff1}}\label{Subsection: diff1}

To complete the proof of Theorem \ref{Theorem: main1}, it remains only to prove Proposition \ref{Proposition: diff1}. Recall that $A_n=[0,n^{2/5}]\times[0,n^{3/5}]$. To prove Proposition \ref{Proposition: diff1}, we need to compare the sets $Anc_k(\Pi_{\lambda_n^*})\cap A_n$ and $Anc_k(\Pi_{\lambda_n^*}\cap A_n)$. Recall from Figure \ref{Figure: diff1} in the introduction that $Anc_k(\Pi_{\lambda_n^*})\cap A_n$ consists of points in $A_n$ that are ancestral to $k$ sampled points in $\R^2$, while $Anc_k(\Pi_{\lambda_n^*}\cap A_n)$ consists of points in $\R^2$ that are ancestral to $k$ sampled points in $A_n$. Typically, points that are ancestral to $k$ sampled points in $A_n$ will themselves be in $A_n$, and points in $A_n$ that are ancestral to $k$ sampled points will have their descendants contained in $A_n$. Therefore, these sets will be close with high probability. Nevertheless, some discrepancies will occur near the boundary of $A_n$. These discrepancies need to be bounded, which requires some rather technical arguments.

Suppose $(y,t)\in\R^2$ is in one of these sets but not the other. Then there are two possibilities. 
\begin{enumerate}
\item $(y,t)\notin A_n$, and $(y,t)\in Anc_k(\Pi_{\lambda_n^*}\cap A_n)$. That is $(y,t)$ is not in the box $A_n$ but ancestral to $k$ points in the box.
\item $(y,t)\in Anc_{j}(\Pi_{\lambda_n^*})\cap A_n$ for some $j\ge k$, and $D_{(y,t)}\cap A_n^c\neq\emptyset$. That is, $(y,t)$ is in the box $A_n$ and is ancestral to at least $k$ of the sample points, and some of its descendants go outside the box $A_n$.
\end{enumerate}
We refer the reader back to Figure \ref{Figure: diff1} for an illustration, where the thick portion of $T_1$ in the right plot corresponds to the first possibility above, and the thick portion of $T_1$ in the left plot corresponds to the second possibility above.
If we write, for $A\subseteq\R^2$,
$$
Anc_{\ge k}(A):=\bigcup_{j\ge k}Anc_j(A),\qquad \text{ for all } A\subseteq\R^2
$$
for the set of points $(y,t)$ that are ancestral to at least $k$ points in $A$, then the discussion above implies that
$$
(Anc_k(\Pi_{\lambda_n^*}) \cap A_n)\ \triangle\  Anc_k(\Pi_{\lambda_n^*}\cap A_n)\subseteq F_{k,n}\cup G_{k,n},
$$
where
$$
F_{k,n}:=Anc_k(\Pi_{\lambda_n^*}\cap A_n)\setminus A_n
$$
is the set described in case 1 above and 
$$
G_{k,n} := \{(y,t)\in Anc_{\ge k}(\Pi_{\lambda_n^*})\cap A_n: D_{(y,t)}\cap A_n^c\neq \emptyset\}
$$
is the set described in case 2 above. We claim that for all positive integers $k$, as $n\rightarrow\infty$, 
\begin{equation}\label{Sampling from the entire population: comparison1}
\frac{l(F_{k,n})}{n}\rightarrow 0,\qquad\text{in probability},   
\end{equation}
and
\begin{equation}\label{Sampling from the entire population: comparison2}
\frac{l(G_{k,n})}{n}\rightarrow0,\qquad\text{in probability}.
\end{equation}
Proposition \ref{Proposition: diff1} is an immediate consequence of \eqref{Sampling from the entire population: comparison1} and \eqref{Sampling from the entire population: comparison2}. We prove \eqref{Sampling from the entire population: comparison1} in Section \ref{Subsubsection: comparison1}, and \eqref{Sampling from the entire population: comparison2} in Section \ref{Subsubsection: comparison2}.

\subsubsection{Proof of \eqref{Sampling from the entire population: comparison1}}\label{Subsubsection: comparison1}

In this section, we prove \eqref{Sampling from the entire population: comparison1} using a first moment computation with truncation. Recall that $F_{k,n}$ is the set of points that are not in $A_n$ but are ancestral to $k$ points in $A_n$.

\begin{proof}[Proof of \eqref{Sampling from the entire population: comparison1}]
To bound $l(F_{k,n})$, we split the branches at time $T_n = n^{3/5}\log\log n$. We have
\begin{equation}\label{Sampling from the entire population, diff1, case1}
l(F_{k,n})=l\big(F_{k,n}\cap(\R\times[0,T_n])\big)+l\big(F_{k,n}\cap(\R\times(T_n,\infty))\big).
\end{equation}
The choice of $T_n$ is somewhat flexible here.  However, we need to choose $T_n \gg n^{3/5}$ to ensure that with high probability, all geodesics starting from the box $A_n$ have coalesced by time $T_n$, and we need $T_n$ to be sufficiently small that \eqref{Sampling from the entire population, branches outside the box} below holds.

For the first term on the right hand side of \eqref{Sampling from the entire population, diff1, case1}, we will show in Lemma \ref{Lemma:  branches outside the box} that
\begin{equation}\label{Sampling from the entire population, branches outside the box}
\E\Big[l\big(F_{k,n}\cap(\R\times[0,T_n])\big)\Big] = o(n).
\end{equation}
It will then follow from Markov's inequality that
$$
\frac{l\big(F_{k,n}\cap(\R\times[0,T_n])\big)}{n}\rightarrow0,\qquad\text{in probability.}
$$
For the second term on the right hand side of \eqref{Sampling from the entire population, diff1, case1}, note that for $n>k$, the set $F_{k,n}\cap(\R\times(T_n,\infty))$ is nonempty only if the geodesics starting from the box $A_n$ have not coalesced into one geodesic at time $T_n=n^{3/5}\log\log  n$. We claim that this event occurs with probability going to 0 and hence
$$
\lim_{n\rightarrow\infty}\P\big(F_{k,n}\cap(\R\times(T_n,\infty))\neq\emptyset\big)=0,
$$
which completes the proof of \eqref{Sampling from the entire population: comparison1}.

It remains to justify our claim that, with probability going to 1, all the geodesics starting from the box $A_n$ have coalesced into one geodesic at time $T_n$. Denote this event by $E_n$. Let $z_n=n^{2/5}\log\log\log n$. Let $E_{n,1}$ be the event that the geodesics $g_{(-z_n,0)}$ and $g_{(n^{2/5}+z_n,0)}$ coalesce by time $T_n$. Let $E_{n,2}$ be the event that 
$$
\sup_{u\in[0,n^{3/5}]}|g_{(-z_n,0)}(u)-(-z_n)|< z_n,\qquad \sup_{u\in[0,n^{3/5}]}|g_{(n^{2/5}+z,0)}(u)-(n^{2/5}+z_n)|< z_n.
$$
Note that if both events $E_{n,1}$ and $E_{n,2}$ occur, then the event $E_n$ occurs. See Figure \ref{Figure: box coalescence} for an illustration. Also, Theorem \ref{Theorem: transversal fluctuation} and Proposition \ref{Proposition: 2-disjoint} guarantee that both events $E_{n,1}$ and $E_{n,2}$ occur with probability going to 1, which completes the proof.
\begin{figure}[h]
\centering
\begin{tikzpicture}[scale=0.25]
\draw [dotted] plot coordinates {(0, 0) (30,0)};
\draw [dotted] plot coordinates {(0, 4) (30,4)};
\draw [dotted] plot coordinates {(0, 11) (30,11)};
\draw[draw=black,fill=gray!30] (13,0) rectangle ++(4,4);
\draw plot [smooth, tension=1] coordinates {(10, 0)(9,2) (11,4) (14.5, 6) (15, 8)};
\draw plot [smooth, tension=1] coordinates {(20, 0)(18,2) (19,5) (15, 8)};
\draw plot [smooth, tension=1] coordinates {(15,8) (14.5, 9) (14,11)};

\fill (10,0) circle(8pt);
\fill (20,0) circle(8pt);
\fill (14,11) circle(8pt);

\node at (10,-1) {$(-z_n,0)$};
\node at (20,-1) {$(n^{2/5}+z_n, 0)$};
\node at (-1,0) {$0$};
\node at (-1,4) {$n^{3/5}$};
\node at (-5,11) {$T_n=n^{3/5}\log\log n$};
\node at (15,2) {$A_n$};
\end{tikzpicture}
\caption{When $E_{n,2}$ occurs, the box $A_n$ is between the geodesics $g_{(-z_n,0)}$ and $g_{(n^{2/5}+z_n,0)}$. If $E_{n,1}$ also occurs, then all the geodesics starting from the box $A_n$ have coalesced at time $T_n$.}
\label{Figure: box coalescence}
\end{figure}
\end{proof}

It remains to prove \eqref{Sampling from the entire population, branches outside the box} in the following Lemma. The dominant contribution comes from the geodesics starting from sampled points in the box $A_n$ that are close to the top boundary of $A_n$.
\begin{lemma}\label{Lemma:  branches outside the box}
For any $T\ge n^{3/5}$, we have
\begin{align*}
&\E\Big[l\big(F_{k,n}\cap (\R\times[0,T])\big)\Big]\le \frac{CTn^{2/5}}{\log n}.
\end{align*}
In particular, for $T_n=n^{3/5}\log\log n$,
$$
\E\Big[l\big(F_{k,n}\cap(\R\times[0,T_n])\big)\Big]\le C(n\log\log n )/\log n = o(n).
$$
\end{lemma}
\begin{proof}
Any point in $F_{k,n}$ must lie on the infinite upward geodesics starting from $k$ of the points in $\Pi_{\lambda_n^*}\cap A_n$. Therefore, if we write $G_{(y,t)}=\{(g_{(y,t)}(u), t+u): u\ge0\}\subseteq\R^2$ for the graph of $g_{(y,t)}$, then we have
$$
F_{k,n}\cap (\R\times[0,T])=\bigcup_{(y,t)\in \Pi_{\lambda_n^*}\cap A_n}F_{k,n}\cap (\R\times[0,T])\cap G_{(y,t)}.
$$
Measuring the length and taking the conditional expectation given $\Pi=(\Pi_{\lambda})_{\lambda>0}$, we have
\begin{align*}
\E\Big[l\big(F_{k,n}\cap (\R\times[0,T])\big)\Big|\Pi\Big]\le \sum_{(y,t)\in \Pi_{\lambda_n^*}\cap A_n}\E\Big[l\big(F_{k,n}\cap (\R\times[0,T])\cap G_{(y,t)}\big)\Big|\Pi\Big].
\end{align*}
We split $A_n$ into three regions, as shown in Figure \ref{Figure: A_n split}. Sampled points in $A_{n,1}$ are likely to contribute to $F_{k,n}$ by having ancestral lines that go above the top of the box.  Similarly, ancestral lines of sampled points in $A_{n,2}$ are likely to contribute to $F_{k,n}$ by going beyond the side of the box.  Ancestral lines of sampled points in $A_{n,3}$ are unlikely to contribute to $F_{k,n}$.  As mentioned before the lemma, the dominant contribution comes from points sampled in $A_{n,1}$.
\begin{enumerate}
\item $A_{n,1}=\{(y,t)\in A_n: t>n^{3/5}-1/\log n\}$,
\item $A_{n,2}=\{(y,t)\in A_n\setminus A_{n,1}: y<3\log^3 n\text{ or } y> n^{2/5}-3\log^3 n\}$,
\item $A_{n,3} = A_n\setminus(A_{n,1}\cup A_{n,2})$.
\end{enumerate}
\begin{figure}[h]
\centering
\begin{tikzpicture}[scale=0.3]
\draw[draw=black] (0,0) rectangle ++(24,16);
\draw[draw=black] (0,12) rectangle ++(24,4);
\draw[draw=black] (0,0) rectangle ++(6,12);
\draw[draw=black] (18,0) rectangle ++(6,12);

\node at (12, 14) {$A_{n,1}$};
\node at (3, 6) {$A_{n,2}$};
\node at (21, 6) {$A_{n,2}$};
\node at (12, 6) {$A_{n,3}$};
\node at (-1, -1) {0};
\node at (6, -1) {$3\log^3 n$ };
\node at (17, -1) {$n^{2/5}-3\log^3 n$ };
\node at (24, -1) {$n^{2/5}$ };
\node at (-5, 12) {$n^{3/5}-1/\log n$ };
\node at (-2, 16) {$n^{3/5}$ };
\end{tikzpicture}
\caption{A partition of $A_n$.}
\label{Figure: A_n split}
\end{figure}
For $(y,t)\in A_{n,1}$, we use the trivial bound
$$
\E\Big[l\big(F_{k,n}\cap (\R\times[0,T])\cap G_{(y,t)}\big)\Big|\Pi\Big]\le T,
$$
and because $A_{n,1}$ has area $n^{2/5}/(\log n)$,
\begin{equation}\label{Sampling from the entire population: B1}
\E\left[\sum_{(y,t)\in \Pi_{\lambda_n^*}\cap A_{n,1}}l\big(F_{k,n}\cap (\R\times[0,T])\cap G_{(y,t)}\big)\right]\le \frac{Tn^{2/5}}{\log n}.
\end{equation}
For $(y,t)\in A_{n,2}\cup A_{n,3}$, because $F_{k,n}=Anc_k(\Pi_{\lambda_n^*})\setminus A_n$, we have the bound
\begin{align}\label{An2 An3}
&\E\Big[l\big(F_{k,n}\cap (\R\times[0,T])\cap G_{(y,t)}\big)\Big|\Pi\Big]\nonumber\\
&\qquad=\int_0^{T-t} \P(g_{(y,t)}(u)\in F_{k,n}|\Pi )\ du\nonumber\\
&\qquad\le\int_0^{T} \min\Big\{\P\big(g_{(y,t)}(u)\in Anc_k(\Pi_{\lambda_n^*}\cap A_n)\big|\Pi\big),\ \P(g_{(y,t)}(u)\in A_n^c|\Pi )\Big\}\ du.
\end{align}
We now need to bound this integral, which we will do separately for $(y,t)\in A_{n,2}$ and $(y,t)\in A_{n,3}$.

For $(y,t)\in A_{n,2}$, suppose without loss of generality that $(y,t)$ is the left part of $A_{n,2}$, i.e., $y\le 3\log^3n$. Let $E$ be the intersection of the following events. The event $E_1$ is the event that the geodesic $g_{(y,t)}$ coalesces within $u$ time units with another geodesic to its right. The events $E_2$ and $E_3$ impose conditions on transversal fluctuations of these two geodesics. The event $E_4$ says that there are at least $k+1$ sampled points in a box between the geodesics.
\begin{enumerate}
\item $E_1$: The geodesics $g_{(y,t)}(\cdot)$ and $g_{(y+3\log^3 n, t)}(\cdot)$ coalesce by time $t+u$.
\item $E_2$: $\sup_{0\le v\le 1/\log n} g_{(y,t)}(v)< y+\log^3 n$.
\item $E_3$: $\inf_{0\le v\le 1/\log n} g_{(y+3\log^3 n,t)}(v)> y+2\log^3 n$.
\item $E_4$: $\#\big(\Pi_{\lambda_n^*}\cap([y+\log^3 n, y+2\log^3 n]\times[t,t+1/\log n])\big)\ge k+1$.
\end{enumerate}
If the events $E_1$, $E_2$, and $E_3$ occur and $u\ge 1/(\log n)$, then $g_{(y,t)}(u)$ is ancestral to all points in the box $[y+\log^3 n, y+2\log ^3n]\times[t,t+1/\log n]$. If the event $E_4$ also occurs, then $g_{(y,t)}(u)$ is ancestral to at least $k+1$ sampled points in $A_n$. See Figure \ref{Figure: An2}. Note that the events $E_1$, $E_2$, and $E_3$ are independent of $\Pi$ and the event $E_4$ is determined by $\Pi$. Therefore, if $u\ge 1/(\log n)$, then
$$
\P\Big(g_{(y,t)}(u)\in Anc_k(\Pi_{\lambda_n^*}\cap A_n)\Big|\Pi\Big)\le \P(E_1^c)+\P(E_2^c)+\P(E_3^c)+\mathbbm{1}_{E_4^c}.
$$
\begin{figure}[h]
\centering
\begin{tikzpicture}[scale=0.25]
\draw [dotted] plot coordinates {(0, 0) (30,0)};
\draw [dotted] plot coordinates {(0, 11) (30,11)};
\draw[draw=black,fill=gray!30] (12,0) rectangle ++(6,2);
\draw plot [smooth, tension=1] coordinates {(10, 0)(8,2) (11,4) (14.5, 6) (15, 8)};
\draw plot [smooth, tension=1] coordinates {(20, 0)(22,2) (18,5) (15, 8)};
\draw plot [smooth, tension=1] coordinates {(15,8) (14.5, 9) (14,11)};

\fill (10,0) circle(8pt);
\fill (20,0) circle(8pt);
\fill (14,11) circle(8pt);

\node at (10,-1) {$(y,t)$};
\node at (20,-1) {$(y+3\log ^3 n, t)$};
\node at (14,12.5) {$g_{(y,t)}(u)$};
\node at (-1,0) {$t$};
\node at (-1,11) {$t+u$};
\end{tikzpicture}
\caption{When $E_1\cap E_2\cap E_3\cap E_4$ occurs, the point $g_{(y,t)}(u)$ is the ancestor of at least $k+1$ points in the shaded box.}
\label{Figure: An2}
\end{figure}\\
By Proposition \ref{Proposition: 2-disjoint}, we have
$$
\P(E_1^c)\le C u^{-2/3}\log^3 n.
$$ 
By Theorem \ref{Theorem: transversal fluctuation}, we have
$$
\P(E_2^c)=\P(E_3^c)\le C_3  e^{-C_4\log^{11} n}.
$$
Using \eqref{An2 An3}, it follows that for $(y,t)\in A_{n,2}$,
\begin{align*}
&\E\Big[l\big(F_{k,n}\cap (\R\times[0,T])\cap G_{(y,t)}\big)\Big|\Pi\Big]\\
&\qquad\le\int_0^{T} \P\Big(g_{(y,t)}(u)\in Anc_k(\Pi_{\lambda_n^*}\cap A_n)\Big|\Pi\Big)\ du\\
&\qquad\le\frac{1}{\log n}+\int_{1/\log n}^T C u^{-2/3}\log^3 n+2C_3  e^{-C_4\log^{11} n}+\mathbbm{1}_{E_4^c}\ du\\
&\qquad\le CT^{1/3}\log^3 n+2C_3  Te^{-C_4\log ^{11}n}+T\mathbbm{1}_{E_4^c}.
\end{align*}
Let $p= (\log^2 n)/n$. Since
\begin{equation}\label{bin}
\P(E_4^c) = \sum_{i=0}^{k} {n\choose i} p^i (1-p)^{n-i}\le \sum_{i=0}^{k}  (np)^i (1-p)^{n-i}\le C e^{-\log^2 n} \log^{2k} n,
\end{equation}
summing over $(y,t)\in \Pi_{\lambda_n^*}\cap A_{n,2}$ and taking the unconditioned expectation gives
\begin{align}\label{Sampling from the entire population: B2}
&\E\left[\sum_{(y,t)\in \Pi_{\lambda_n^*}\cap A_{n,2}}l\big(F_{k,n}\cap (\R\times[0,T])\cap G_{(y,t)}\big)\right]\nonumber\\
&\qquad\qquad\le 6n^{3/5}\log^3 n\left(CT^{1/3}\log^3 n+CTe^{-\log^2 n} \log^{2k} n\right).
\end{align}

For $(y,t)\in A_{n,3}$, we write $r_1 = n^{3/5}-t$ for the distance from $(y,t)$ to the top boundary and $r_2=\min\{y, n^{2/5}-y\}$ for the distance to the left or right boundary. To bound $
\P(g_{(y,t)}(u)\in A_n^c|\Pi )$ for $t+u<n^{3/5}$, or, equivalently, for $u<r_1$, note that for such choice of $u$, the geodesic $g_{(y,t)}(\cdot)$ could only exit the box $A_n$ by hitting the left or right boundary of $A_n$. By Theorem \ref{Theorem: transversal fluctuation}, we have
\begin{equation}\label{An3 1}
\P(g_{(y,t)}(u)\in A_n^c|\Pi )\le C_3e^{-C_4 r_2^3/u^2},\qquad\text{ if } u<r_1.    
\end{equation}
To bound $\P(g_{(y,t)}(u)\in Anc_k(\Pi_{\lambda_n^*}\cap A_n)|\Pi)$, we use an argument similar to that for $(y,t)\in A_{n,2}$, except that this time, we consider three geodesics, including one on each side of $g_{(y,t)}$.  We will argue that as long as at least two of these three geodesics intersect by time $t + u$, the point $g_{(y,t)}(u)$ will be the ancestor of more than $k$ sampled points in $A_n$, and we will use Theorem~\ref{Theorem: 3-disjoint} to bound this probability.  To make this argument precise, let $E$ be the intersection of the following events and see Figure \ref{Figure: A_{n,3}} for an illustration.
\begin{enumerate}
\item $E_1$: The geodesic $g_{(y,t)}(\cdot)$ coalesces with either $g_{(y-3\log^3 n,t)}(\cdot)$ or $g_{(y+3\log^3 n, t)}(\cdot)$ by time $t+u$.
\item $E_2$: $\sup_{0\le v\le 1/\log n} |g_{(y,t)}(v)-y|< \log^3 n$.
\item $E_3$: $\sup_{0\le v\le 1/\log n} |g_{(y+3\log^3 n,t)}(v)-(y+3\log^3 n)|< \log^3 n$.
\item $E_4$: $\sup_{0\le v\le 1/\log n} |g_{(y-3\log^3 n,t)}(v)-(y-3\log^3 n)|< \log^3 n$.
\item $E_5$: $\#(\Pi_{\lambda_n^*}\cap([y+\log^3 n, y+2\log^3 n]\times[t,t+1/\log n]))\ge k+1$.
\item $E_6$: $\#(\Pi_{\lambda_n^*}\cap([y-2\log^3 n, y-\log^3 n]\times[t,t+1/\log n]))\ge  k+1$.
\end{enumerate}
\begin{figure}[h]
\centering
\begin{tikzpicture}[scale=0.3]
\draw[draw=black] (8,0) rectangle ++(24,12);
\draw[draw=black,fill=gray!30] (15.33,4) rectangle ++(2.33,2);
\draw[draw=black,fill=gray!30] (22.33,4) rectangle ++(2.33,2);
\draw plot coordinates {(8, 0) (8, 16)};
\draw plot [smooth, tension=1] coordinates {(20, 4)(22,6) (25,9) (25, 13)};
\draw plot [smooth, tension=1] coordinates {(27, 4)(27,6) (28,9) (25, 13)};
\draw plot [smooth, tension=1] coordinates {(25,13)(24,15)};
\draw plot [smooth, tension=1] coordinates {(13, 4)(12,6) (15,9) (15, 13)(16, 15)};
\draw [dotted] plot coordinates {(13, 0) (13, 4)};
\draw [dotted] plot coordinates {(20, 0) (20, 4)};
\draw [dotted] plot coordinates {(27, 0) (27, 4)};
\draw [dotted] plot coordinates {(8, 4) (32, 4)};
\draw [dotted] plot coordinates {(8, 15) (32, 15)};
\draw [dotted] plot coordinates {(8, 6) (32, 6)};

\fill (13,4) circle(6.67pt);
\fill (20,4) circle(6.67pt);
\fill (27,4) circle(6.67pt);
\fill (24,15) circle(6.67pt);

\node at (33,-1) {$n^{2/5}$};
\node at (6.5,12) {$n^{3/5}$};
\node at (13,-1) {$y-3\log^3 n$};
\node at (20,-1) {$y$};
\node at (27,-1) {$y+3\log^3 n$};
\node at (7,4) {$t$};
\node at (6.5,15) {$t+u$};
\node at (4.5,6) {$t+1/\log n$};
\node at (25, 16.5) {$g_{(y,t)}(u)$};
\end{tikzpicture}
\caption{On the event $E$, there are at least $k+1$ points sampled from each of the two shaded regions. The points in one of these two regions will be descendants of $g_{(y,t)}(u)$.}
\label{Figure: A_{n,3}}
\end{figure}
If the event $E$ occurs, then $g_{(y,t)}(u)$ is an ancestor of at least $k+1$ points in $\Pi_{\lambda_n^*}\cap A_n$, as illustrated in Figure \ref{Figure: A_{n,3}}. Therefore,
$$
\P\Big(g_{(y,t)}(u)\in Anc_k(\Pi_{\lambda_n^*}\cap A_n)\Big|\Pi\Big)\le \P(E^c|\Pi)
$$
Also, the events $E_1$, $E_2$, $E_3$ and $E_4$ are independent of $\Pi$ and the events $E_5$ and $E_6$ are determined by $\Pi$, and we have
\begin{equation}\label{An3 E}
\P\Big(g_{(y,t)}(u)\in Anc_k(\Pi_{\lambda_n^*}\cap A_n)\Big|\Pi\Big)\le \P(E_1^c)+\P(E_2^c)+\P(E_3^c)+\P(E_4^c)+\mathbbm{1}_{E_5^c}+\mathbbm{1}_{E_6^c}.  
\end{equation}
By a scaling argument using part 3 of Proposition \ref{Proposition: properties of the landscape}, we can apply Theorem \ref{Theorem: 3-disjoint} with $u/(\log ^{9/2} n)$ in place of $t$ to get
\begin{equation}\label{An3 E1}
\P(E_1^c)\le Cu^{-5/3}(\log^{15/2} n )(\log^{52} (u\vee 2)). 
\end{equation}
By Theorem \ref{Theorem: transversal fluctuation}, we have
\begin{equation}\label{An3 E234}
\P(E_2^c)=\P(E_3^c)=\P(E_4^c)\le C_3  e^{-C_4\log^{11} n}.
\end{equation}
The argument in \eqref{bin} gives 
\begin{equation}\label{An3 E56}
\P(E_5^c)=\P(E_6^c)\le C e^{-\log^2n}\log ^{2k}n.
\end{equation}
Combining \eqref{An3 E}, \eqref{An3 E1}, \eqref{An3 E234}, and \eqref{An3 E56}, we get
\begin{equation}\label{An3 2}
\P\Big(g_{(y,t)}(u)\in Anc_k(\Pi_{\lambda_n^*}\cap A_n)\Big)\le Cu^{-5/3}(\log^{15/2} n )(\log^{52} (u\vee 2))+C e^{-\log^2n}\log ^{2k}n.
\end{equation}
Using \eqref{An2 An3}, we have for $(y,t)\in A_{n,3}$,
\begin{align*}
&\E\Big[l\big(F_{k,n}\cap (\R\times[0,T])\cap G_{(y,t)}\big)\Big|\Pi\Big]\\
&\qquad\le \int_0^T\min\Big\{\P\big(g_{(y,t)}(u)\in A_n^c\big|\Pi\big),\ \P\big(g_{(y,t)}(u)\in Anc_k(\Pi_{\lambda_n^*}\cap A_n)\big|\Pi\big)\Big\}\ du\\
&\qquad\le\int_0^T \P\big(g_{(y,t)}(u)\in A_n^c\big|\Pi\big)\mathbbm{1}_{\{u\le r_1\wedge r_2^{3/2}/(\log n)\}}\ du\\
&\qquad\qquad+\int_0^T\P\big(g_{(y,t)}(u)\in Anc_k(\Pi_{\lambda_n^*}\cap A_n)\big|\Pi\big)\mathbbm{1}_{\{u> r_1\wedge r_2^{3/2}/(\log n)\}}\ du.
\end{align*}
Recall that $\Pi_{\lambda_n^*}\cap A_n$ has the same distribution as $U_n$. Therefore, if $(y,t)$ is picked uniformly at random from $\Pi_{\lambda_n^*}\cap A_n$ and conditioned to be in $A_{n,3}$, then the random variable $r_1$ is uniformly distributed over $[1/\log n, n^{3/5}]$ and $r_2$ is uniformly distributed over $[3\log^3 n,n^{2/5}/2]$. Taking the unconditioned expectation, we have
\begin{align}\label{An3 integral}
&\E\left[\sum_{(y,t)\in \Pi_{\lambda_n}^*\cap A_{n,3}}l\big(F_{k,n}\cap (\R\times[0,T])\cap  G_{(y,t)}\big)\right]\nonumber\\
&\qquad\le\int_{1/\log n}^{n^{3/5}} \int_{3\log ^3 n}^{n^{2/5}/2} \int_0^{r_1\wedge r_2^{3/2}/(\log n)} \P\big(g_{(y,t)}(u)\in A_n^c\big)\ du\ dr_2\ dr_1\nonumber\\
&\qquad\qquad+\int_{1/\log n}^{n^{3/5}} \int_{3\log ^3 n}^{n^{2/5}/2} \int_{r_1\wedge r_2^{3/2}/(\log n)}^T \P\big(g_{(y,t)}(u)\in Anc_k(\Pi_{\lambda_n^*}\cap A_n)\big)\ du\ dr_2\ dr_1.
\end{align}
We bound the first integral in \eqref{An3 integral} using \eqref{An3 1} and the second integral using \eqref{An3 2}. We have
\begin{align}\label{An3 integral1}
&\int_{1/\log n}^{n^{3/5}} \int_{3\log ^3 n}^{n^{2/5}/2} \int_0^{r_1\wedge r_2^{3/2}/(\log n)} \P\big(g_{(y,t)}(u)\in A_n^c\big)\ du\ dr_2\ dr_1\nonumber\\
&\qquad\le\int_{1/\log n}^{n^{3/5}} \int_{3\log ^3 n}^{n^{2/5}/2} \int_0^{r_1\wedge r_2^{3/2}/(\log n)}C_3e^{-C_4 r_2^3/u^2}\ du\ dr_2\ dr_1\nonumber\\
&\qquad\le \int_{1/\log n}^{n^{3/5}} \int_{3\log ^3 n}^{n^{2/5}/2}r_1 C_3e^{-C_4 \log ^2 n}\ dr_2\ dr_1\nonumber\\
&\qquad\le Cn^{8/5} e^{-C_4\log^2 n}.
\end{align}
Also,
\begin{align}\label{An3 integral2}
&\int_{1/\log n}^{n^{3/5}} \int_{3\log ^3 n}^{n^{2/5}/2} \int_{r_1\wedge r_2^{3/2}/(\log n)}^T \P\big(g_{(y,t)}(u)\in Anc_k(\Pi_{\lambda_n^*}\cap A_n)\big)\ du\ dr_2\ dr_1\nonumber\\
&\qquad\le \int_{1/\log n}^{n^{3/5}} \int_{3\log ^3 n}^{n^{2/5}/2} \int_{r_1\wedge r_2^{3/2}/(\log n)}^T \Big(Cu^{-5/3}(\log ^{15/2} n)(\log^{52}(u\vee 2))\nonumber\\
&\qquad\qquad +Ce^{-\log^2 n}\log^{2k} n\Big)\ du\ dr_2\ dr_1\nonumber\\
&\qquad\le C(\log ^{15/2} n)(\log^{52}T)\int_{1/\log n}^{n^{3/5}} \int_{3\log ^3 n}^{n^{2/5}/2}\Big( r_1^{-2/3}+r_2^{-1}(\log^{2/3}n)\Big)\ dr_2 \ dr_1\nonumber\\
&\qquad\qquad+CnTe^{-\log^2 n} \log^{2k}n\nonumber\\
&\qquad\le Cn^{3/5}(\log^{55/6}n)(\log^{52}T)+CnT e^{-\log^2 n}\log^{2k}n.
\end{align}
Combining \eqref{An3 integral}, \eqref{An3 integral1}, and \eqref{An3 integral2}, we get
\begin{align}\label{Sampling from the entire population: B3}
&\E\left[\sum_{(y,t)\in \Pi_{\lambda_n^*}\cap A_{n,3}}l\big(F_{k,n}\cap (\R\times[0,T])\cap G_{(y,t)}\big)\right]\nonumber\\
&\qquad\le C n^{3/5}(\log^{55/6} n)(\log ^{52}T)+CnT e^{-\log^2 n}\log^{2k} n.   
\end{align}
Note that if $n$ is sufficiently large and $T\ge n^{3/5}$, then the bound in \eqref{Sampling from the entire population: B1} dominates those in \eqref{Sampling from the entire population: B2}, and \eqref{Sampling from the entire population: B3}. Therefore, combining \eqref{Sampling from the entire population: B1}, \eqref{Sampling from the entire population: B2}, and \eqref{Sampling from the entire population: B3} gives the result.
\end{proof}

\subsubsection{Proof of (\ref{Sampling from the entire population: comparison2})}\label{Subsubsection: comparison2}

In this section, we prove \eqref{Sampling from the entire population: comparison2}. The idea is that if $(y,t)\in A_n$ is ancestral to points both in and outside of $A_n$, then $(y,t)$ should be close to the boundary of $A_n$. Therefore, the total length of the branches indicated on the left-hand side of \eqref{Sampling from the entire population: comparison2} will be small.

\begin{proof}[Proof of \eqref{Sampling from the entire population: comparison2}]
Let $\eps_n = (\log n)/\sqrt{n}$. Since $\P(\lambda_n^*\in[1-\eps_n,1+\eps_n])\rightarrow1$ by \eqref{Sampling from the population: Poisson deviation} as $n\rightarrow\infty$ and $Anc_{\ge k}(\Pi_{\lambda})\subseteq Anc_{\ge k}(\Pi_{\lambda'})$ whenever $\lambda<\lambda'$, it follows that as $n\rightarrow\infty$, 
$$
\P(Anc_{\ge k}(\Pi_{\lambda_n^*})\subseteq Anc_{\ge k}(\Pi_{1+\eps_n}))\rightarrow1.
$$
Therefore, we will consider the case where the sampling rate is deterministic. If $(y,t)\in A_n$  is ancestral to at least $k$ points in $\Pi_{1+\eps_n}$ and at least one point in $A_n^c$, then there are three possibilities: either $(y,t)$ is close to the boundary of $A_n$, $H(D_{(y,t)})$ is large, or $W(D_{(y,t)})$ is large. See Figure \ref{Figure An0} for an illustration. More precisely, recall the definitions of $\H_a$ and $\W_a$ from \eqref{Definition: height and width of descendants}. Let $A_{n,0}=[3n^{1/5},n^{2/5}-3n^{1/5}]\times[n^{1/5},n^{3/5}]$ be a subset of $A_n$. We have the following decomposition:
\begin{align}\label{Sampling from the entire population: decompose bubbles outside of the box}
&G_{k,n}\subseteq\Big(Anc_{\ge k}(\Pi_{\lambda})\cap (A_n\setminus A_{n,0})\Big)\cup \Big(Anc_{\ge k}(\Pi_{\lambda})\cap A_{n,0}\cap\H_{n^{1/5}}\Big)\nonumber\\
&\qquad\qquad\cup\Big(Anc_{\ge k}(\Pi_{\lambda})\cap A_{n,0}\cap\H_{n^{1/5}}^c\cap\W_{3n^{1/5}}\Big). 
\end{align}

\begin{figure}[h]
\centering
\begin{tikzpicture}[scale=0.25]
\draw[draw=black] (-10,0) rectangle ++(20,12);
\draw[draw=black] (-8,2) rectangle ++(16,10);
\draw [dashed] plot [smooth cycle,tension=1] coordinates {(-6,6) (-7,0) (-8, -2) (-5, -3) (-6, -1) (-5, 2)};
\draw [dashed] plot [smooth cycle,tension=1] coordinates {(6,6) (4,4) (3, 3) (5, 2) (9, 1) (13, 4)(9,5)};

\fill (-6,6) circle(8pt);
\fill (6,6) circle(8pt);

\node at (-11.5,12){$A_n$};
\node at (-6,10.5){$A_{n,0}$};
\node at (-6, 7.5){$(y,t)$};
\node at (6,7.5){$(y',t')$};

\end{tikzpicture}
\caption{If a point $(y,t)\in A_{n,0}$ has a descendant in $A_n^c$, then $D_{(y,t)}$ needs to be ``tall" enough (left) or ``wide" enough (right).}\label{Figure An0}
\end{figure}

For the first term on the right-hand side of \eqref{Sampling from the entire population: decompose bubbles outside of the box}, by \eqref{sfs1}, we have
\begin{equation}\label{Sampling from the entire population: bubbles near the boundary}
\E\Big[l\big(Anc_{\ge k}(\Pi_\lambda)\cap (A_n\setminus A_{n,0})\big)\Big]=\sum_{i=k}^{\infty} \left(\frac{2C_2'\lambda^{2/5}}{5}\cdot m(A_n\setminus A_{n,0}) \cdot\frac{\Gamma(i-2/5)}{i!}\right)\le C\lambda^{2/5}n^{4/5}.
\end{equation}
For the the second term on the right-hand side of \eqref{Sampling from the entire population: decompose bubbles outside of the box}, we will show in Lemma \ref{Lemma: tall bubbles}
that
\begin{equation}\label{Sampling from the entire population, tall bubbles}
\E\Big[l\big(Anc_{\ge k}(\Pi_\lambda)\cap A_{n,0}\cap\H_{n^{1/5}}\big)\Big]\le Cn^{13/15}.    
\end{equation}
For the the third term on the right-hand side of \eqref{Sampling from the entire population: decompose bubbles outside of the box}, we will show in Lemma \ref{Lemma: wide bubbles}
that if $E$ denotes the event that $Anc_{\ge k}(\Pi_\lambda)\cap A_{n,0}\cap\H_{n^{1/5}}^c\cap\W_{3n^{1/5}}$ is not empty, then
\begin{equation}\label{Sampling from the entire population, wide bubbles}
\P(E)\le C e^{-C' n^{1/5}}.
\end{equation}
Then \eqref{Sampling from the entire population: comparison2} follows from \eqref{Sampling from the entire population: decompose bubbles outside of the box}, \eqref{Sampling from the entire population: bubbles near the boundary}, \eqref{Sampling from the entire population, tall bubbles}, and \eqref{Sampling from the entire population, wide bubbles}.
\end{proof}

We now bound the second term on the right-hand side of \eqref{Sampling from the entire population: decompose bubbles outside of the box}.

\begin{lemma}\label{Lemma: tall bubbles}
Let $C_2''\in(0,\infty)$ be the constant in Remark \ref{Remark: scaling of H_a}. Then
$$
\E\Big[l\big(Anc_{\ge k}(\Pi_\lambda)\cap A_{n,0}\cap\H_{n^{1/5}}\big)\Big]\le C_2'' n^{13/15}.
$$
\end{lemma}
\begin{proof}
By Remark \ref{Remark: scaling of H_a}, for all $t\in\R$ and all measurable $\widetilde{A}\subseteq\R$, we have
$$
\E[\#(\H_{a}^t\cap \widetilde{A})] = C_2''\widetilde{m}(\widetilde{A})a^{-2/3}.
$$
Then, for any measurable $A\subseteq\R^2$, we have
$$
\E[l(\H_a \cap A)]=\int_{-\infty}^\infty \E[\#(\H_a \cap A)^t]\ dt =\int_{-\infty}^\infty C_2''\widetilde{m}(A^t)a^{-2/3}\ dt= C_2''m(A)a^{-2/3},
$$
and hence
$$
\E\Big[l\big(Anc_{\ge k}(\Pi_\lambda)\cap A_{n,0}\cap\H_{n^{1/5}}\big)\Big]\le \E\Big[l(\mathcal{H}_{n^{1/5}}\cap A_{n,0})\Big]\le n\cdot C_2''n^{-2/15}=C_2''n^{13/15},
$$
which completes the proof.
\end{proof}

We now bound the third term on the right-hand side of \eqref{Sampling from the entire population: decompose bubbles outside of the box}.

\begin{lemma}\label{Lemma: wide bubbles}
Let $E$ be the event that $Anc_{\ge k}(\Pi_\lambda)\cap A_{n,0}\cap\H_{n^{1/5}}^c\cap\W_{3n^{1/5}}$ is not empty. Then
$$
\P(E)\le C e^{-C' n^{1/5}}.
$$
\end{lemma}
\begin{proof}
Here we need to show that it is unlikely that there are any points with $k$ or more descendants whose set of descendants is both short and wide. We will argue that the existence of such a point would require geodesics to have unusually large transversal fluctuations. Consider the grid point $(i,j)$ for $i\in \mathcal{I}=\{0,n^{1/5},2n^{1/5},\dots, \lceil n^{1/5}\rceil n^{1/5}\}$ and $j\in\mathcal{J}=\{0,n^{1/5},2n^{1/5},\dots,\lceil n^{2/5}\rceil n^{1/5}\}$. For each $(i,j)$, let $E_{i,j}$ be the event that 
$$
\sup_{0\le u\le 2n^{1/5}} |g_{(i,j)}(u)-i|\ge \frac{1}{3}n^{1/5}.
$$
We claim that 
\begin{equation}\label{Sampling from the entire population, grid}
E\subseteq\bigcup_{i\in\mathcal{I},j\in\mathcal{J}} E_{i,j}.   
\end{equation}
If \eqref{Sampling from the entire population, grid} holds, then the result follows from Theorem \ref{Theorem: transversal fluctuation}.

We now prove \eqref{Sampling from the entire population, grid}. Suppose the event $E$ occurs. Let $(y_0,t_0)\in A_{n,0}$ be a point such that $H(D_{(y_0,t_0)})\le  n^{1/5}$ and $W(D_{(y_0,t_0)})> 3n^{1/5}$. Let $i_0:=\sup\{i\in\mathcal{I}: i+n^{1/5}/3\le y_0\}$, $i_1:=\inf\{i\in\mathcal{I}: i-n^{1/5}/3\ge y_0\}$
and $j_0:=\sup\{j\in\mathcal{J}: j+n^{1/5}\le t_0\}$. Note that since $(y_0,t_0)\in A_{n,0}$ is away from the boundary, the quantities $i_0$, $i_1$, and $j_0$ are well-defined and we have $i_1-i_0\le 2n^{1/5}$. Suppose that neither $E_{i_0,j_0}$ nor $E_{i_1,j_0}$ holds, that is
\begin{equation}\label{Sampling from the entire population, transversal fluctuation on the grid}
\sup_{0\le u\le 2n^{1/5}} |g_{(i_0,j_0)}(u)-i_0|< \frac{1}{3}n^{1/5},\qquad\sup_{0\le u\le 2n^{1/5}} |g_{(i_1,j_0)}(u)-i_1|<\frac{1}{3}n^{1/5}.
\end{equation}
The assumption $H(D_{(y_0,t_0)})\le n^{1/5}$ implies that $D_{(y_0,t_0)}\subseteq\R\times[j_0,t_0]$. Then, \eqref{Sampling from the entire population, transversal fluctuation on the grid} implies that
$$
W(D_{(y,t)})\le i_1-i_0+\frac{2}{3}n^{1/5}\le 3n^{1/5},
$$
which contradicts $W(D_{(y,t)})>3n^{1/5}$.
\end{proof}

\section{The three-arm event}\label{Section: 3-arm}

We finally arrive at the proof of the three-arm probability bound. 
For convenience let us restate an equivalent version of the statement of the theorem as also recorded in \eqref{scaled123}.
\newcommand{\e}{\varepsilon}

\begin{theorem}\label{threearm}
Consider the infinite upward geodesics $g_{(-\e,0)}$, $g_{(0,0)}$, and $g_{(\e,0)}$. Then there exists a constant $C\in(0,\infty)$ such that, for all $\e \le 1/2,$
\begin{equation}\label{threearm1}
\P(g_{(-\e,0)}, g_{(0,0)}, \text{ and } g_{(\e,0)} \text{ are disjoint for all } s\in[0,1])\le C\e^{5/2}\log^{52} (1/\eps).
\end{equation}
\end{theorem}

Continuing the discussion following \eqref{scaled123}, one can indeed predict that the exponent $5/2$ is the right one by reducing to events about the parabolic Airy line ensemble which is indeed how our proof will proceed. We now include a short review of the basics of the parabolic Airy LE that will be needed for our purpose and which will allow us to explain the source of the $5/2$ exponent. 

\newcommand{\cA}{\mathcal{P}}
\newcommand{\ocA}{\overline{\mathcal P}}

\subsection{The parabolic Airy line ensemble and its properties}  \label{ss:ale}

The parabolic Airy line ensemble (LE) $\cA=\{\cA_n\}_{n=1}^\infty$ is a family of random processes on $\R$ which are ordered, and was constructed in \cite{ch14}.
Some crucial properties are presented next. Define $\ocA_n(x):=\cA_n(x)+x^2$ for any $x\in \R$, $n\in \N$.
\begin{enumerate}
    \item Stationarity: $\ocA=\{\ocA_n\}_{n=1}^\infty$ is stationary with respect to horizontal shift.
   \item Ergodicity: $\ocA=\{\ocA_n\}_{n=1}^\infty$ is ergodic with respect to horizontal shift.
   \item Brownian Gibbs property: for any $m\le n\in \N$ and $x\le y$, let $S = \llbracket m,n\rrbracket \times [x, y]$.
 $\cA$, as a (random) function on $\N\times \R$, admits the following local Gibbs specification: the conditional distribution of the process $\cA|_S$ given $\cA|_{S^c}$ is just $n-m + 1$ Brownian bridges with variance parameter 2 connecting up the points $\cA_i(x)$ and $\cA_i(y)$, for $m\le i \le n$, conditioned on a global non-intersection, i.e., $\cA_i(z)>\cA_{i+1}(z)$ for any $1\vee (m-1) \le i \le n$ and $z\in [x,y]$.
\end{enumerate}

\newcommand{\boo}{\mathbf 0}
\newcommand{\cL}{\mathcal{L}}
\newcommand{\cB}{\mathcal{B}}

The first property is from its construction, and the second property is proved in \cite{cs14}. The third property is proved in \cite{ch14}.

The first line $\cA_1$ (of $\cA$) is the parabolic Airy$_2$ process on $\R$, and equals in law to the function $x\mapsto \cL(\boo;x,1)$, where we write $\boo = (0,0)$.
A particularly important property of the parabolic Airy LE is a strong form of local Brownianity that it admits which helps reduce estimating probabilities of several events to tractable Brownian computations. This has formed the basis of several recent advances. While a result of this form first appeared in \cite{chh19}, several strengthenings have followed. For our application we will use the following version recently obtained in \cite[Theorem 3.8]{d24w}. 

\begin{theorem}\label{duncan12} Fix $t\ge 1$ and $k, \ell \in \N$.  Fix $\mathbf{a} \in \R^{\ell}$ such that if $\ell \ge 2$, then $a_j+t < a_{j+1}$ for all $j \in \llbracket 1, \ell-1 \rrbracket.$ Define $U(\mathbf{a})=\bigcup_{j=1}^\ell(\llbracket 1,k \rrbracket \times (a_{j}, a_{j}+t))$. Then there exists a random sequence of continuous functions $\mathfrak{L}^{\mathbf{a}}=\mathfrak{L}^{t,k,\mathbf{a}}=\{\mathfrak{L}_{i}^{\mathbf{a}}: \mathbb{R} \to \mathbb{R}, i \in \N\}$ such that the following hold:
\begin{enumerate}
\item Almost surely, $\mathfrak{L}^{\mathbf{a}}_i(r)> \mathfrak{L}^{\mathbf{a}}_{i+1}(r)$ for all pairs $(i,r) \notin U(\mathbf{a}).$

\item The ensemble has the following Gibbs property: For any $m\in \N$ and $a<b$, consider the set $S=\llbracket 1,m \rrbracket \times [a,b].$ Then, conditional on the values of $\mathfrak{L}^{\mathbf{a}}_i(r)$ for $(i,r) \notin S$, the distribution of $\mathfrak{L}^{\mathbf{a}}_i(r)$ for $(i,r) \in S$ is given by $m$ independent Brownian bridges $B_1, B_2, \ldots, B_m$ with variance parameter 2 from $(a, \mathfrak{L}_i^{\mathbf{a}}(a))$ to $(b, \mathfrak{L}_i^{\mathbf{a}}(b))$ for $i \in \llbracket 1,m \rrbracket,$ conditioned on the event $B_i(r)> B_{i+1}(r)$ whenever $(i,r) \notin U(\mathbf{a}).$

\item We have $$\P(\mathfrak{L}_1^{\mathbf{a}}> \mathfrak{L}_2^{\mathbf{a}}> \ldots) \ge e^{-c_{k,\ell}^* t^3}$$ where $c_{k,\ell} > 0$ depends only on $k$ and $\ell$.  Also, conditional on the event $\{\mathfrak{L}_1^{\mathbf{a}}> \mathfrak{L}_2^{\mathbf{a}}> \ldots\},$ the ensemble $\mathfrak{L}^{\mathbf{a}}$ is equal in law to the parabolic Airy LE.
\end{enumerate}
\end{theorem}

In words, the above says that given a domain $U$, there is a line ensemble $\mathfrak{L}^{\mathbf{a}}$ for which the Gibbs property on $U$ simply involves Brownian bridges without the non-intersection constraints, hence making computations about events on $U$ simpler.  On the other hand, conditional on non-intersection everywhere (which happens with reasonably large probability) the line ensemble has the distribution of the parabolic Airy LE thus allowing one to transfer estimates to the latter.

\subsection{Geodesic watermelons}
A key object in our analysis will be geodesic watermelons. 
Informally, geodesic watermelons are disjoint optimizers, i.e., a $k$-geodesic watermelon is a collection of $k$ distinct paths whose cumulative energy is the maximum. Note that the $1$-watermelon is simply a geodesic. See \cite{dz24d} for formal definitions, and arguments showing their existence and uniqueness in the directed landscape. 

We will use the following key fact from \cite{dz24d} which is a consequence of the RSK correspondence.  To state it we introduce the notion of a Gelfand-Tsetlin pattern. 
For $k\in \N$ let $T_k = \{(i, j) \in \N^2: i + j \le k + 1\}$.
We say that $w \in \R^{T_k}$ is a $k$-level Gelfand-Tsetlin pattern if it satisfies the interlacing
inequalities $w_{i,j} \le w_{i+1,j} \le w_{i,j+1}$
for all $i+j \le k$, and let $GT_k$ be the space of $k$-level Gelfand-Tsetlin patterns. For $x \in \R^{k+1}_{\le},$
we let $GT_k(x)$ be the set of all $w\in GT_k$ such that $(w,x) \in GT_{k+1}$ where $(w,x)_{i,j} = w_{i-1,j}$
for $i \ge 2$ and $(w,x)_{1,j} = x_j$.
 
While the top line in the parabolic Airy LE traces out $\cL({\bf 0};x,1)$ (recall that $\bf0$ is used to denote $(0,0)$), the lower lines allow access to disjointness of geodesics via the following relation. 

\newcommand{\cS}{\mathcal{S}}

\begin{proposition}[\cite{dz24d}, Proposition 5.9]\label{GTprop}
The directed landscape $\cL$ and the parabolic Airy line ensemble $\cA$ can be coupled so that $\cA_1(x) = \cL(\boo;x,1)$ for all $x \in \R$ and the following holds.
For a vector $\mathbf{x}= (x_1, x_2, x_3)$ with $x_1 < x_2 < x_3$, denote by $$\gamma(\mathbf{x})=(\gamma_1(x_1),\ldots,\gamma_3 (x_3))$$ the geodesic watermelon from $\boo$ to $(\mathbf{x},1)$, whose existence is the one of the main results in \cite{dz24d}.  For brevity, write $\cL(\gamma(\mathbf{x})) = \| \gamma_1 \|_{\mathcal{L}} + \| \gamma_2 \|_{\mathcal{L}} + \| \gamma_3 \|_{\mathcal{L}}.$  Then 
\begin{equation}\label{GTpattern}
\cL(\gamma(\mathbf{x})) = \sum_{i=1}^3 \cL(\boo; x_i,0)-\inf_{ w\in GT_{2}(x) }\sum_{i=1}^2 \sum_{j=1}^{3-i} \big( \cA_i(w_{i,j}) - \cA_{i+1}(w_{i,j}) \big).
\end{equation}
\end{proposition}

Note that indeed the LHS is bounded above by the first term of the RHS (since the geodesic watermelons have energies bounded by that of the corresponding geodesics) and the second term quantifies the deficit. Further, each of the summands in the second term is strictly positive on account of the non-intersection property of $\cA.$  Note also that the infimum on the RHS can be expressed as
$$\inf_{\substack{x_1 < w_{11} < x_2 < w_{12} < x_3 \\ w_{11} < w_{21} < w_{12}}} \Big[\big(\mathcal{P}_1(w_{11}) - \mathcal{P}_2(w_{11})\big) + \big(\mathcal{P}_2(w_{21}) - \mathcal{P}_3(w_{21})\big) + \big(\mathcal{P}_1(w_{12}) - \mathcal{P}_2(w_{12})\big)\Big].$$

Now on the event that that there are disjoint geodesics from $\boo$ to $(\mathbf{x},1),$ the error term must be zero. This corresponds to exact touches between $\cA_i$ and $\cA_{i+1}$ at the locations $w_{i,j}.$ This is a zero probability event since $\cA$ is non-intersecting, but we will work with an approximate version of this event where instead of exactly touching, the curves will come close to each other. Further, note that while the above only considers finite geodesics going from $\boo$ to $(\mathbf{x},1)$, we are interested in the further constraint that $\gamma_1,\gamma_2, \gamma_3$ are initial segments of semi-infinite geodesics. We next see how to incorporate this constraint.  Recall the Busemann function $\cB(x,s)$ from \eqref{busemann1}, which should be thought of as a properly centered version of the distance to infinity, in the vertical direction.  For convenience, we will use a modified definition of the Busemann function in this section, with $\mathcal{L}(0,0;0,n)$ on the right-hand side of \eqref{busemann1} replaced by $\mathcal{L}(0,1;0,n)$; note that this change only adds a constant shift to the Busemann function.  When $s=1$, we will denote $\cB(x,1)$ by $\cB(x).$  We also denote $\cL(0,0;x,1)$ by $\mathcal{L}(x)$.
By Theorem \ref{Theorem: Brownian motion}, the condition that the components of $\gamma$
are initial segments of semi-infinite geodesics is equivalent to the condition that for all $i=1,2,3,$
\begin{equation}\label{argmax}
\cB(x_i)+\cL(x_i)=\max_{x\in \R} \big(\cB(x)+\cL(x) \big)
\end{equation}
and hence in particular,
\begin{equation}\label{argmax1}
\cB(x_1)+\cL(x_1)=\cB(x_2)+\cL(x_2)=\cB(x_3)+\cL(x_3).
\end{equation}
Importantly, $\cB$ is independent of $\cL(\cdot,0;\cdot,1)$ which allows one to decouple the events in \eqref{GTpattern} and \eqref{argmax}.
We are now in a position to explain the source of the exponent $5/2$ appearing in the statement of Theorem \ref{threearm}.

Note that for the error term in \eqref{GTpattern} to be zero each of the three non-negative terms must be zero. Further \eqref{argmax1} stipulates that $\cB(x_1)+\cL(x_1)-\cB(x_2)-\cL(x_2)$ and $\cB(x_2)+\cL(x_2)-\cB(x_3)-\cL(x_3)$ are zero. 
In our argument all of the zeros will be replaced by $\sqrt{\e}$ (for reasons that will be apparent shortly). Thus, our event $\ncg$ consists of five events which, ignoring a certain degree of correlation between them, all require two independent Brownian motions to come within $\sqrt \e$ of each other on a unit order interval along with certain non-intersection constraints. It will turn out that the probability of each of these events will be of the order $\sqrt{\e}$.  The points will also typically be well separated, implying that the correlation between them will be weak enough to yield the $\e^{5/2}$ bound. 

However, a formal argument must address situations where the points are close to each other, increasing the correlation between the events. Nonetheless, the contributions of such situations will turn out to be negligible, albeit proving which turns out to be somewhat complicated.  We now proceed to making the above reasoning precise.  

The first point we need to address is that in Theorem \ref{threearm}, the geodesics do not emanate from the same point but rather are spaced $\eps$ apart.  For convenience of writing, instead of starting from height $0$, we will instead consider the three infinite upward geodesics emanating from the points $(-\eps, \eps^{3/2})$, $(0, \eps^{3/2})$, and $(\eps, \eps^{3/2})$, i.e., their starting points are shifted upward by $\eps^{3/2}$.  Let us call them $\gamma_1,\gamma_2,\gamma_3.$  A simple scaling argument shows that it suffices to prove the upper bound in \eqref{threearm1} for the geodesics $\gamma_1,\gamma_2,\gamma_3$. Towards this, we will perform a, by now, standard surgery argument to tie them all to the point $(0,0)$ at the cost of losing the geodesic status but being a near geodesic. Such an argument had first appeared in \cite{h20}.   We define three relevant events below, the first of which will be pertinent for this surgery argument.

\begin{itemize}
\item There is a constant $C_6 > 0$ such that with probability greater than $1/2$ one can find three disjoint curves $\eta_1,\eta_2,\eta_3$  joining $(0,0)$ and $(-\e,\eps^{-3/2}), (0,\eps^{-3/2})$, and $(\e,\eps^{-3/2})$ respectively, all of whose lengths are between $[-C_6 \sqrt{\e}, C_6 \sqrt{\e}]$. This can be seen, for instance, by scaling space by $\eps^{-1}$ and time by $\eps^{-3/2}$ and appealing to \eqref{GTpattern} (see also \cite{h20} where a different argument ensuring this was employed). Let us call this event $\mathcal{S}_1.$  On $\mathcal{S}_1$, let the paths obtained by concatenating $\gamma_i$ and $\eta_i$ be denoted by $\pi_i.$
Thus the lengths of $\pi_i$ and $\gamma_i$ differ by no more than $C_6 \sqrt \e.$

\item Let $\mathcal{S}_2$ be the event that the intersections of the semi-infinite geodesics emanating from $(0,0)$, $(-\e,\eps^{3/2}), (0,\eps^{3/2})$, and $(\e,\eps^{3/2})$ with the line $y=1$ all lie in the interval\\ $[-\log^{1/2}(1/\eps), \log^{1/2}(1/\eps)].$  By Theorem \ref{Theorem: transversal fluctuation}, for sufficiently small $\eps$, we have $\P(\mathcal{S}_2)\ge 1-\e^{10}.$

\item Let $\cS_3$ be the event that for each $x \in [-\log^{1/2}(1/\eps), \log^{1/2}(1/\eps)]$ and $i\in \{-1,0,1\},$  we have
$$|\cL(0,0; x,1)-\cL(\e i, \eps^{3/2}; x,1)| \leq \sqrt{\e}\log^{3/2}(1/\eps).$$
It is possible to reduce the power of the logarithm, but this definition will be sufficient for our purposes.
\end{itemize}

\begin{lemma}
For sufficiently small $\eps$, we have $\P(\cS_3)\ge 1-\e^{10}.$
\end{lemma}

\begin{proof} 
By an application of Theorem \ref{Theorem: transversal fluctuation} for finite geodesics, all the geodesics $g_{(0,0; x,1)}$ with $|x|\le {\log^{1/2}(1/\eps)}$ intersect the line $y = \eps^{3/2}$ at points $z$ with absolute value no more than $\eps \log^{1/2}(1/\eps)$, with probability at least $1 - \frac{1}{3} \eps^{10}$.  By the triangle inequality,
$$|\cL(0,0; x,1)-\cL(\eps i, \eps^{3/2}; x,1)|\le |\cL(0,0; z,\eps^{3/2})|+|\cL(\eps i, \eps^{3/2}; x,1)-\cL(z,\eps^{3/2}; x,1)|.$$
To bound the first term on the RHS, note that by scaling, $\cL(0,0; z,\eps^{3/2})$ has the same law as $\sqrt{\eps} \cL(0,0; \eps^{-1} z, 1)$.  Theorem \ref{hold432} then establishes that when $\eps$ is sufficiently small, we have $\mathcal{L}(0, 0; \eps^{-1} z, 1) \leq 2 \log(1/\eps)$ with probability at least $\frac{1}{3} \eps^{10}$; indeed the random variable $R$ in Theorem \ref{hold432} has the property that $\P(R > \log^a(1/\eps))$ decays faster than any power of $\eps$ as $\eps \rightarrow 0$ as long as $a > 2/3$. 
The second term can be bounded by applying Theorem \ref{modcont} with $\tau = 0$ and $\xi = |\eps i - z|$.
\end{proof}

Armed with the above, we state the following proposition. 

\begin{proposition}\label{Airyprop}
Let $\mathcal{P}_1$, $\mathcal{P}_2$, and $\mathcal{P}_3$ denote the top three lines of the parabolic Airy line ensemble.  Let $(W(t), t \in \R)$ be two-sided Brownian motion with variance parameter 2, independent of the parabolic Airy line ensemble.  Let $T = \lfloor \log^{1/2}(1/\eps) + 2 \rfloor$.  Let $A$ be the event that there exist times $a$, $b$, and $c$ with $-\log^{1/2}(1/\eps) \leq a < b < c \leq \log^{1/2}(1/\eps)$ such that the following hold:
\begin{itemize}
\item We have $\mathcal{P}_1(a) - \mathcal{P}_2(a) < \eps$, $\mathcal{P}_1(c) - \mathcal{P}_2(c) < \eps$, and $\mathcal{P}_2(b) - \mathcal{P}_3(b) < \eps$.

\item Let $$m_1 = \max_{-T \leq t \leq a} (W(t) + \mathcal{P}_1(t)), \quad m_2 = \max_{a \leq t \leq c} (W(t) + \mathcal{P}_1(t)), \quad m_3 = \max_{c \leq t \leq T} (W(t) + \mathcal{P}_1(t)).$$  We have $|m_1 - m_2| < \eps$, $|m_1 - m_3| < \eps$, and $|m_2 - m_3| < \eps$.
\end{itemize}
Then there exists a constant $C > 0$ such that $\P(A) \leq C \eps^{5} \log^{44}(1/\eps)$ for sufficiently small $\eps$.
\end{proposition}

Theorem \ref{threearm} is therefore an immediate consequence of the following lemma and an elementary scaling argument.  Proposition \ref{Airyprop} and Lemma \ref{reduction} give a power of the logarithm of $44 + 5(3/2) = 51.5$, which we round up to 52 in Theorem \ref{threearm}, as we do not aim to optimize the exponent.

\begin{lemma}\label{reduction}
Let $A$ be the event defined in Proposition \ref{Airyprop}, when we put $4 \sqrt{\eps} \log^{3/2} (1/\eps)$ in place of $\eps$.
Let $\ncg$ be the event defined following the statement of Theorem \ref{threearm}, but with the geodesics starting from height $\eps^{3/2}$ instead of $0$.
Then for sufficiently small $\eps$, we have
$\P(\ncg) \leq 2 \P(A)+ 4\e^{10}$.
\end{lemma}

\begin{proof}
Recall the events $\cS_1$, $\cS_2$, and $\cS_3$, and note that $\ncg$ is independent of $\cS_1$ and that $x \mapsto \cL({\bf 0};x,1)$ has the same law as $\cA_1$.  Also, by Theorem \ref{Theorem: Brownian motion}, the process $x \mapsto \cB(x)$ is a two-sided Brownian motion with variance parameter 2, which is independent of the process $x \mapsto \cL({\bf 0};x)$.  Therefore, we can set $\cA_1(x) = \cL(\boo;x,1)$ and $W(x) = \cB(x)$.
We now claim that $\ncg\cap \cS_1\cap \cS_2\cap \cS_3 \subseteq A.$ That this immediately implies the lemma follows by observing that $\P(\ncg\cap \cS_1\cap \cS_2\cap \cS_3)\ge \frac{1}{2}\P(\ncg)-2\e^{10}$.

It remains to prove the above claim. 
Note that on $\cS_1\cap \cS_2\cap \cS_3\cap \ncg$, the portions of the paths $\pi_1,\pi_2,\pi_3$ up to height $1$ are disjoint paths originating from $(0,0)$ and ending at $(x_1, 1)$, $(x_2, 1)$, and $(x_3, 1)$.  The lengths of these disjoint paths differ from $\cL(0,0; x_i, 1)$ by at most $\sqrt{\e}\log^{3/2}(1/\eps)  + C_6 \sqrt{\eps}$.
By Proposition \ref{GTprop}, it follows that for sufficiently small $\eps$, there exist $x_1< w_{11}< x_2 < w_{12} < x_3$ and $w_{11}<w_{21}< w_{12}$ such that 
\begin{align}\label{neartouch1}
|\cA_1(w_{11})-\cA_2(w_{11})| \leq 4 \sqrt{\e}\log^{3/2}(1/\eps), \quad & |\cA_1(w_{12})-\cA_2(w_{12})|\leq 4 \sqrt{\e}\log^{3/2}(1/\eps), \\
\nonumber
|\cA_2(w_{21})-\cA_3(w_{21})| \leq &\,\,\, 4\sqrt{\e}\log^{3/2}(1/\eps).
\end{align}
Because on $\cS_{2}$, all of the $x_{i}$ have absolute value at most $\log^{1/2}(1/\e)$, we can therefore take $w_{11}=a, w_{12}=c$, and $w_{21}=b$ to see that the first condition in the definition of $A$ is satisfied.

It remains to check the second condition in the definition of $A$ by showing that approximate versions of \eqref{argmax} and \eqref{argmax1} hold.  For $x \in \R$, we write $\cL({\bf 0};x)$ in place of $\cL({\bf 0}; x,1)$.  Let $s=\argmax_{x} (\cB(x)+\cL({\bf 0}; x))$.  On $\cS_{2}$, $s$ has absolute value at most $\log^{1/2}(1/\e).$ Then for $i=1,2,3,$ we observe, using twice the definition of $\cS_3$, that
\begin{align*}
\cB(x_i) + \cL({\bf 0}; x_i) + \sqrt{\eps} \log^{3/2}(1/\eps) &\geq \cB(x_i) + \cL(\eps i, \eps^{3/2}; x_i, 1) \\
&\geq \cB(s) + \cL(\eps i, \eps^{3/2}; s, 1) \\
&\geq \cB(s) + \cL({\bf 0}; s) - \sqrt{\eps} \log^{3/2}(1/\eps).
\end{align*}
Therefore,
\begin{equation}\label{argmax2}
|\cB(x_i)+\cL({\bf 0}; x_i)-\cB(s)-\cL({\bf 0}; s)| \leq 2 \sqrt{\e}\log^{3/2}(1/\eps).
\end{equation}
That is, the maximum value of $\cB(x)+\cL({\bf 0}; x)$ is nearly attained at the points $x_1$, $x_2$, and $x_3$.  Recalling that $\mathcal{P}_1(x) = \cL({\bf 0}; x)$ and $W(x) = \cB(x)$,
it follows that the second condition in the definition of $A$ is satisfied, as the values of $m_1$, $m_2$, and $m_3$ in Proposition \ref{Airyprop} are all within $2 \sqrt{\e}\log^{3/2}(1/\eps)$ of the global maximum.
\end{proof}

\newcommand{\PP}{\mathbb{P}}
Thus it remains to prove Proposition \ref{Airyprop}.The argument has many pieces and the entirety of the remainder of the paper is devoted to this. 

\subsection{Proof of Proposition \ref{Airyprop}}

We begin by collecting several lemmas which will be useful in the proof.  The first lemma is a standard formula for the maximum of a Brownian bridge, which can be found, for example, on page 63 of \cite{borsal}.

\begin{lemma}\label{maxbridge}
Let $(B_s)_{0 \leq s \leq t}$ be a standard Brownian bridge going from $0$ to $z$ in time $t$.  For $y\geq \max\{0,z\}$, we have
$$\P \bigg( \sup_{0 \leq s \leq t} B_s \leq y \bigg) = 1 - \exp \bigg(- \frac{2y(y-z)}{t} \bigg) \leq \frac{2y(y-z)}{t}.$$
\end{lemma}

The second is an entropic repulsion estimate of a Brownian excursion. 
\begin{lemma}\label{onebridge}
Let $(X_t)_{0 \leq t \leq  1}$ be a standard Brownian bridge from $x_1$ to $x_2$, where $x_1 > 0$ and $x_2 > 0$.  Let $\delta \in (0, 1/2)$.  Then there exists a positive constant $C_7$, depending on $\delta$ but not on $x_1$ or $x_2$, such that for all $\eps > 0$ and all $t \in [\delta, 1 - \delta]$, we have $$\P(X_t < \eps, X_s > 0 \textup{ for all } s \in [0,1]) \leq C_7 \eps^3.$$
\end{lemma}

\begin{proof}
We have
\begin{equation}\label{condbridge}
\P(X_t < \eps, X_s > 0 \textup{ for all } s \in [0,1]) \leq \P(X_t < \eps \,|\, X_s > 0 \textup{ for all } s \in [0,1]).
\end{equation}
The Brownian bridge from $x_1$ to $x_2$ conditioned to stay positive is a three-dimensional Bessel bridge from $x_1$ to $x_2$.  The three-dimensional Bessel bridge from $0$ to $0$ is the same as a standard Brownian excursion $(B^{ex}_t)_{0 \leq t \leq 1}$.  Therefore using the monotonicity coming, for example, from the construction of the Bessel bridge in (0.22) of \cite{pitcsp}, we get that the probability on the right-hand side of \eqref{condbridge} is bounded above by $\P(B^{ex}_t < \eps)$.  It is well-known (see, for example, page 76 of \cite{itomckean}) that $B^{ex}_t$ has probability density function $$f_t(x) = \frac{2x^2}{\sqrt{2 \pi t^3(1-t^3)}} e^{-x^2/(2t(1-t))}, \qquad x > 0.$$
Integrating this formula over $x$ from $0$ to $\eps$, we get that $\P(B^{ex}_t > \eps) = O(\eps^3)$, which completes the proof.
\end{proof}

We now consider independent Brownian bridges conditioned not to intersect, which we call nonintersecting Brownian bridges.
The next lemma records various near touch events for nonintersecting Brownian bridge ensembles.

\begin{lemma}
Let $(X_t)_{0 \leq t \leq 1}$, $(Y_t)_{0 \leq t \leq 1}$, and $(Z_t)_{0 \leq t \leq 1}$ be nonintersecting standard Brownian bridges going from $x_1$ to $x_2$, from $y_1$ to $y_2$, and from $z_1$ to $z_2$ respectively. Suppose $x_1 > y_1 > z_1$ and $x_2 > y_2 > z_2$.  Let $\delta \in (0, 1/2)$.  Then there exist positive constants $C_8$ and $C_9$, depending only on $\delta$, such that for all $\eps > 0$ and all $t \in [\delta, 1 - \delta]$, we have
\begin{equation}\label{twoclose}
\P(X_t - Y_t < \eps) \leq C_8 \eps^{3}
\end{equation}
and
\begin{equation}\label{threeclose}
\P(X_t - Y_t < \eps, Y_t - Z_t < \eps) \leq C_9 \eps^8.
\end{equation}
\end{lemma}
While we will not require and hence not pursue lower bounds, the above bounds are in fact sharp.  

\begin{proof}
It was proved in Proposition 5.1 of \cite{ah23} that the gap processes in an ensemble of nonintersecting Brownian bridges are monotone in the endpoint gaps.  As a consequence, the probabilities in \eqref{twoclose} and \eqref{threeclose} are largest when $x_1 = y_1 = z_1 = x_2 = y_2 = z_2 = 0$, so it suffices to prove the bound in this case.  Then the process $(X_t, Y_t, Z_t)_{0 \leq t \leq 1}$ is known as the Brownian 3-watermelon, and the joint density of $(X_t, Y_t, Z_t)$ can be found using the formula of Karlin and McGregor \cite{km59}.  As noted, for example, in Remark 1.3 of \cite{LDA24}, the joint density is given by
$$f_t(x,y,z) = c_t (x-y)^2 (x-z)^2 (y-z)^2 \exp \bigg(- \frac{x^2 + y^2 + z^2}{2t(1-t)} \bigg) {\bf 1}_{\{x > y > z\}},$$
where $c_t$ is a constant depending on $t$.  Using that $x^2 + y^2 + z^2 \geq x^2$ and that, when $0 < x-y < \eps$ and $0 < y-z < \eps$, we have $(x-y)^2 (x-z)^2 (y-z)^2 \leq 4 \eps^6$, we get
\begin{align*}
\P(Y_t - X_t < \eps, Z_t - Y_t < \eps)
&= \int_{y - \eps}^y \int_{x - \eps}^x \int_{-\infty}^{\infty} f_t(x,y,z) \: dx \: dy \: dz \\
&\leq 4c_t \eps^6 \int_{y - \eps}^y \int_{x - \eps}^x \int_{-\infty}^{\infty} \exp \bigg(- \frac{x^2}{2t(1-t)} \bigg) \: dx \: dy \: dz \\
&= 4 c_t \eps^8 \int_{-\infty}^{\infty} \exp \bigg(- \frac{x^2}{2t(1-t)} \bigg) \: dx.
\end{align*}
Because the integral is finite, this expression is bounded above by $C_9 \eps^8$, proving \eqref{threeclose}.
Likewise,
\begin{align*}
\P(Y_t - X_t < \eps) &= \int_{-\infty}^y \int_{x - \eps}^x \int_{-\infty}^{\infty} c_t (x-y)^2 (x-z)^2 (y-z)^2 \exp \bigg(- \frac{x^2 + y^2 + z^2}{2t(1-t)} \bigg) \: dx \: dy \: dz.
\end{align*}
This is easily seen to be $O(\eps^3)$ because $(x - y)^2 \leq \eps^2$ and the middle integral is over an interval whose length is $\eps$.  The result \eqref{twoclose} follows.
\end{proof}

The next lemma is a H\"older regularity estimate for the top three lines of the parabolic Airy line ensemble, which follows from Theorem 1.5 of \cite{dv21a}. Note that the regularity of the top line was already quoted in Theorem \ref{modcont}.

\begin{lemma}\label{AiryHolder}
Let $(W(t), t \in \R)$ be a two-sided Brownian motion with variance parameter 2, and let $\mathcal{P}_1$, $\mathcal{P}_2$, and $\mathcal{P}_3$ denote the top three lines of the parabolic Airy line ensemble.  Let $T = \lfloor \log^{1/2}(1/\eps) + 2 \rfloor$.  There exists a constant $C_{10} > 0$ such that with probability at least $1 - \eps^{100}$, for all $s, t \in [-T, T]$ with $\eps^{100} \leq |t - s| \leq 1$, we have
$$\max\{|W(t) - W(s)|, |\mathcal{P}_1(t) - \mathcal{P}_1(s)|, |\mathcal{P}_2(t) - \mathcal{P}_2(s)|, |\mathcal{P}_3(t) - \mathcal{P}_3(s)|\} \leq C_{10} \log(1/\eps) \sqrt{|t - s|}.$$
\end{lemma}

\begin{proof}
Theorem 1.5 of \cite{dv21a} states that there are positive constants $c$ and $d$ such that for all positive integers $k$ and all $s, t \in [0,1]$, we have
$$|\mathcal{P}_k(s) - \mathcal{P}_k(t)| \leq R_k |t - s|^{1/2} \log^{1/2}\bigg( \frac{2}{|t-s|} \bigg),$$
where $R_k$ is a random constant satisfying $\P(R_k > m) \leq e^{ck - dm^2}$ for all $m > 0$.  
Therefore, by choosing $m = C' \log^{1/2}(1/\eps)$ for a sufficiently large constant $C'$, we see that there is a positive constant $C$ such that, with probability at least $1 - \eps^{101}$, for $k \in \{1, 2, 3\}$ and all $s,t \in [0,1]$ with $\eps^{100} \leq |t - s| \leq 1$,
we have
\begin{equation}\label{Pkdiff}
|\mathcal{P}_k(s) - \mathcal{P}_k(t)| \leq C |t - s|^{1/2} \log(1/\eps).
\end{equation}
Standard modulus of continuity estimates imply that this result also holds with Brownian motion $W$ in place of $\mathcal{P}_k$ (this is already noted in \cite{dv21a}).

As observed in \cite{dv21a}, we can extend this result beyond the interval $[0,1]$ by using the stationarity of $(\mathcal{P}_k(t) + t^2, t \in \R)$.  For $s, t \in [0,1]$, we have $|s^2 - t^2| \leq 2|t - s| \leq 2|t - s|^{1/2}$.  Therefore, the bound \eqref{Pkdiff} can be rewritten as
\begin{equation}\label{diff2}
|(\mathcal{P}_k(s) - s^2) - (\mathcal{P}_k(t) - t^2)| \leq C |t - s|^{1/2} \log(1/\eps).
\end{equation}
By considering overlapping time intervals of length 1 started at the half integers, we get that \eqref{diff2} holds with probability at least $1 - \eps^{100}$ for $k \in \{1, 2, 3\}$ and all $s, t \in [-T, T]$ with $\eps^{100} \leq |t - s| \leq 1$.  Now for $s, t \in [-T, T]$ with $\eps^{100} \leq |t - s| \leq 1$,
\begin{align*}
|\mathcal{P}_k(s) - \mathcal{P}_k(t)| &= |(\mathcal{P}_k(s) + s^2) - (\mathcal{P}_k(t) + t^2) + (t^2 - s^2)| \\
&\leq |(\mathcal{P}_k(s) + s^2) - (\mathcal{P}_k(t) + t^2)| + 2T |t - s| \\
&\leq C \log(1/\eps) \sqrt{|t - s|}.
\end{align*}
This inequality implies the result.
\end{proof}

Given the above inputs, we are now in a position to dive into the proof of Proposition \ref{Airyprop}.

\begin{proof}[Proof of Proposition \ref{Airyprop}]
We first observe that Lemma \ref{AiryHolder} allows us to restrict our attention to a finite mesh for the points $a$, $b$, and $c$.
Let $m = 2^h$, where $$h = \min\{k \in \Z: 2^{-k} \leq \eps^2 \log^{-3}(1/\eps)\}.$$  Note that $\frac{1}{2} \eps^2 \log^{-3}(1/\eps) < 1/m \leq \eps^2 \log^{-3}(1/\eps)$.
Let $L = \{i/m: -m(T-1) \leq i \leq m(T - 1)\}$ be the set consisting of the integer multiples of $1/m$ that are in the interval $[-(T-1), T-1]$.

Given $u, v, w \in L$ with $u < v < w$, we define the event
$$A_{u,v,w} = \{\mathcal{P}_1(u) - \mathcal{P}_2(u) < 2 \eps, \: \mathcal{P}_1(w) - \mathcal{P}_2(w) < 2 \eps, \: \mathcal{P}_2(v) - \mathcal{P}_3(v) < 2 \eps\}.$$
Also, let $$m_1^* = \max_{-T \leq t \leq u} (W(t) + \mathcal{P}_1(t)), \quad m_2^* = \max_{u \leq t \leq w} (W(t) + \mathcal{P}_1(t)), \quad m_3^* = \max_{w \leq t \leq T} (W(t) + \mathcal{P}_1(t)).$$  Then define the event $$B_{u,v,w} = \{|m^*_1 - m^*_2| < 2 \eps, |m^*_2 - m^*_3| < 2 \eps\}.$$
Recall the event $A$ defined in the statement of the proposition, and let $G$ be the event in Lemma~\ref{AiryHolder}.  It follows from Lemma \ref{AiryHolder} that if $A \cap G$ occurs, then we can find points $u$, $v$, and $w$ in $L$ that are sufficiently close to $a$, $b$, and $c$ respectively that as long as $\eps$ is small enough, the event $A_{u,v,w} \cap B_{u,v,w}$ must occur.  Therefore,
\begin{equation}\label{mainPA}
\P(A) \leq \P(G^c) + \sum_{\substack{u, v, w \in L \\ u < v < w}} \P(A_{u,v,w} \cap B_{u,v,w} \cap G).
\end{equation}
We will therefore bound $\P(A_{u,v,w} \cap B_{u,v,w} \cap G)$.  We will first bound $\P(A_{u,v,w})$ and then essentially bound  $\P(B_{u,v,w}\cap G\mid A_{u,v,w})$. The former task involves considering three cases involving the separation of $u,v,$ and $w$.

\medskip
\noindent {\bf Case 1}:  Suppose $v - u > 2$ and $w - v > 2$.  Let $a_1 = u - 1/2$, $a_2 = v - 1/2$, and $a_3 = w - 1/2$.  We can then apply Theorem \ref{duncan12} with $t = 1$, $k = \ell = 3$, and ${\bf a} = (a_1, a_2, a_3)$.  
Letting $\mathfrak{L}^{\bf a}$ be the ensemble defined in Theorem \ref{duncan12}, we get
\begin{align*}
&\P(A_{u,v,w}) \\
&\quad = \P(\mathfrak{L}^{\bf a}_1(u) - \mathfrak{L}^{\bf a}_2(u) < 2 \eps, \mathfrak{L}^{\bf a}_1(w) - \mathfrak{L}^{\bf a}_2(w) < 2 \eps, \mathfrak{L}^{\bf a}_2(v) - \mathfrak{L}^{\bf a}_3(v) < 2 \eps \,|\, \mathfrak{L}^{\bf a}_1 > \mathfrak{L}^{\bf a}_2 > \dots) \\
&\quad \leq e^{c_{3,3}^*}  \P(\mathfrak{L}^{\bf a}_1(u) - \mathfrak{L}^{\bf a}_2(u) < 2 \eps, \mathfrak{L}^{\bf a}_1(w) - \mathfrak{L}^{\bf a}_2(w) < 2 \eps, \mathfrak{L}^{\bf a}_2(v) - \mathfrak{L}^{\bf a}_3(v) < 2 \eps , \mathfrak{L}^{\bf a}_1 > \mathfrak{L}^{\bf a}_2 > \dots).
\end{align*}
Let $\mathcal{G}$ be the $\sigma$-field generated by the values of $\mathfrak{L}^{\bf a}$ outside $U({\bf a})$.
It follows from the Gibbs property stated in Theorem \ref{duncan12} with $m = 3$ that conditional on $\mathcal{G}$, the processes $\mathfrak{L}^{\bf a}_1$, $\mathfrak{L}^{\bf a}_2$, $\mathfrak{L}^{\bf a}_3$ in the intervals $(a_j, a_j + 1)$ for $j = 1, 2, 3$ are independent Brownian bridges.
Let $D_1$ be the event that $\mathfrak{L}^{\bf a}_1(s) > \mathfrak{L}^{\bf a}_2(s)$ for all $s \in (a_1, a_1 + 1)$.  Let $D_2$ be the event that $\mathfrak{L}^{\bf a}_2(s) > \mathfrak{L}^{\bf a}_3(s)$ for all $s \in (a_2, a_2 + 1)$, and let $D_3$ be the event that $\mathfrak{L}^{\bf a}_1(s) > \mathfrak{L}^{\bf a}_2(s)$ for all $s \in (a_3, a_3 + 1)$.  Then
{\begin{align*}
&\P(\mathfrak{L}^{\bf a}_1(u) - \mathfrak{L}^{\bf a}_2(u) < 2 \eps, \mathfrak{L}^{\bf a}_1(w) - \mathfrak{L}^{\bf a}_2(w) < 2 \eps, \mathfrak{L}^{\bf a}_2(v) - \mathfrak{L}^{\bf a}_3(v) < 2 \eps , \mathfrak{L}^{\bf a}_1 > \mathfrak{L}^{\bf a}_2 > \dots | \mathcal{G}) \\
&\quad \leq \P(\mathfrak{L}^{\bf a}_1(u) - \mathfrak{L}^{\bf a}_2(u) < 2 \eps, D_1 | \mathcal{G}) \P(\mathfrak{L}^{\bf a}_2(v) - \mathfrak{L}^{\bf a}_3(v) < 2 \eps, D_2 | \mathcal{G}) \P( \mathfrak{L}^{\bf a}_1(w) - \mathfrak{L}^{\bf a}_2(w) < 2 \eps, D_3 | \mathcal{G}).
\end{align*}}
Because the difference between two Brownian bridges is a Brownian bridge with twice the variance parameter, we can apply Lemma \ref{onebridge} with $\mathfrak{L}^{\bf a}_1 - \mathfrak{L}^{\bf a}_2$ or $\mathfrak{L}^{\bf a}_2 - \mathfrak{L}^{\bf a}_3$ in place of $X$ to get that each of the three probabilities on the right-hand side is $O(\eps^3)$.  
It follows that
\begin{equation}\label{case1}
\P(A_{u,v,w}) = O(\eps^9).
\end{equation}

\medskip
\noindent {\bf Case 2}:  Suppose either $v - u > 4$ and $w - v \leq 2$, or $v - u \leq 2$ and $w - v > 4$.  Because these two possibilities can be handled in the same way, it is enough to consider $v - u > 4$ and $w - v \leq 2$.  This time, instead of considering three intervals, we will consider the two intervals $(u - 5/2, u + 1/2)$ and $(v - 1/2, v + 5/2)$.  Note that $v + 5/2 \geq w + 1/2$ and $(v - 1/2) - (u + 1/2) > 3$.  Therefore, we can apply Theorem \ref{duncan12} with $\ell = 2$, $a_1 = u - 5/2$, $a_2 = v - 1/2$, and $t = 3$.  As in the previous case, let $\mathfrak{L}^{\bf a}$ be the ensemble defined in Theorem \ref{duncan12}, and let $\mathcal{G}$ be the $\sigma$-field generated by the values of $\mathfrak{L}^{\bf a}$ outside $U({\bf a})$. Similarly, let $D_1$ be the event that $\mathfrak{L}_1^{\bf a}(r) > \mathfrak{L}_2^{\bf a}(r)$ for all $r \in (a_1, a_1 + 3)$, and let $D_2$ be the event that $\mathfrak{L}_1^{\bf a}(r) > \mathfrak{L}_2^{\bf a}(r) > \mathfrak{L}^{\bf a}_3(r)$ for all $r \in (a_2, a_2 + 3)$.  Also, let $G_2$ be the event that $|\mathfrak{L}^{\bf a}_i(t) - \mathfrak{L}_i^{\bf a}(s)| \leq C_{10} \log(1/\eps) \sqrt{|t - s|}$ whenever $i \in \{1, 2, 3\}$ and $s,t \in [a_2, a_2 + 3]$ with $\eps^{100} \leq |t - s| \leq 1$.  Reasoning as in Case 1, we can apply Theorem \ref{duncan12} to get
\begin{align} \label{AuvwG}
&\P(A_{u,v,w} \cap G) \nonumber \\
&\: \leq C \P(\mathfrak{L}^{\bf a}_1(u) - \mathfrak{L}^{\bf a}_2(u) < 2 \eps, \mathfrak{L}^{\bf a}_1(w) - \mathfrak{L}^{\bf a}_2(w) < 2 \eps, \mathfrak{L}^{\bf a}_2(v) - \mathfrak{L}^{\bf a}_3(v) < 2 \eps, \mathfrak{L}^{\bf a}_1 > \mathfrak{L}^{\bf a}_2 > \dots, G_2).
\end{align}
Also,
\begin{align} \label{case2cond}
&\P(\mathfrak{L}^{\bf a}_1(u) - \mathfrak{L}^{\bf a}_2(u) < 2 \eps, \mathfrak{L}^{\bf a}_1(w) - \mathfrak{L}^{\bf a}_2(w) < 2 \eps, \mathfrak{L}^{\bf a}_2(v) - \mathfrak{L}^{\bf a}_3(v) < 2 \eps, \mathfrak{L}^{\bf a}_1 > \mathfrak{L}^{\bf a}_2 > \dots, G_2|\mathcal{G}) \nonumber \\
&\: \: \leq \P(\mathfrak{L}^{\bf a}_1(u) - \mathfrak{L}^{\bf a}_2(u) < 2 \eps, D_1|\mathcal{G}) \P(\mathfrak{L}^{\bf a}_1(w) - \mathfrak{L}^{\bf a}_2(w) < 2 \eps, \mathfrak{L}^{\bf a}_2(v) - \mathfrak{L}^{\bf a}_3(v) < 2 \eps, D_2 \cap G_2|\mathcal{G}).
\end{align}
Applying Lemma \ref{onebridge} and following the argument used for Case 1, we get
\begin{equation}\label{firstprob}
\P(\mathfrak{L}^{\bf a}_1(u) - \mathfrak{L}^{\bf a}_2(u) < 2 \eps, D_1|\mathcal{G}) = O(\eps^3).
\end{equation}

It remains to bound $\P(\mathfrak{L}^{\bf a}_1(w) - \mathfrak{L}^{\bf a}_2(w) < 2 \eps, \mathfrak{L}^{\bf a}_2(v) - \mathfrak{L}^{\bf a}_3(v) < 2 \eps, D_2 \cap G_2|\mathcal{G})$.  It follows from Theorem \ref{duncan12} that conditional on $\mathcal{G}$, the processes $\mathfrak{L}^{\bf a}_1$, $\mathfrak{L}^{\bf a}_2$, and $\mathfrak{L}^{\bf a}_3$ over the interval $[a_2, a_2 + 3]$ are independent Brownian bridges.  Note that $D_2$ is the event that these bridges do not intersect.  Therefore, if we let $X$, $Y$, and $Z$ be nonintersecting independent Brownian bridges on the time interval $[a_2, a_2 + 3]$ with the same endpoints as $\mathfrak{L}^{\bf a}_1$, $\mathfrak{L}^{\bf a}_2$, and $\mathfrak{L}^{\bf a}_3$, and if $G_2^*$ is defined in the same way as $G_2$ but with $X$, $Y$, and $Z$ in place of $\mathfrak{L}^{\bf a}_1$, $\mathfrak{L}^{\bf a}_2$, and $\mathfrak{L}^{\bf a}_3$, then
\begin{align}\label{LXYZbound}
&\P(\mathfrak{L}^{\bf a}_1(w) - \mathfrak{L}^{\bf a}_2(w) < 2 \eps, \mathfrak{L}^{\bf a}_2(v) - \mathfrak{L}^{\bf a}_3(v) < 2 \eps, D_2 \cap G_2|\mathcal{G}) \nonumber \\
&\qquad \qquad \qquad \qquad \qquad \leq \P(X_w - Y_w < 2 \eps, Y_v - Z_v < 2 \eps, G_2^*).
\end{align}
We will therefore work with these nonintersecting Brownian bridges. The analysis proceeds by zooming in on an appropriate scale to probe the behaviors near $v$ and $w.$
Let $\ell$ be the nonnegative integer such that $2^{-\ell} < w - v \leq 2^{-(\ell - 1)}$.  Define the times $$s_1 = v - \frac{2^{-\ell}}{4}, \quad s_2 = \frac{v+w}{2}, \quad s_3 = w + \frac{2^{-\ell}}{4}.$$
Note that $a_2 < s_1 < v < s_2 < w < s_3 < a_2 + 3$.  Also, $s_1 - a_2 \geq 1/4$ and $(a_2 + 3) - s_3 \geq 1/4$.  Let $E_1$ be the event that $X_{s_1} - Y_{s_1} \leq 2^{-\ell/2} \log^2(1/\eps)$ and $Y_{s_1} - Z_{s_1} \leq 2^{-\ell/2} \log^2(1/\eps)$.  If the events $X_w - Y_w < 2 \eps$, $Y_v - Z_v < 2 \eps$, and $G_2^*$ all occur, then by the regularity imposed by the latter, $X_{s_1} - Y_{s_1}$ and $Y_{s_1} - Z_{s_1}$ are both bounded above by $2 \eps + 3 C_{10} 2^{-\ell/2} \log(1/\eps)$.  Because $2^{-(\ell - 1)} \geq \frac{1}{2} \eps^2 \log^{-3}(1/\eps)$, it follows that for sufficiently small $\eps$, the event $E_1$ must occur.
By \eqref{threeclose},
\begin{equation}\label{E1G}
\P(E_1|\mathcal{G}) = O(2^{-4 \ell} \log^{16}(1/\eps)).
\end{equation} 
Let $\mathcal{J}$ be the $\sigma$-field generated by $\mathcal{G}$ and by the values $X_{s_i}$, $Y_{s_i}$, and $Z_{s_i}$ for $i \in \{1, 2, 3\}$.  Note that $E_1 \in \mathcal{J}$.  
Let $E_2$ be the event that $X_w - Y_w < 2 \eps$ and $Y_v - Z_v < 2 \eps$.
We have $s_2 - s_1 = s_3 - s_2 = (w - v) + \frac{1}{2} 2^{-\ell} \in (\frac{3}{2} 2^{-\ell}, \frac{5}{2} 2^{-\ell}]$.  Therefore, we can use Brownian scaling to scale the time intervals $[s_1, s_2]$ and $[s_2, s_3]$ to $[0,1]$ and then apply \eqref{twoclose} twice, one for each interval, with $2 \eps (s_2 - s_1)^{-1/2}$ in place of $\eps$, to get
\begin{equation}\label{E2H}
\P(E_2|\mathcal{J}) = O(\eps^6 2^{3 \ell}).
\end{equation}
It follows from \eqref{E1G} and \eqref{E2H} that $\P(X_w - Y_w < 2 \eps, Y_v - Z_v < 2 \eps, G_2^*) = O(\eps^6 2^{-\ell} \log^{16}(1/\eps))$.
Because this bound does not depend on the endpoints of the Brownian bridges, combining this result with \eqref{LXYZbound} yields
\begin{equation}\label{secondprob}
\P(\mathfrak{L}^{\bf a}_1(w) - \mathfrak{L}^{\bf a}_2(w) < 2 \eps, \mathfrak{L}^{\bf a}_2(v) - \mathfrak{L}^{\bf a}_3(v) < 2 \eps, D_2 \cap G_2|\mathcal{G}) = O(\eps^6 2^{-\ell} \log^{16}(1/\eps)).
\end{equation}
Combining \eqref{AuvwG}, \eqref{case2cond}, \eqref{firstprob}, and \eqref{secondprob}, we get
\begin{equation}\label{case2}
\P(A_{u,v,w} \cap G) = O(\eps^9 2^{-\ell} \log^{16}(1/\eps)).
\end{equation}

\medskip
\noindent {\bf Case 3}:  Suppose $v - u \leq 4$ and $w - v \leq 4$.  We apply Theorem \ref{duncan12} with $\ell = 1$, $a_1 = u - 2$, and $t = 12$.  Let $\mathfrak{L}^{\bf a}$ be the ensemble defined in Theorem \ref{duncan12}.  Let $D$ be the event that $\mathfrak{L}_1^{\bf a}(r) > \mathfrak{L}_2^{\bf a}(r) > \mathfrak{L}^{\bf a}_3(r)$ for all $r \in [a_1, a_1 + 12]$.  Also, let $G_3$ be the event that $|\mathfrak{L}^{\bf a}_i(t) - \mathfrak{L}_i^{\bf a}(s)| \leq C_{10} \log(1/\eps) \sqrt{|t - s|}$ whenever $i \in \{1, 2, 3\}$ and $s,t \in (a_1, a_1 + 12)$ with $\eps^{100} \leq |t - s| \leq 1$.  Theorem \ref{duncan12} implies that
\begin{align*}
\P(A_{u,v,w} \cap G) \leq C \P(\mathfrak{L}^{\bf a}_1(u) - \mathfrak{L}^{\bf a}_2(u) < 2 \eps, \mathfrak{L}^{\bf a}_1(w) - \mathfrak{L}^{\bf a}_2(w) < 2 \eps, \mathfrak{L}^{\bf a}_2(v) - \mathfrak{L}^{\bf a}_3(v) < 2 \eps, D \cap G_3).
\end{align*}
It follows from Theorem \ref{duncan12} that conditional on $\mathcal{G}$, the processes $\mathfrak{L}^{\bf a}_1$, $\mathfrak{L}^{\bf a}_2$, and $\mathfrak{L}^{\bf a}_3$ over the interval $[a_1, a_1 + 12]$ are independent Brownian bridges. 
As in Case 2, we will consider nonintersecting Brownian bridges $X$, $Y$, and $Z$ on the interval $[a_1, a_1 + 12]$ with arbitrary endpoints.  Define $G_3^*$ the same way as $G_3$ but with $X$, $Y$, and $Z$ in place of $\mathfrak{L}^{\bf a}_1$, $\mathfrak{L}^{\bf a}_2$, and $\mathfrak{L}^{\bf a}_3$.  Then
\begin{equation}\label{main3bound}
\P(A_{u,v,w} \cap G) \leq C \P(X_u - Y_u < 2 \eps, X_w - Y_w < 2 \eps, Y_v - Z_v < 2 \eps, G_3^*).
\end{equation}
We again do a scale decomposition.
Let $k$ be the integer such that $2^{-k} < w - u \leq 2^{-(k-1)}$, and let $\ell$ be the integer such that $2^{-\ell} < v - u \leq 2^{-(\ell - 1)}$.  Note that $k \geq -1$ and $\ell \geq -1$.  We may assume without loss of generality that $v - u \leq \frac{1}{2}(w - u)$, which implies that $\ell > k$.  We now define the following times:
$$t_1=u-\frac{2^{-k}}{2}, \quad t_2=w- \frac{2^{-k}}{4}, \quad t_3=w+\frac{2^{-k}}{2}, \quad t_4=u- \frac{2^{-\ell}}{2}, \quad t_5 = \frac{u+v}{2}, \quad t_6=v+\frac{2^{-\ell}}{2}.$$
Note that $t_1 < t_4 < u < t_5 < v < t_6 < t_2 < w < t_3$.  Also, the points $t_1 < u < t_2 < w < t_3$ all fit within an interval of length at most $3 \cdot 2^{-k}$.  Likewise, the points $t_4 < u < t_5 < v < t_6$ all fit within an interval of length at most $3 \cdot 2^{-\ell}$.  We will use these times to study the process on two different time scales.  See the figure below.

\begin{figure}[h]
\centering
\begin{tikzpicture}[scale=0.8]
\draw[thick] (1,2)--(20,2);
\draw[fill=black](6,2) circle (3pt);
\draw[fill=black](8,2) circle (3pt);
\draw[fill=black](15,2) circle (3pt);
\draw (2,1.8)--(2,2.2);
\draw (5,1.8)--(5,2.2);
\draw (7,1.8)--(7,2.2);
\draw (9,1.8)--(9,2.2);
\draw (12,1.8)--(12,2.2);
\draw (19,1.8)--(19,2.2);
\node at (2, 1.4){$t_1$};
\node at (5, 1.4){$t_4$};
\node at (6, 1.4){$u$};
\node at (7, 1.4){$t_5$};
\node at (8, 1.4){$v$};
\node at (9, 1.4){$t_6$};
\node at (12, 1.4){$t_2$};
\node at (15, 1.4){$w$};
\node at (19, 1.4){$t_3$};
\end{tikzpicture}
\label{points}
\end{figure}

We now define four new events.
\begin{itemize}
\item Let $A_1$ be the event that $X_{t_2} - Y_{t_2} \leq 2^{-k/2} \log^2(1/\eps)$ and $Y_{t_2} - Z_{t_2} \leq 2^{-k/2} \log^2(1/\eps)$.
For sufficiently small $\eps$, if $G_3^*$ occurs and if both $X_u - Y_u$ and $Y_v - Z_v$ are less than $2 \eps$, then the regularity imposed by $G_3^*$ ensures that the closeness of the curves must be felt at $t_2$ as well and therefore the event $A_1$ must occur. 

\item Let $A_2$ be the event that $X_w - Y_w < 2 \eps$.

\item Let $A_3$ be the event that $X_{t_4} - Y_{t_4} \leq 2^{-\ell/2} \log^2(1/\eps)$ and $Y_{t_4} - Z_{t_4} \leq 2^{-\ell/2} \log^2(1/\eps)$.  Note that for sufficiently small $\eps$, if $G$ occurs and if both $X_u - Y_u$ and $Y_v - Z_v$ are less than $2 \eps$, then $A_3$ must occur. 

\item Let $A_4$ be the event that $X_u - Y_u < 2 \eps$ and $Y_v - Z_v < 2 \eps$.
\end{itemize}
This discussion implies that the event on the right-hand side of \eqref{main3bound} whose probability we need to bound is contained in
$A_1 \cap A_2 \cap A_3 \cap A_4$.

Let $\mathcal{G}$ be the $\sigma$-field generated by the values of $\mathfrak{L}^{\bf a}$ outside $U({\bf a})$.  We first bound $P(A_1|\mathcal{G})$.  We can scale the time interval $[a_1, a_1 + 12]$ to $[0,1]$ using Brownian scaling and apply \eqref{threeclose} with $\frac{1}{\sqrt{12}} 2^{-k/2} \log^2(1/\eps)$ in place of $\eps$ to get
\begin{equation}\label{PA1}
\P(A_1) = O(2^{-4k} \log^{16}(1/\eps)).
\end{equation}

Let $\mathcal{F}_1$ be the $\sigma$-field generated by $\mathcal{G}$ and by $X_{t_2}$, $Y_{t_2}$, $Z_{t_2}$, $X_{t_3}$, $Y_{t_3}$, and $Z_{t_3}$.  Conditional on $\mathcal{F}_1$, the processes $X$, $Y$, and $Z$ between times $t_2$ and $t_3$ are nonintersecting Brownian bridges.  
We have $t_3 - t_2 = \frac{3}{4} 2^{-k}$.  Also, $w - t_2 = \frac{1}{4} 2^{-k}$ and $t_3 - w = \frac{1}{2}2^{-k}$.  Therefore, we can use Brownian scaling to scale the time interval $[t_2, t_3]$ to $[0,1]$, and then apply \eqref{twoclose} with $2 \eps (t_3 - t_2)^{-1/2}$ in place of $\eps$ to get
\begin{equation}\label{PA2}
\P(A_2|\mathcal{F}_1) = O(2^{3k/2} \eps^3).
\end{equation}

Analogous to above, let $\mathcal{F}_2$ be the $\sigma$-field generated by $\mathcal{F}_1$, $X_{t_1}$, $Y_{t_1}$, $Z_{t_1}$, $X_w$, $Y_w$, and $Z_w$.
Conditional on $\mathcal{F}_2$, the processes $X$, $Y$, and $Z$ between times $t_1$ and $t_2$ are nonintersecting Brownian bridges. 
We have $\frac{5}{4} 2^{-k} \leq t_2 - t_1 \leq \frac{9}{4} 2^{-k}$.  Also, $t_4 - t_1 \geq \frac{1}{4} 2^{-k}$ because $\ell > k$, and $t_2 - t_4 \geq \frac{3}{4}2^{-k}$.  Therefore, we can use Brownian scaling to scale the time interval $[t_1, t_2]$ to $[0,1]$, and then we can apply \eqref{threeclose} with $2^{-\ell/2} \log^2(1/\eps) (t_2 - t_1)^{-1/2}$ in place of $\eps$ to get
\begin{equation}\label{PA3}
\P(A_3|\mathcal{F}_2) = O(2^{4k-4\ell} \log^{16}(1/\eps)).
\end{equation}

Let $\mathcal{F}_3$ be the $\sigma$-field generated by $\mathcal{F}_2$ and by $X_{t_i}$, $Y_{t_i}$, and $Z_{t_i}$ for $i \in \{4, 5, 6\}$.
Conditional on $\mathcal{F}_3$, the processes $X$, $Y$, and $Z$ between times $t_4$ and $t_5$, and between times $t_5$ and $t_6$, are independent Brownian bridges.  
We have $2^{-\ell} \leq t_5 - t_4 = t_6 - t_5 \leq \frac{3}{2} 2^{-\ell}$.  Also, $u$ and $v$ are at least a distance $\frac{1}{2} 2^{-\ell}$ from the points $t_4$, $t_5$, and $t_6$.  Therefore, we can use Brownian scaling to scale the time intervals $[t_4, t_5]$ and $[t_5, t_6]$ to $[0,1]$, and then we can apply \eqref{twoclose} twice, once for each interval, with $2 \eps (t_5 - t_4)^{-1/2}$ in place of $\eps$ to get
\begin{equation}\label{PA4}
\P(A_4|\mathcal{F}_3) = O(2^{3 \ell} \eps^6).
\end{equation}

Because $A_1 \in \mathcal{F}_1$, $A_1 \cap A_2 \in \mathcal{F}_2$, and $A_1 \cap A_2 \cap A_3 \in \mathcal{F}_3$, we can combine the estimates in \eqref{PA1}, \eqref{PA2}, \eqref{PA3}, and \eqref{PA4} to get
\begin{equation}\label{case3}
\P(A_{u,v,w} \cap G) \leq C \P(A_1 \cap A_2 \cap A_3 \cap A_4) = O(2^{3k/2} 2^{-\ell} \eps^9 \log^{32}(1/\eps)).
\end{equation}

The number of points in $L$ is $O(\eps^{-2} \log^{7/2}(1/\eps))$.  Therefore, the number of possible choices for $(u,v,w)$ in Case 1 is $O(\eps^{-6} \log^{21/2}(1/\eps))$.  In Case 2, given the values of $u$, $v$, and $\ell$, the number of possible choices for $w$ is $O(2^{-\ell} \eps^{-2} \log^3(1/\eps))$.  Therefore, for a given value of $\ell$, the number of possible choices for $(u,v,w)$ in Case 2 is $O(\eps^{-6} \log^{10}(1/\eps) 2^{-\ell})$.  Likewise, in Case 3, given the values of $u$, $k$, and $\ell$, the number of possible choices for $v$ is $O(2^{-\ell} \eps^{-2} \log^3(1/\eps))$, and the number of possible choices for $w$ is $O(2^{-k} \eps^{-2} \log^3(1/\eps))$.  Therefore, for given values of $k$ and $\ell$, the number of possible choices for $(u,v,w)$ in Case 3 is $O(2^{-k} 2^{-\ell} \eps^{-6} \log^{19/2}(1/\eps))$.
Now we can combine the results \eqref{case1}, \eqref{case2}, and \eqref{case3}, along with Lemma \ref{AiryHolder}, to get
\begin{align}\label{Auvw}
\sum_{\substack{u, v, w \in L \\ u < v < w}} \P(A_{u,v,w}) &\leq \sum_{\substack{u, v, w \in L \\ u < v < w}} \P(G^c) + \sum_{\substack{u, v, w \in L \\ u < v < w}} \P(A_{u,v,w} \cap G) \nonumber \\
&\leq C \eps^{100} \cdot \eps^{-6} \log^{21/2}(1/\eps) + C \eps^9 \cdot \eps^{-6} \log^{21/2}(1/\eps) \nonumber \\
&\qquad+ C \sum_{\ell = 0}^{\infty} \eps^9 2^{-\ell} \log^{16}(1/\eps) \cdot \eps^{-6} \log^{10}(1/\eps) 2^{-\ell} \nonumber \\
&\qquad \qquad + C \sum_{k=-1}^{\infty} \sum_{\ell = k+1}^{\infty} 2^{3k/2} 2^{-\ell} \eps^9 \log^{32}(1/\eps) \cdot 2^{-k} 2^{-\ell} \eps^{-6} \log^{19/2}(1/\eps) \nonumber \\
&\leq C \eps^{94} \log^{11}(1/\eps) + C \eps^3 \log^{11}(1/\eps) + C \eps^3 \log^{26}(1/\eps) \sum_{\ell = 0}^{\infty} 2^{-2 \ell}  \nonumber \\
&\qquad + C \eps^3 \log^{42}(1/\eps) \sum_{k=-1}^{\infty} 2^{k/2} \sum_{\ell=k+1}^{\infty} 2^{-2 \ell} \nonumber \\
&\leq C \eps^3 \log^{42}(1/\eps).
\end{align}

Let $\mathcal{H}$ be the $\sigma$-field generated by $(W(t), u \leq t \leq w)$ and by $\mathcal{P}_i(t)$ for $t \in [u, w]$ and $i \in \{1, 2, 3\}$.  Note that $A_{u,v,w} \in \mathcal{H}$, and $m_2^*$ is $\mathcal{H}$-measurable. 

Having bounded $\P(A_{u,v,w})$ we now aim to bound $\P(B_{u,v,w} \cap G|\mathcal{H})$.  
Let $G_u \subseteq G$ be the event that the bound in Lemma \ref{AiryHolder} holds when $s,t \in [-T, u]$, and let $G_w \subseteq G$ be the event that the bound in Lemma \ref{AiryHolder} holds when $s,t \in [w, T]$.  If $|m_1^* - m_2^*|  < 2 \eps$ and $G_u$ occurs, we claim that one of the following must hold:
\begin{itemize}
\item There exists $r \in L \cap [-T, (u - T)/2]$ such that $$r = \min\{t \in L \cap [-T, (u - T)/2]: W(t) + \mathcal{P}_1(t) \geq m_2^* - 3 \eps\}$$ and 
$$\max_{t \in [r, r + 1/2]} (W(t) + \mathcal{P}_1(t) - W(r) - \mathcal{P}_1(r)) < 5 \eps.$$  Also, we must have $|W(r) + \mathcal{P}_1(r) - (W(r+1/2) + \mathcal{P}_1(r +1/2))| < \sqrt{2} C_{10} \log(1/\eps)$.

\item There exists $r \in L \cap ((u - T)/2, u]$ such that $$r = \max\{t \in L \cap ((u - T)/2, u]: W(t) + \mathcal{P}_1(t) \geq m_2^* - 3 \eps\}$$ and $$\max_{t \in [r, r - 1/2]} (W(t) + \mathcal{P}_1(t) - W(r) - \mathcal{P}_1(r)) < 5 \eps.$$  Also, we must have $|W(r) + \mathcal{P}_1(r) - (W(r-1/2) + \mathcal{P}_1(r -1/2))| < \sqrt{2} C_{10} \log(1/\eps)$.
\end{itemize}
To see this, note that Lemma \ref{AiryHolder} implies that for $\eps$ sufficiently small, when $|m_1^* - m_2^*| < 2 \eps$ and $G_u$ occurs, the process $(W(t) + \mathcal{P}_1(t), -T \leq t \leq u)$ must get to within $\eps$ of its maximum, and therefore within $3 \eps$ of $m_2^*$, at one of the points in $L$.  This point in $L$ must either be in $[-T, (u - T)/2]$ or in $((u - T)/2, u]$.  In the first case, we consider the points in $L$ in increasing order, and let $r$ be the first point at which the process gets within $3 \eps$ of $m_2^*$.  If $|m_1^* - m_2^*| < 2 \eps$, then the process $(W(t) + \mathcal{P}_1(t), -T \leq t \leq u)$ can not increase by more than $5 \eps$ during the time interval $[r, r + 1/2]$.  Note, in particular, that $r + 1/2 < u$ because $u \geq -(T-1)$.  The second case is handled the same way, except that we work backwards from time $u$.

Now suppose $r \in L \cap [-T, (u - T)/2]$, and let $\mathcal{H}_r$ be the $\sigma$-field generated by $\mathcal{P}_i$ for $i \geq 2$ and by $W(t)$ and $\mathcal{P}_1(t)$ for all $t \notin (r, r+1/2)$.  A consequence of the Brownian Gibbs Property of the parabolic Airy line ensemble is that conditional on $\mathcal{H}_r$, the process $(W(t), t \in [r, r + 1/2])$ is a Brownian bridge, and the process $(\mathcal{P}_1(t), t \in [r, r + 1/2])$ is a Brownian bridge conditioned to stay above $\mathcal{P}_2$.  A standard monotonicity argument given, for example, in Lemma 2.6 of \cite{ch14} implies that conditional on $\mathcal{H}_r$, the process $(\mathcal{P}_1(t), t \in [r, r + 1/2])$ can be coupled with a Brownian bridge $(B(t), t \in [r, r+1/2])$ having the same endpoints such that $B(t) \leq \mathcal{P}_1(t)$ for all $t \in [r, r + 1/2]$.  Therefore,
\begin{align}\label{probincrease}
&\P \Big( \max_{t \in [r, r + 1/2]} (W(t) + \mathcal{P}_1(t) - W(r) - \mathcal{P}_1(r)) < 5 \eps \,\Big|\, \mathcal{H}_r \Big) \nonumber \\
&\qquad \leq \P \Big( \max_{t \in [r, r + 1/2]} (W(t) + B(t) - W(r) - B(r)) < 5 \eps \,\Big|\, \mathcal{H}_r \Big).
\end{align}
Because the sum of two Brownian bridges is a Brownian bridge with twice the variance parameter, it follows from Lemma \ref{maxbridge} that on the event that $|W(r) + \mathcal{P}_1(r) - (W(r+1/2) + \mathcal{P}_1(r +1/2))| < \sqrt{2} C_{10} \log(1/\eps)$, the conditional probability on the right-hand side of \eqref{probincrease} is $O(\eps \log(1/\eps))$.
Because the Brownian bridge is symmetric under time reversal, the same bound can be obtained when $r \in L \cap ((u - T)/2, u]$.  It follows that $$\P(G_u \cap \{|m_1^* - m_2^*| < 2 \eps\}|\mathcal{H}) = O(\eps \log(1/\eps)).$$
The same argument, applied to the interval $[w, T]$ instead of $[-T, u]$ gives $$\P(G_w \cap \{|m_2^* - m_3^*| > 2 \eps\}|\mathcal{H}) = O(\eps \log(1/\eps)),$$
and also establishes the conditional independence of these two events given $\mathcal{H}$.
Combining these results, we get $\P(B_{u,v,w} \cap G|\mathcal{H}) = O(\eps^2 \log^2(1/\eps))$, and therefore because $A_{u,v,w} \in \mathcal{H}$, there is a positive constant $C$ such that
\begin{equation}\label{Buvw}
\P(B_{u,v,w} \cap G|A_{u,v,w}) \leq C \eps^2 \log^2(1/\eps).
\end{equation}
Combining \eqref{mainPA}, \eqref{Auvw}, and \eqref{Buvw}, we get
$$\P(A) \leq \P(G^c) + \sum_{\substack{u, v, w \in L \\ u < v < w}} \P(A_{u,v,w}) \P(B_{u,v,w} \cap G|A_{u,v,w}) \leq C \eps^5 \log^{44}(1/\eps),$$
which completes the proof.
\end{proof}

\section{Simulation results}\label{simsec}

\begin{figure}[h!]
\centering
\includegraphics[scale=0.35, trim={1cm 5.5cm 2cm 6cm}, clip]{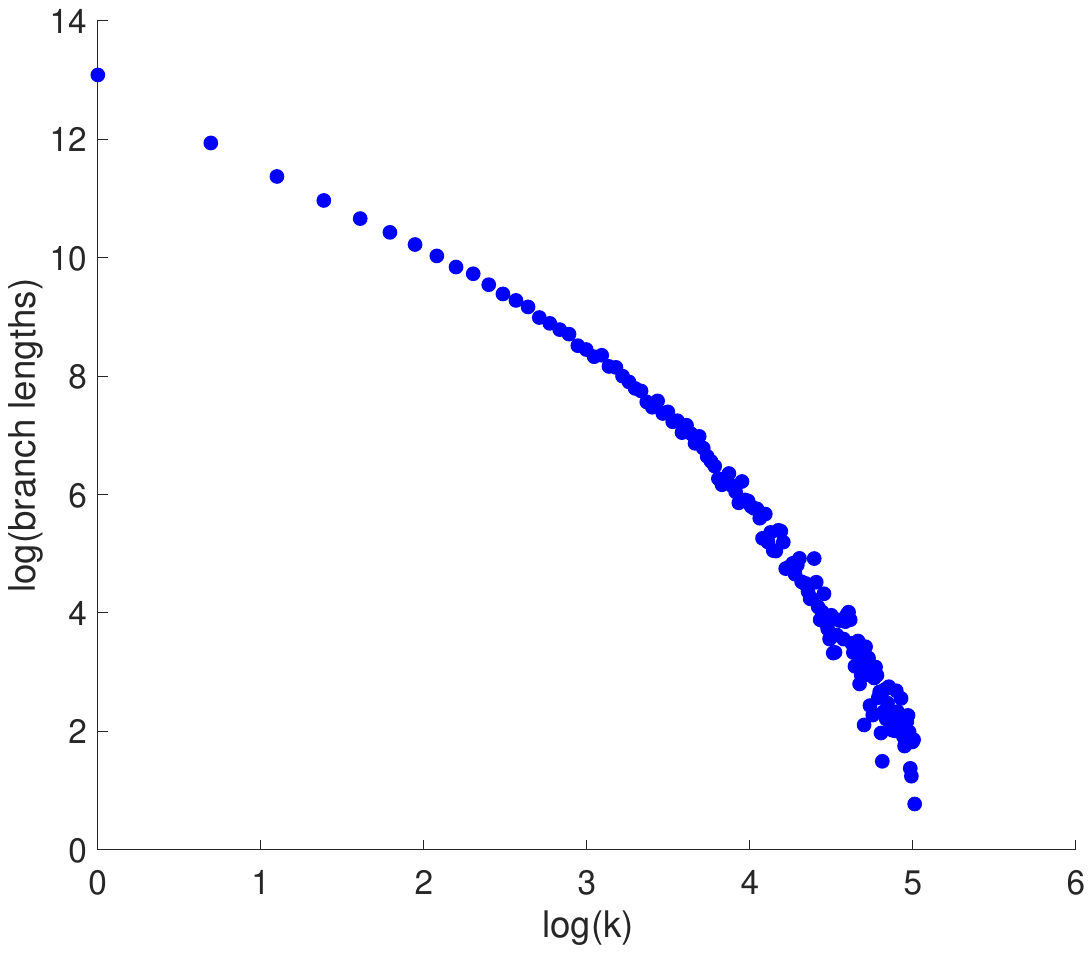}
\caption{The logarithm of the sum of the lengths of the branches on the ancestral lines of $k$ individuals, plotted against the value of $k$ on a log-log scale, for $1 \leq k \leq 150$.  Here $n = 1000$ points were sampled at random from the entire population of size $N = 100,000,000$, and the results were averaged over 100 simulation runs.}
\label{bendfig}
\end{figure}

\begin{figure}[h!]
\centering
\includegraphics[scale=0.38, trim={0cm 7.5cm 1cm 7.8cm}, clip]{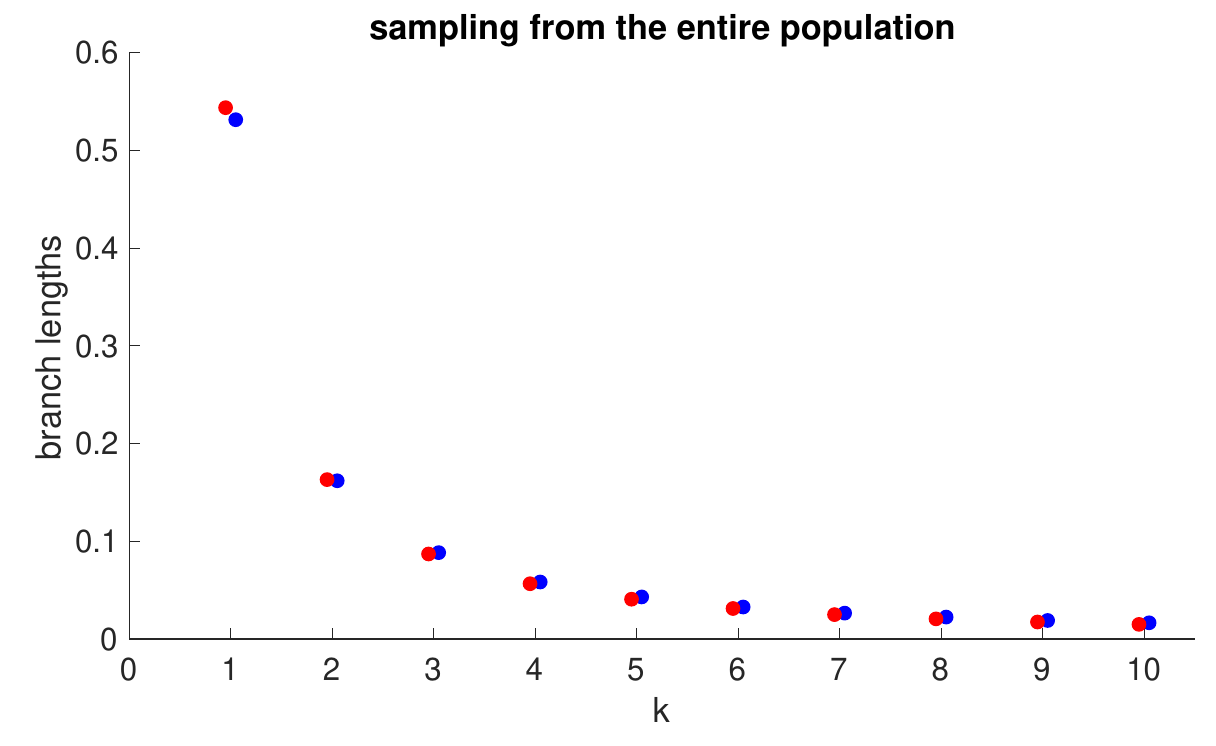}
\includegraphics[scale=0.38, trim={0cm 7.5cm 1cm 7.8cm}, clip]{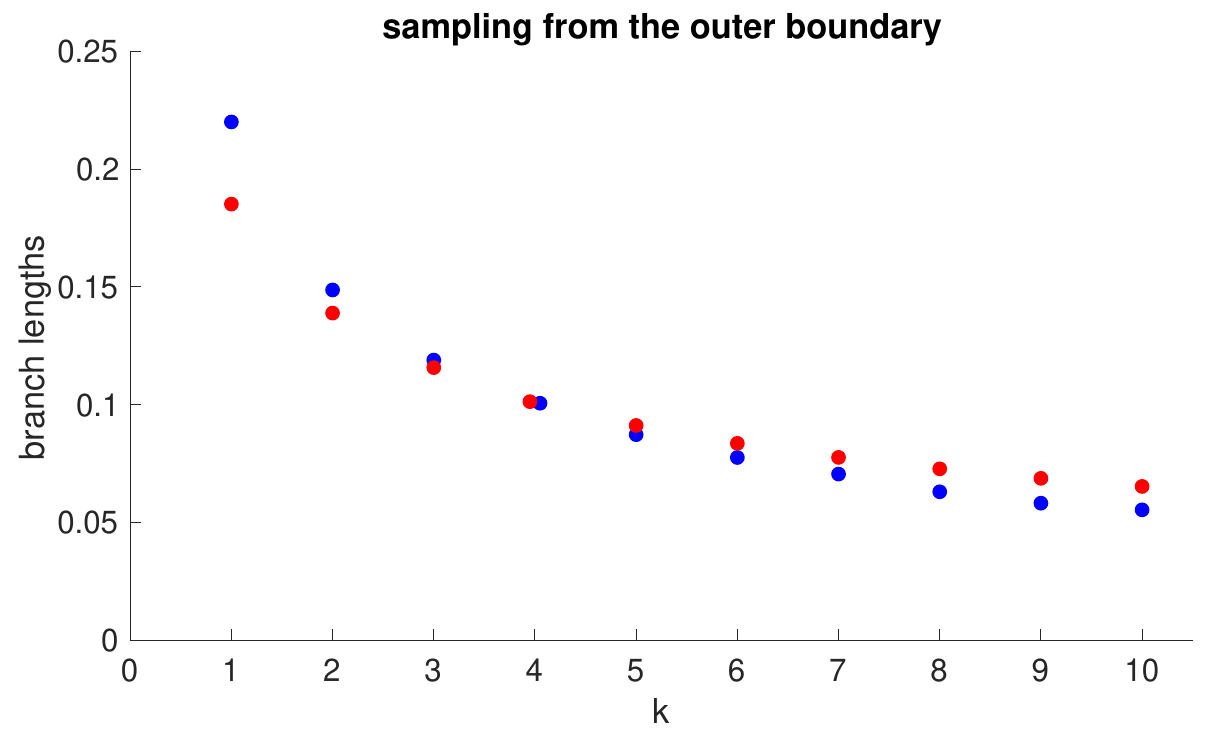}
\caption{For $1 \leq k \leq 10$, the lengths of the branches on the ancestral lines of $k$ individuals, normalized to sum to $1$, are plotted in blue.  Plotted in red are the predictions from Theorems~\ref{Theorem: main2} and \ref{Theorem: main1}, in which the numbers are proportional to $\Gamma(k - 2/5)/k!$ when sampling from the entire population (left panel) and proportional to $\Gamma(k+1/2)/k!$ when sampling from the outer boundary (right panel).  Here $n = 5000$ points were sampled at random from the population of size $N = 100,000,000$, and the results were averaged over 100 simulation runs.}
\label{compare}
\end{figure}

\begin{figure}[h!]
\centering
\includegraphics[scale=0.38, trim={0.5cm 5.5cm 0.5cm 6.5cm}, clip]{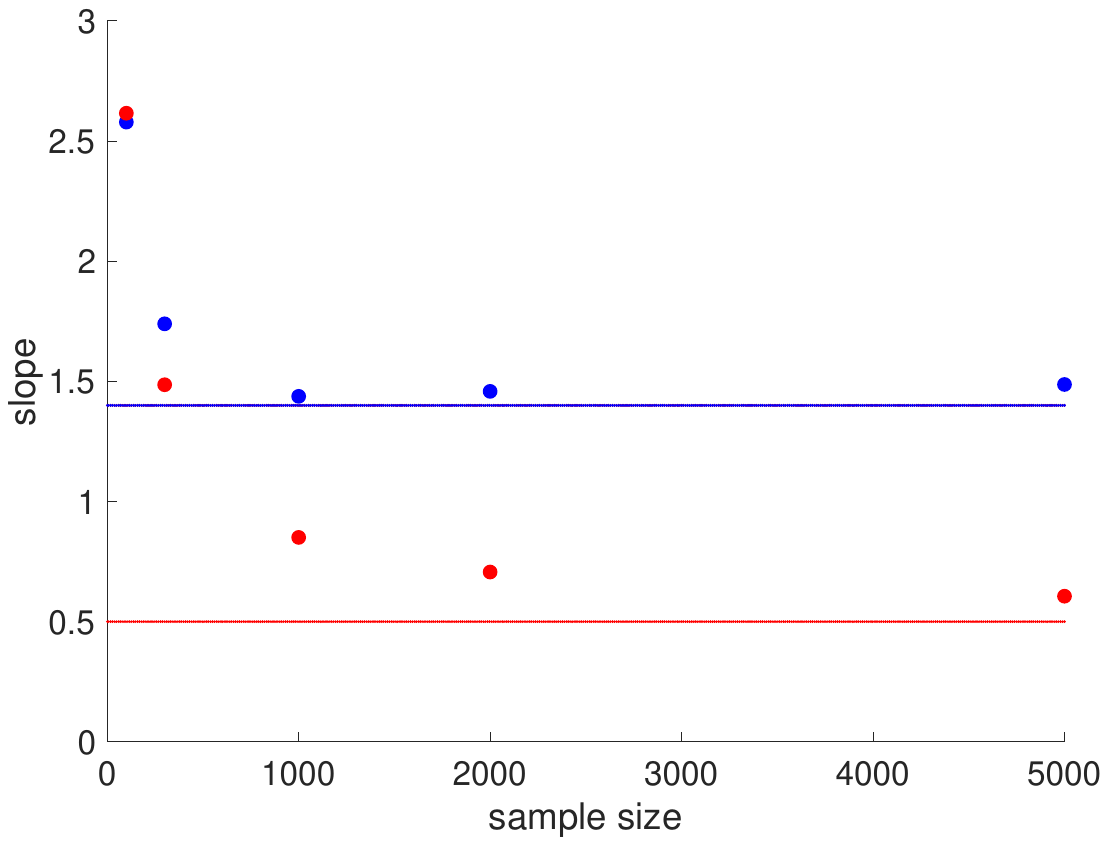}
\caption{This figure shows the absolute value of the slope of the least-squares line when the lengths of the branches on the ancestral lines of $k$ individuals are plotted against $k$, on a log-log scale for $1 \leq k \leq 10$.  Results are shown for sample sizes $n = 100, 300, 1000, 2000, 5000$.  Results when the sample comes from the entire population are in blue, and results for sampling from the outer boundary are in red.  The horizontal red and blue lines show the predictions from Theorems \ref{Theorem: main2} and \ref{Theorem: main1} respectively.}
\label{slopefig}
\end{figure}

We carried out simulations in MATLAB to compare to the conclusions of Theorems \ref{Theorem: main2} and \ref{Theorem: main1}. 
We simulated the growth of a first passage percolation cluster until the number of occupied sites reached $N = 100,000,000$.  We then sampled $n$ occupied sites at random, either from the entire cluster or from the boundary of the cluster, for $n = 100, 300, 1000, 2000$, and $5000$.  We found the geodesics connecting these points to the origin, which correspond to the ancestral lines of $n$ individuals sampled from a spatially growing population.  For all five values of $n$ and all $k = 1, 2, \dots, n-1$, we counted the number of sites that were ancestors of exactly $k$ of the $n$ sampled individuals, corresponding to $Anc_k(U_n)$ in Theorem~\ref{Theorem: main1} when we sample from the entire population and $Anc_k(\widetilde{U}_n)$ in Theorem \ref{Theorem: main2} when we sample from the boundary.  The entire simulation was repeated 100 times, and the results were averaged.  Figures \ref{bendfig}, \ref{compare}, and \ref{slopefig} present results from these simulations.

Figure \ref{bendfig} shows the lengths of the branches belonging to the ancestral lines of $k$ sampled lineages for $k = 1, 2, \dots, 150$, when $n = 1000$ individuals were sampled from the entire population.  Theorem~\ref{Theorem: main1} suggests that, when plotted on a log-log scale as in the figure, the points should lie on a straight line of slope $-7/5$.  This is approximately what is observed in Figure \ref{bendfig} for small values of $k$.  For larger values of $k$, the slope in Figure~\ref{bendfig} becomes more steep, consistent with what was observed in \cite{fgkah16}.

Figure \ref{compare} compares the simulation results for $k = 1, 2, \dots, 10$ to the theoretical results when the sample size is $n = 5000$.  The left panel shows that when we sample from the entire population, the simulations are in excellent agreement with the predictions of Theorem~\ref{Theorem: main2}. The right panel compares the simulation results to the predictions from Theorem~\ref{Theorem: main1} when we sample from the outer boundary of the population.  In this case, there are noticeable discrepancies between the simulation results and the theoretical predictions.  Nevertheless, as predicted by our theoretical results, in simulations when we sample from the outer boundary, we observe longer branch lengths associated with higher values of $k$, and shorter branch lengths associated with $k = 1$, relative to the results when sampling from the entire population.

Figure \ref{slopefig} shows how the results depend on the sample size.
For different sample sizes, we determined the slope of the least-squares line when the lengths of the branches on the ancestral line of $k$ sampled individuals for $k = 1, 2, \dots 10$ are plotted against the value of $k$ on a log-log scale. Theorems \ref{Theorem: main2} and \ref{Theorem: main1} predict that the absolute value of the slope of the line should be 1.4 when we sample from the entire population and 0.5 when we sample from the outer boundary.  Figure~\ref{slopefig} shows that the slope is close to 1.4 when sampling from the entire population as long as the sample size is at least $n = 1000$.  When we sample only from the outer boundary of the population, the convergence is slower and the prediction of Theorem \ref{Theorem: main1} does not become accurate until the sample size is $n = 5000$.  This may explain why in simulation results in \cite{epf22}, when sampling from the outer boundary, the slope was found to be closer to $2/3$ than $1/2$.

\end{document}